\documentclass[11pt,reqno]{amsart}

\usepackage{amsmath,amsfonts,amsthm,amssymb}
\usepackage{graphicx}
\usepackage{verbatim}
\usepackage{color}
\usepackage{amscd}
\usepackage{verbatim}

\begin{document}
\newcommand{\M}{{\mathcal M}}
\newcommand{\loc}{{\mathrm{loc}}}
\newcommand{\core}{C_0^{\infty}(\Omega)}
\newcommand{\sob}{W^{1,p}(\Omega)}
\newcommand{\sobloc}{W^{1,p}_{\mathrm{loc}}(\Omega)}
\newcommand{\merhav}{{\mathcal D}^{1,p}}
\newcommand{\be}{\begin{equation}}
\newcommand{\ee}{\end{equation}}
\newcommand{\mysection}[1]{\section{#1}\setcounter{equation}{0}}
%%%%%%%%%%%%
\newcommand{\laplace}{\Delta}%{\triangle}
\newcommand{\pl}{\laplace_p}
\newcommand{\grad}{\nabla}%{\bigtriangledown}
\newcommand{\pd}{\partial}
\newcommand{\bo}{\pd}
\newcommand{\csub}{\subset \subset}
\newcommand{\sm}{\setminus}
\newcommand{\ssm}{:}
\newcommand{\diver}{\mathrm{div}\,}
%%%%%%%%%%%%%%%
\newcommand{\bea}{\begin{eqnarray}}
\newcommand{\eea}{\end{eqnarray}}
\newcommand{\bean}{\begin{eqnarray*}}
\newcommand{\eean}{\end{eqnarray*}}
\newcommand{\thkl}{\rule[-.5mm]{.3mm}{3mm}}
%%%%%%%%%%%%%%%%%%%%%%%%%%%
\newcommand{\cw}{\stackrel{\rightharpoonup}{\rightharpoonup}}
\newcommand{\id}{\operatorname{id}}
\newcommand{\supp}{\operatorname{supp}}
\newcommand{\wlim}{\mbox{ w-lim }}
\newcommand{\mymu}{{x_N^{-p_*}}}
\newcommand{\R}{{\mathbb R}}
\newcommand{\N}{{\mathbb N}}
\newcommand{\Z}{{\mathbb Z}}
\newcommand{\Q}{{\mathbb Q}}
\newcommand{\abs}[1]{\lvert#1\rvert}
%%%%%%%%%%%
\newtheorem{theorem}{Theorem}[section]
\newtheorem{corollary}[theorem]{Corollary}
\newtheorem{lemma}[theorem]{Lemma}
\newtheorem{notation}[theorem]{Notation}
\newtheorem{definition}[theorem]{Definition}
\newtheorem{remark}[theorem]{Remark}
\newtheorem{proposition}[theorem]{Proposition}
\newtheorem{assertion}[theorem]{Assertion}
\newtheorem{problem}[theorem]{Problem}
%%%%%%%%%%%%%%%%%%
\newtheorem{conjecture}[theorem]{Conjecture}
\newtheorem{question}[theorem]{Question}
\newtheorem{example}[theorem]{Example}
\newtheorem{Thm}[theorem]{Theorem}
\newtheorem{Lem}[theorem]{Lemma}
\newtheorem{Pro}[theorem]{Proposition}
\newtheorem{Def}[theorem]{Definition}
\newtheorem{Exa}[theorem]{Example}
\newtheorem{Exs}[theorem]{Examples}
\newtheorem{Rems}[theorem]{Remarks}
\newtheorem{Rem}[theorem]{Remark}

\newtheorem{Cor}[theorem]{Corollary}
\newtheorem{Conj}[theorem]{Conjecture}
\newtheorem{Prob}[theorem]{Problem}
\newtheorem{Ques}[theorem]{Question}
\newtheorem*{corollary*}{Corollary}
\newtheorem*{remark*}{Remark}
\newtheorem*{theorem*}{Theorem}
\newcommand{\pf}{\noindent \mbox{{\bf Proof}: }}

\renewcommand{\theequation}{\thesection.\arabic{equation}}
\catcode`@=11 \@addtoreset{equation}{section} \catcode`@=12
%%%%%%%%%%%%%%%%%%%%%%
\newcommand{\Real}{\mathbb{R}}
\newcommand{\real}{\mathbb{R}}
\newcommand{\Nat}{\mathbb{N}}
\newcommand{\ZZ}{\mathbb{Z}}
\newcommand{\CC}{\mathbb{C}}
\newcommand{\Pess}{\opname{Pess}}
\newcommand{\Proof}{\mbox{\noindent {\bf Proof} \hspace{2mm}}}
\newcommand{\mbinom}[2]{\left (\!\!{\renewcommand{\arraystretch}{0.5}
\mbox{$\begin{array}[c]{c}  #1\\ #2  \end{array}$}}\!\! \right )}
\newcommand{\brang}[1]{\langle #1 \rangle}
\newcommand{\vstrut}[1]{\rule{0mm}{#1mm}}
\newcommand{\rec}[1]{\frac{1}{#1}}
\newcommand{\set}[1]{\{#1\}}
\newcommand{\dist}[2]{$\mbox{\rm dist}\,(#1,#2)$}
\newcommand{\opname}[1]{\mbox{\rm #1}\,}
\newcommand{\mb}[1]{\;\mbox{ #1 }\;}
\newcommand{\undersym}[2]
 {{\renewcommand{\arraystretch}{0.5}  \mbox{$\begin{array}[t]{c}
 #1\\ #2  \end{array}$}}}
\newlength{\wex}  \newlength{\hex}
\newcommand{\understack}[3]{%
 \settowidth{\wex}{\mbox{$#3$}} \settoheight{\hex}{\mbox{$#1$}}
 \hspace{\wex}  \raisebox{-1.2\hex}{\makebox[-\wex][c]{$#2$}}
 \makebox[\wex][c]{$#1$}   }%
%%Macros for changing font size in math.
\newcommand{\smit}[1]{\mbox{\small \it #1}}% only for letters, numbers
\newcommand{\lgit}[1]{\mbox{\large \it #1}}% only for letters, numbers
\newcommand{\scts}[1]{\scriptstyle #1}
\newcommand{\scss}[1]{\scriptscriptstyle #1}
\newcommand{\txts}[1]{\textstyle #1}
\newcommand{\dsps}[1]{\displaystyle #1}
%%%%%%%%%%%%%%%%%%%%%%%%%%%%%%%%%%%%%%%%%%%%%%
\newcommand{\dx}{\,\mathrm{d}x}
\newcommand{\dy}{\,\mathrm{d}y}
\newcommand{\dz}{\,\mathrm{d}z}
\newcommand{\dt}{\,\mathrm{d}t}
\newcommand{\dr}{\,\mathrm{d}r}
\newcommand{\du}{\,\mathrm{d}u}
\newcommand{\dv}{\,\mathrm{d}v}
\newcommand{\dV}{\,\mathrm{d}V}
\newcommand{\ds}{\,\mathrm{d}s}
\newcommand{\dS}{\,\mathrm{d}S}
\newcommand{\dk}{\,\mathrm{d}k}

\newcommand{\dphi}{\,\mathrm{d}\phi}
\newcommand{\dtau}{\,\mathrm{d}\tau}
\newcommand{\dxi}{\,\mathrm{d}\xi}
\newcommand{\deta}{\,\mathrm{d}\eta}
\newcommand{\dsigma}{\,\mathrm{d}\sigma}
\newcommand{\dtheta}{\,\mathrm{d}\theta}
\newcommand{\dnu}{\,\mathrm{d}\nu}

%%%%%%%%%%%%%%%%%%%%%%%%%%%%%%%Macros for Greek letters.
\def\ga{\alpha}     \def\gb{\beta}       \def\gg{\gamma}
\def\gc{\chi}       \def\gd{\delta}      \def\ge{\epsilon}
\def\gth{\theta}                         \def\vge{\varepsilon}
\def\gf{\phi}       \def\vgf{\varphi}    \def\gh{\eta}
\def\gi{\iota}      \def\gk{\kappa}      \def\gl{\lambda}
\def\gm{\mu}        \def\gn{\nu}         \def\gp{\pi}
\def\vgp{\varpi}    \def\gr{\rho}        \def\vgr{\varrho}
\def\gs{\sigma}     \def\vgs{\varsigma}  \def\gt{\tau}
\def\gu{\upsilon}   \def\gv{\vartheta}   \def\gw{\omega}
\def\gx{\xi}        \def\gy{\psi}        \def\gz{\zeta}
\def\Gg{\Gamma}     \def\Gd{\Delta}      \def\Gf{\Phi}
\def\Gth{\Theta}
\def\Gl{\Lambda}    \def\Gs{\Sigma}      \def\Gp{\Pi}
\def\Gw{\Omega}     \def\Gx{\Xi}         \def\Gy{\Psi}
%%%%%%%%%%%%%%%%%%%%%%%%%%%%%%%%%%%%%%%%%%%%%%%%%%%%%%%%%%%%%%%%%%

\renewcommand{\div}{\mathrm{div}}
\newcommand{\red}[1]{{\color{red} #1}}
%\newcommand{\green}[1]{{\color{green} #1}}

%\begin{titlepage}

\title[Optimal Hardy Weight]{Optimal Hardy Weight for Second-Order Elliptic Operator: an answer to a problem of Agmon}

\author{Baptiste Devyver}
\address{Baptiste Devyver, Department of Mathematics,  Technion - Israel Institute of Technology, Haifa 32000, Israel}
\email{baptiste.devyver@univ-nantes.fr}
\author{Martin Fraas}
\address{Martin Fraas, Theoretische Physik, ETH Zurich, 8093 Zurich, Switzerland}
\email{martin.fraas@gmail.com}
\author{Yehuda Pinchover}
\address{Yehuda Pinchover,
Department of Mathematics, Technion - Israel Institute of
Technology,   Haifa 32000, Israel}

\email{pincho@techunix.technion.ac.il}
 \maketitle
\newcommand{\dnorm}[1]{\thkl #1 \thkl\,}

\begin{abstract}
For a general subcritical second-order elliptic operator $P$ in a domain $\Gw\subset \mathbb{R}^n$ (or noncompact manifold), we construct Hardy-weight $W$ which is {\em optimal} in the following sense. The operator
$P - \lambda W$ is subcritical in $\Gw$ for all $\lambda < 1$, null-critical in $\Gw$ for $\lambda = 1$, and supercritical near any neighborhood of infinity in $\Gw$ for any $\lambda > 1$. Moreover, if $P$ is symmetric and $W>0$, then the
spectrum and the essential spectrum of $W^{-1}P$ are equal to $[1,\infty)$, and the corresponding Agmon metric is complete.

Our method is based on the theory of positive solutions and applies to both symmetric and nonsymmetric operators. The constructed Hardy-weight is given by an explicit simple formula involving two distinct positive solutions of the equation $Pu=0$, the existence of which depends on the subcriticality of $P$ in $\Omega$.
\\[2mm]
\noindent  2000  \! {\em Mathematics  Subject  Classification.}
Primary  \! 35B09; Secondary  35J08, 35J20, 35P05.\\[1mm]
\end{abstract}
%\end{titlepage}

\mysection{Introduction}
Let $P$ be a {\em symmetric} and {\em nonnegative} second-order linear elliptic operator with real coefficients which is defined on a domain  $\Omega \subset \mathbb{R}^n$  or on a noncompact manifold $\Gw$, and let  $q$ be the associated quadratic form defined on $C_0^\infty(\Omega)$. A {\em Hardy-type inequality} with a weight $W \geq 0$ has the form
\begin{equation}
\label{HardyType}
q(\vgf) \geq \lambda \int_\Omega W(x) |\varphi(x)|^2 \dx \quad \mbox{for all } \varphi \in C_0^\infty(\Omega),
\end{equation}
where $\lambda  > 0$ is a constant. Such an inequality aims to quantify the positivity of $P$: for instance, if \eqref{HardyType} holds with $W\equiv \mathbf{1}$ it means that the bottom of the spectrum of the corresponding operator is positive. A nonnegative operator $P$ is called {\it critical} in $\Gw$ if the inequality $P \geq 0$ cannot be improved, meaning that (\ref{HardyType}) holds true if and only if $W \equiv 0$. On the other hand, when (\ref{HardyType}) holds with a nontrivial $W$, then the operator is {\it subcritical} in $\Gw$. Given a subcritical operator $P$  in $\Omega$, there is a huge set of weights $W$ satisfying the inequality \eqref{HardyType}; We will call these weights, \textit{Hardy-weights}. A natural question is to find ``large" Hardy-weights.

The search for Hardy-type inequalities with ``as large as possible" weight function $W$  was proposed by Agmon \cite[Page 6]{Ag82}, and we feel that it deserves the name {\em Agmon's problem}. Agmon raised this problem in connection with his theory of exponential decay for solutions of second-order elliptic equations. Given a Hardy-type inequality \eqref{HardyType}, there is an associated Agmon metric; if this Riemannian metric turns out to be {\em complete}, then Agmon's theory gives the exponential decay at infinity (with respect to the Agmon metric) of solutions of the equation $Pu=f$.

Before proceeding, we recall a classical Hardy-type inequality, in order to motivate the concept of ``large" Hardy-weights:
\begin{example}\label{Ex1} {\em
For $\Omega=\R^n\setminus \{0\}$, $n\geq3$, the following Hardy-type inequality for $P=-\Delta$ holds
\begin{equation}\label{eq:1}
\int_{\R^n\setminus\{0\}} |\grad \vgf|^2\dx \geq \left(\frac{n-2}{2}\right)^2\int_{\R^n\setminus\{0\}} \frac{|\vgf(x)|^2}{|x|^2} \dx \qquad \forall \vgf \in C_0^\infty(\R^n\setminus\{0\}).
\end{equation}
 }
\end{example}
In Example \ref{Ex1}, the Hardy-weight $W$ decays to zero at infinity and blows up at zero, and furthermore its behavior is borderline for the Hardy-type inequalities under consideration. Perhaps the easiest way to illustrate this is the following: for any $\varepsilon\in \R$, define a smooth positive weight $W_\varepsilon$ which is equal to
$$W_\varepsilon:=|x|^{-2+\varepsilon}$$
outside the unit ball. If $\varepsilon<0$, then $W_\varepsilon$ is a short-range potential, while if $\varepsilon\geq0 $, then $W_\varepsilon$ is long-range. More precisely, if $\varepsilon<0$, then for any constant $C>0$ there exists $R>0$ such that
\begin{equation}\label{eq_sr}
\int_{\R^n} |\nabla \vgf|^2\dx\geq C \int_{\R^n} W_\varepsilon (x)|\vgf|^2\dx \quad \forall \vgf \in C_0^\infty(\{|x|> R\}),
\end{equation}
and the operator $W_\varepsilon^{-1}P$ has a discrete positive spectrum. In particular, the corresponding Rayleigh-Ritz variational problem
admits a minimizer. On the other hand, for any $\varepsilon>0$, there are no constants $C>0$ and $R>0$ such that \eqref{eq_sr} holds true, and the bottom of the (essential) spectrum of the operator $W_\varepsilon^{-1}P$ equals $0$. Therefore, $W_0$, which agrees with $|x|^{-2}$ outside the unit ball, is the only long-range potential in the family $\{W_\varepsilon\}_{\varepsilon\in \R}$ such that the Hardy-type inequality \eqref{HardyType} holds. Moreover, $\gl=C_H:=(n-2)^2/4$ is the best constant
for \eqref{HardyType} not only in the punctured space, but in a fixed neighborhood of either zero or infinity. On the other hand, the corresponding Rayleigh-Ritz variational problem
does not admit a minimizer.

This indicates -- in a very rough way -- that the weight $C_H|x|^{-2}$ is a ``large" Hardy-weight for $P=-\Delta$ on $\R^n\setminus\{0\}$.

Agmon's theory gives the following (almost optimal) a priori decay estimates for nongrowing solutions $u$ of the Poisson equation in $\R^n$: for every $\varepsilon>0$, there is a constant $C=C(f,\varepsilon)$ such that for every $x$ outside the unit ball,
$$|u(x)|\leq  C|x|^{2-n-\varepsilon}.$$
 We thus might expect that the construction of good Hardy-weights will lead to valuable spectral information about $P$.

In this article we use a general albeit simple construction of Hardy-weights which allows one to recover practically \textit{all} classical Hardy inequalities in a unified way. We use this construction to study Agmon's problem. In particular, in some important cases we find an optimal Hardy weight; This includes the case of a general {\em nonselfadjoint} operator $P$ defined on a punctured domain.

\mysection{Our results}
In this section, we describe in more detail the main results of the paper.
\subsection{The supersolution construction}

The first result that we obtain is the aforementioned general construction of Hardy-weights $W$ satisfying the inequality \eqref{HardyType}. We first recall that the relationship between spectral (or equivalently, functional) properties of a symmetric operator $P$ to properties of positive solutions of the equation $Pu=0$ is well understood using the Agmon-Allegretto-Piepenbrink (AAP) theory \cite{Agmon,PReview}. In particular, the existence of a positive supersolution $v$ of the equation $(P-\lambda W)u=0$ in $\Gw$ is equivalent to the Hardy-type inequality \eqref{HardyType}, and hence (assuming $W>0$ in $\Gw$) it is equivalent to the inclusion of the spectrum of the associated symmetric operator $W^{-1}P$ in $[\lambda, \infty)$. Moreover, the existence of a positive supersolution $v$ of the equation $(P-\lambda W)u=0$ in a neighborhood of infinity in $\Gw$ is equivalent to the inclusion  of the corresponding essential spectrum in $[\lambda, \infty)$ \cite{Baptiste}.

Our construction relies on two observations, which are both well known. First, using (AAP) theory, we will see that there is a correspondence between positive supersolutions of $P$ and nonnegative Hardy-weights. Explicitly, to every positive supersolution $v$ of $P$, we associate the weight $W:=Pv/v$, which satisfies \eqref{HardyType} with $\lambda=1$. The second step (that we call the \textit{supersolution construction}) is a way of producing positive supersolutions of $P$ -- hence Hardy-weights. The construction is the following: let  $v_0$ and $v_1$ be  two linearly independent positive (super)solutions of the equation $Pu=0$ in $\Gw$. Then for $0\leq \ga \leq 1$, the function
$$
v_\ga := v_0^{1-\ga} v_1^\ga
$$
is a positive supersolution of the equation $Pu=0$ in $\Gw$, thus yielding a Hardy-weight $W_\ga:=Pv_\ga/v_\ga$.
We will find that all these weights are proportional,
$$
W_\alpha = 4\alpha (1-\alpha) W(v_0,v_1), \quad W(v_0,v_1) = \frac{1}{4}\left|\nabla \log\left(\frac{v_0}{v_1}\right)\right|_A^2,
$$
and the prefactor $4\alpha (1-\alpha)$ achieves its maximum $1$ at $\alpha = 1/2$.
In particular, if the equation $Pu=0$ admits two linearly independent positive (super)solutions in $\Gw$, then $P$ is subcritical in $\Gw$. Moreover, with the freedom of choosing $v_0$ and $v_1$, this construction allows us in fact, to recover in a unified way all the classical Hardy inequalities. It is also a very easy method for producing new examples.
\subsection{Agmon's problem and optimal weights}

The aim of the paper is to show that with a careful choice of $v_0$ and $v_1$, the preceding construction gives rise to Hardy-weights $W(v_0,v_1)$ which deserve the title of \textit{optimal weights}. We first give a \textit{temporary} definition of optimal weights.

\begin{definition}\label{temp_def}
{\em
Consider a symmetric subcritical operator in $\Gw$, and let $W$ be a nonzero nonnegative weight satisfying the Hardy inequality
 \begin{equation}
\label{HardyType_def}
q(\vgf) \geq \lambda \int_\Omega W(x) |\varphi(x)|^2 \dx \quad \mbox{for all } \varphi \in C_0^\infty(\Omega),
\end{equation}
with $\gl>0$. We  denote by $\lambda_0=\lambda_0(P,W,\Gw)$ the best constant satisfying \eqref{HardyType_def}; $\gl_0$ is called the {\em generalized principal eigenvalue}.
The weight $\gl_0 W$ is said to be an \textit{optimal Hardy-weight} for the operator $P$ in $\Gw$ if the following properties hold:
\begin{enumerate}
\item[(a)] The operator   $P-\lambda_0 W$ is \textit{critical} in $\Gw$; that is, the inequality
$$q(\vgf) \geq  \int_{\Gw} V(x)\vgf^2(x) \dx \qquad \forall \vgf \in C_0^\infty(\Omega)$$
is not valid for any $V \gneqq \lambda_0W$.

\item[(b)] The constant $\lambda_0$ is also the best constant for \eqref{HardyType_def} with test functions supported in the exterior of any fixed compact set in $\Omega$.
\item[(c)] The operator $P - \lambda_0 W$ is {\it null-critical} in $\Gw$; that is, the corresponding Rayleigh-Ritz variational problem
\begin{equation}\label{RRVP}
\inf_{\vgf\in \mathcal{D}_P^{1,2}(\Gw)}\left\{\frac{q(\varphi)}{\int_{\Gw} W(x) |\vgf(x)|^2 \dx}\right\}
\end{equation}
 admits no minimizer. Here $\mathcal{D}_P^{1,2}(\Omega)$ is the completion of $C_0^\infty(\Omega)$ with respect to the norm $u\mapsto \sqrt{q(u)}$.
\end{enumerate}
}
\end{definition}
Properties (b) and (c) indicates in a way that $W$ is  ``long range". Note that contrary to the ``short range" case, the validity of (a) in the case of a ``long range" potential is quite delicate. Indeed, it is known \cite{Ag2,ky2,P3} that $P-\lambda_0 W$ is \text{always} critical when property (b) does not hold (see also \cite{FT}). On the other hand, in the ``long range" case $P-\gl_0 W$ is in general subcritical.

In order to motivate the definition, let us mention that the weight $C_H|x|^{-2}$ of Example \ref{Ex1} is an optimal Hardy-weight (see Section~\ref{short_pf} for a short proof for this special case).
\subsection{The result for operators in punctured domains}

Motivated by Example~\ref{Ex1}, we study in detail the case of a general (nonsymmetric) subcritical operator $P$ in the \textit{punctured domain} $\Gw^\star:=\Gw\setminus \{0\}$: Theorem~\ref{thm_nsa} states that if one chooses two positive solutions $v_0$, and $v_1$ appropriately in $\Gw^\star$, then for $\ga=1/2$, the corresponding weight $W(v_0,v_1)$ constructed by the supersolution construction is an {\em optimal Hardy-weight in $\Gw^\star$}. The following theorem states the result for {\em symmetric} operators.

\begin{theorem}\label{thm*}
Let $P$ be a symmetric subcritical operator in $\Gw$, and let $G(x):=G_P^\Gw(x,0)$ be its minimal positive Green function with a pole at $0\in \Gw$. Let $u$ be a positive solution of the equation $Pu=0$ in $\Gw$ satisfying
\begin{equation}\label{u1u0}
\lim_{x\to\infty} \frac{G(x)}{u(x)}=0,
\end{equation}
where $\infty$ is the ideal point in the one-point compactification of $\Gw$. Consider the supersolution
$ v := \sqrt{G u}$. Then $$W:=\frac{Pv}{v}=\frac{1}{4}\left|\nabla \log\left(\frac{G}{u}\right)\right|_A^2$$ is  an  optimal Hardy-weight with respect to $P$ and the punctured domain $\Gw^\star=\Gw\setminus \{0\}$. If furthermore $W>0$, then the spectrum and the essential spectrum of the Friedrichs extension of the operator $-W^{-1}\Gd$ on $L^2(\Gw, W\dx)$ are equal to  $[\lambda_0,\infty)$ and the corresponding Agmon metric is complete.

\end{theorem}
One can look at a punctured domain as a domain with two singular points or more precisely a domain with two ends. In Theorem \ref{thm_bs}, we will treat the case where the two singular points are at the boundary of $\Omega$, and in Theorem \ref{lem_sever} we will treat the case of several ends. We will illustrate our results in a variety of examples: see in particular Examples \ref{ex1} and \ref{ex1a} for explicit examples of optimal Hardy inequalities.

\vskip 3mm

As mentioned, our results essentially also hold in the general case of a (not necessarily symmetric) second-order, linear, subcritical differential operator. Criticality theory, which is the qualitative theory of positive solutions of the equation $Pu=0$ in $\Gw$  for a general nonsymmetric second-order elliptic operator $P$ with real coefficients (see Section~\ref{1}), extends the  functional/spectral formulation of nonnegativity. We show that when Agmon's problem is interpreted in the terminology of criticality theory, our results apply to a general nonsymmetric second-order elliptic operator $P$ in a domain $\Gw$. In particular, all statements of Theorem~\ref{thm*} (interpreted in the terminology of criticality theory), excluding the statement about the whole spectrum hold true for such a general $P$. Note that the relation to the {\em integral} Hardy-type inequality (\ref{HardyType}) is lost in the nonsymmetric case.
\subsection{Comparison with previous results}

The classical positive supersolution approach to spectral problems and variational inequalities was studied by many authors
\cite[and the references therein]{Ag82,Barta, Kasue,MMP,NP,crit2}. The idea of using the Green function to get, albeit via integral identities, Hardy-type inequalities appears in a few recent papers for a {\em symmetric} divergence form operator (without a potential), see for example \cite[and references therein]{AS,Carron,CZ,C,Dou}. For the Laplace-Beltrami operator on Riemannian manifolds such an approach was used by Carron in \cite{Carron} to provide Hardy-type inequalities in particular for minimal hypersurfaces of a Euclidean space, and for submanifolds of Cartan-Hadamard manifolds. Carron's results have been later rediscovered by Li and Wang \cite{LW}, where the authors give also applications to
structure theorems for complete manifolds.
Compared to these results, we provide a novel nonvariational method that applies to a general operator $P$ satisfying minimal regularity assumptions.
Furthermore, the (null)-criticality of $P - W$, the optimality near infinity, and the characterization of the (essential) spectrum of the weighted operator, seem to be new even in the case of the Laplace (let alone Schr\"odinger) operators on domains in $\R^n$, and the Laplace-Beltrami operator on Riemannian manifolds. For  recent results concerning sharp Hardy inequalities see for example \cite[and references therein]{BEL,BFT,GM,LLL}.

We note that some results of the present paper have been recently announced by the authors in \cite{CR}.
\subsection{The organization of the article}

The outline of the present paper is as follows. In Section~\ref{short_pf} we provide a short proof of Theorem~\ref{thm*} for the classical Hardy inequality \eqref{eq:1}. This will illuminate the main ideas and steps of the proof in the general case. In Section~\ref{1} we review the theory of positive solutions and formulate precisely our main result in the nonsymmetric case (Theorem~\ref{thm_nsa}). Section~\ref{sec:construction} explains in detail the supersolution construction of Hardy-weights. Sections \ref{sec2}, \ref{sec22}, \ref{sec_nc} and \ref{sec_essential} are then devoted to the four-steps proof of Theorem~\ref{thm_nsa} (see, theorems~\ref{thm_Sch}, \ref{thm_Sch7}, \ref{thm_null-critical}, and \ref{essential_spectrum}). In Section~\ref{sec_agmon_m} we prove the completeness of Agmon's metric induced by our optimal Hardy-weight in $\Omega^\star$, derive Rellich-type inequalities, and obtain decay estimates for solutions of the equation $Pu=f$.

Our main result deals with a weight $W$ which has an isolated singularity in $\Gw$, in Section~\ref{sec_bound_sing} we describe how our methods and results can be extended to the case of positive solutions with boundary singularities. In Section~\ref{sec_sever}, we study the case of a symmetric subcritical operator which is defined on a manifold $M$ with $N$ ends, where $N\geq 2$. It turns out that in this case, the supersolution construction produces an $(N-1)$-parameter family of {\em critical} Hardy-weights  (see Theorem~\ref{lem_sever}). Finally, in Section~\ref{examples} we discuss some examples, extensions, and applications. In particular, we discuss some generalizations to quasilinear equations.

The proofs of the main results of the present paper hinges on a one-variable approach. Indeed, it is based on a thorough analysis of a space of ``radial" generalized eigenfunctions.  This is particularly evident in Section~\ref{sec_essential}, where the appropriate ``radial" space is defined. To further elaborate this point, we consider in the appendix, the class of Schr\"odinger operators with radially symmetric potentials defined on radially symmetric domains and study the corresponding radial solutions, and present a purely ODE proof of some of our results for this important case.

\vskip 3mm

\paragraph{\bf Notation:} Throughout the paper and without loss of generality, we assume that $0\in \Omega$ and denote $\Gw^\star:=\Gw\setminus \{0\}$. In addition, we fix a reference point  $x_1 \in \Omega$, $x_1\neq 0$. When there is no danger of confusion we will omit indices. In particular, for a matrix $A(x)=\big[a^{ij}(x)\big]$ and a vector field $b(x)$ we denote
$$
(A(x) \xi)^i = \sum_{j=1}^n a^{ij}(x) \xi_j,\;\; b(x) \cdot \xi  = \sum_{j=1}^n b^j(x) \xi_j , \; \mbox{where } \xi=(\xi_1,\ldots,\xi_n)\in \mathrm{R}^n .
$$
Moreover, for $x\in \Gw$ we introduce a norm on $\mathbb{R}^n$ associated to a positive definite symmetric  matrix $A(x)$,
$$ |\xi|_A^2 := \xi \cdot A \xi \,.$$
We write $\Omega_1 \Subset \Omega_2$ if $\Omega_2$ is open, $\overline{\Omega_1}$ is
compact and $\overline{\Omega_1} \subset \Omega_2$.

%%%%%%%%%%%%%%%%%%%%%%%%%%%%%%%%%%%%%%%%%%%%%%%%%
\mysection{A short proof of Theorem~\ref{thm*} for the classical Hardy inequality}\label{short_pf}
%%%%%%%%%%%%%%%%%%%%%%%%%%%%%%%%%%%%%%%%%%%%%
Before embarking to the general setting and proofs, we give a short proof of Theorem~\ref{thm*} for the case of the classical Hardy inequality \eqref{eq:1}. This will illuminate the main ideas and steps of the proof in the general case.

\begin{example}[Example~\ref{Ex1} continued] \label{H1}{\em
Let $P=-\Gd$ be the Laplace operator on $\Gw^\star:= \mathbb{R}^n\setminus \{0\}$, where $n \geq 3$, and denote by
$G(x):=|x|^{2-n}$ the corresponding positive minimal Green function with a pole at zero (up to a multiplicative constant).

Consider the positive superharmonic function in $\Gw^\star$
$$v(x) := \sqrt{G(x)\mathbf{1}}=G(x)^{1/2}= |x|^{(2-n)/2}.$$ We obtain the Hardy-weight $Pv/v=C_H|x|^{-2}$, and by the (AAP) theory we get the classical Hardy inequality (\ref{eq:1}).

 To prove that we indeed obtain an {\em optimal} Hardy-weight, we analyze the oscillatory properties of the corresponding radial equation
\begin{equation}\label{eq_r1}
-u''-\frac{n-1}{r}u'-\gh \frac{C_H }{r^2}u=0\qquad r\in (0,\infty),
\end{equation}
where $\gh\in \R$. Note that \eqref{eq_r1} is Euler's equation.
Consequently, for $\gh\neq 1$ two linearly independent solutions of \eqref{eq_r1} are given by
\begin{equation}\label{upm}
u_\pm(r)=r^{(2-n)/2}\big(r^{(2-n)/2}\big)^{\pm \sqrt{1-\gh}},
\end{equation}
while  for $\gh= 1$ two linearly independent solutions of \eqref{eq_r1} are expressed by
\begin{equation}\label{upm1}
 u_+(r) =r^{(2-n)/2}, \;\;\;  u_-(r) =r^{(2-n)/2}\log (r^{2-n}).
\end{equation}
The difference in the structure of the solutions for $\gh<1$, $\gh = 1$ and $\gh > 1$ cannot be over-stressed.

For $\gh < 1$ both solutions are positive, and therefore, the operator  $P - \gh C_H|x|^{-2}$ is
subcritical in $\Gw^\star$.

On the other hand, for $\gh = 1$ only $u_+(r) =r^{(2-n)/2}$ is positive, and moreover,
it is dominated by $|u_-|$ near both ends $r=0$ and $r=\infty$. By Proposition~\ref{minimal} we infer that $u_+$ is a \textit{ground state} and the operator $P -  W$ is critical in $\Gw^\star$, where $W:=-\Gd(u_+)/u_+=C_H |x|^{-2}$ is the corresponding Hardy-weight.

Furthermore, an elementary calculation shows that for $\gh= 1$ we have that the ground state $u_+$ is not in  $L^2(\Gw^\star, W\dx)$, which shows the null-criticality of the Hardy operator
$-\Delta-C_H|x|^{-2}$ in $\Gw^\star$.

Finally,  for $\gh>1$ the solution of \eqref{eq_r1} given by
\begin{equation}\label{osc}
\mathfrak{Re} \{u_+(r)\}=r^{(2-n)/2}\cos \left[\frac{\sqrt{\gh-1}}{2}\log (r^{2-n})\right]
 \end{equation}
oscillates near zero and near infinity, and therefore, the best possible constant for the validity of the Hardy inequality in any neighborhood of either the origin or infinity is also $C_H$. In particular, the bottom of the spectrum and the bottom of the essential spectrum of the corresponding weighted Laplacian (with weight $W=C_H^{-1}|x|^2$) is equal $1$.

The entire (essential) spectrum of the operator $\tilde{P}:=C_H^{-1}|x|^2(-\Delta)$ is obtained by an explicit  spectral representation of the operator $\tilde{P}$ restricted to the radial functions, using the \textit{Mellin transform}. Denote by $L_{\mathrm{rad}}^2(\Gw^\star,W\dx)$ the subspace of radially symmetric functions in $L^2(\Gw^\star,W\dx)$. Recall that the  Mellin transform $\mathcal{M} : L^2(0,\infty)\longrightarrow L^2(\R)$ is the unitary operator defined by
$$\mathcal{M}f(\xi):=\frac{1}{\sqrt{2\pi}}\int_0^\infty f(r)r^{\mathrm{i}\xi-\frac{1}{2}}\dr.$$
In fact, the composition of the unitary operator
$$L^2\Big((0,\infty),r^{n-1}\frac{C_H}{r^2}\dr\Big)\rightarrow L^2(0,\infty);  \quad f(r) \mapsto \frac{\sqrt{|n-2|}}{2}f(r^{1/(n-2)}),$$
and the Mellin transform, gives a unitary operator
$$\mathfrak{U} : L_{\mathrm{rad}}^2(\Gw^\star,W\dx)\cong L^2\Big((0,\infty),r^{n-1}\frac{C_H}{r^2}\dr\Big) \rightarrow L^2(\R),$$ which is a spectral representation for $\tilde{P}$ restricted to radial functions. In this representation, $\tilde{P}$ is just the multiplication by $(1+4\xi^2)$. Indeed, this follows from the fact that due to \eqref {eq_r1} and \eqref{upm} (with $\xi=\sqrt{\gh-1}/2$), we have
\begin{equation}\label{eq_spect_rep}
\left(C_H^{-1}|x|^2(-\Delta)-(4\xi^2+1)\right)\left(r^{n-2}\right)^{\mathrm{i}\xi-\frac{1}{2}}=0.
\end{equation}
 }
\end{example}

The proof of the main  result (Theorem~\ref{thm_nsa}) in the {\em general} case is based on similar considerations and calculations. Loosely speaking,
to obtain the general result, we just replace $r^{(2-n)}$ in equations (\ref{upm}), (\ref{upm1}), (\ref{osc}), and \eqref{eq_spect_rep} by $G/u$  (cf. lemmas~\ref{lem_logu} and \ref{lem_construction1}).
%%%%%%%%%%%%%%%%%%%
\mysection{Preliminaries}\label{1}
%%%%%%%%%%%%%%%%%%%%%%%%%%%%%%%%%%%%%%%%%%%
In the present section we review the theory of positive solutions and formulate our main result for nonsymmetric operators defined on punctured domains.

Let $\Omega\subset \mathbb{R}^n$, $n\geq 2$ be a domain (or more generally, a smooth noncompact manifold $\Omega$ of dimension $n$). We assume that $\nu$ is a positive measure on $\Omega$, satisfying $\dnu=f\,\mathrm{vol}$ with $f$ a positive function; $\mathrm{vol}$ being the volume form of $\Omega$ (which is just the Lebesgue measure in the case of a domain of $\R^n$). Consider a second-order elliptic operator $P$ with real coefficients which (in any coordinate system
$(U;x_{1},\ldots,x_{n})$) is either of the form
\be \label{P}
Pu=-a^{ij}(x)\partial_{i}\partial_{j}u + b(x)\cdot\nabla u+c(x)u,
\end{equation}
or in the divergence form
\begin{equation} \label{div_P}
Pu=-\div \left[\big(A(x)\nabla u +  u\tilde{b}(x) \big) \right]  +
 b(x)\cdot\nabla u   +c(x)u,
\end{equation}
Here, the minus divergence is the formal adjoint of the gradient with respect to the measure $\nu$.  We assume that for every $x\in\Gw$ the matrix $A(x):=\big[a^{ij}(x)\big]$ is symmetric and that the real quadratic form
\be\label{ellip}
 \xi \cdot A(x) \xi := \sum_{i,j =1}^n \xi_i a^{ij}(x) \xi_j \qquad
 \xi \in \Real ^n
\end{equation}
is positive definite. Moreover, throughout the paper it is assumed that $P$ is locally uniformly elliptic, and the coefficients of $P$ are locally sufficiently regular in $\Omega$. All our results
hold for example when $P$ is of the form \eqref{div_P}, and $A,\, f$ are locally H\"{o}lder continuous, $b,\,\tilde{b} \in L^p_{\mathrm{loc}}(\Omega; \mathbb{R}^n,\dx)$, and $c \in L^{p/2}_{\mathrm{loc}}(\Omega; \mathbb{R},\dx)$ for some $p > n$. However it would be apparent from the proofs that any conditions that guarantee standard elliptic theory are sufficient.

The formal adjoint $P^*$ of the operator $P$ is defined on its natural space $L^2(\Omega, \dnu)$. When $P$ is in divergence form (\ref{div_P}) and $b = \tilde{b}$, the operator
\be
\nonumber
Pu = - \div \left[ \big(A \grad u + u b\big) \right] + b \cdot \grad u + c u,
\ee
is {\em symmetric} in the space $L^2(\Omega, \dnu)$. Throughout the paper, we call this setting the {\em symmetric case}.  We note that if $P$ is symmetric and $b$ is smooth enough, then $P$ is in fact a Schr\"odinger-type operator of the form
\be
\nonumber
Pu = - \div \big(A \grad u \big) + \big(c-\div b\big)u.
\ee

In the paragraphs below we recall basic notions and theorems from the
theory of positive solutions. We refer the reader to a review \cite{PReview} for
details and further references.

\begin{definition}{\em Denote by $\mathcal{C}_{P}(\Omega)$ the cone of all positive solutions of the
elliptic equation $Pu=0$ in $\Omega$. The operator $P$ is said to be {\em nonnegative in} $\Gw$, and write $P \geq  0$ in $\Gw$, if $\mathcal{C}_{P}(\Omega)\neq \emptyset$. We say that $P$ satisfies the {\em positive Liouville theorem} in $\Gw$ if $\dim \mathcal{C}_{P}(\Omega)=1$.

For a nonzero (real valued) function $W$, let
$$\lambda_0=\lambda_0(P,W,\Omega)
:= \sup\{\lambda \in \mathbb{R} \mid P-\lambda W\geq 0\; \mbox{ in } \Gw\}$$ be the {\em generalized principal
eigenvalue} of the operator $P$ with respect to the potential $W$ in $\Omega$.
We  also denote
$$\gl_{\infty}:= \gl_{\infty}(P,W,\Gw):=\sup\{\gl\in\Real\mid \exists K\subset
\subset\Gw\mbox{ s.t. } P-\gl W\geq 0 \; \mbox{ in }  \Gw\setminus K \}.$$
 }\end{definition}
Clearly, $ \gl_0\leq\gl_{\infty} $. Moreover, $P-\gl_0 W\geq 0$ in $\Gw$. If $P$ is a symmetric operator, then in light of the Agmon-Allegretto-Piepenbrink (AAP) theory (see for example \cite{Agmon} and \cite{Baptiste}), $\gl_0$ and $\gl_\infty$ have the following spectral interpretation:
\begin{proposition}
\label{allegretto}
Assume that the operator $P$ is a symmetric in $L^2(\Omega,\dnu)$,  and $W>0$. Suppose also that $\lambda_0(P,W,\Gw)  >-\infty$. Define $$\tilde{P}:=W^{-1}P.$$
Then $\tilde{P}$ is symmetric on $L^2(\Gw, \, W\mathrm{d}\nu)$, has the same quadratic form as $P$, and $\lambda_0$ (resp. $\lambda_\infty$) is the infimum of the spectrum (resp. essential spectrum) of the Friedrichs extension of $\tilde{P}$.

Denote by $q$ the quadratic form associated to $P$, and assume that $P\geq 0$ in $\Gw$. Then the following Hardy-type inequality holds true with the best constant $\lambda_0=\lambda_0(P,W,\Gw)\geq 0$:
\begin{equation}
\label{HardyIn}
 q(\vgf) \geq \lambda_0 \int_\Omega W \vgf^2\dnu \qquad\forall \vgf\in C_0^\infty(\Gw).
\end{equation}

\end{proposition}
Next, we introduce the definition of (sub)criticality:
  \begin{definition}\label{def_sc}
{\em
Assume that $P \geq 0$ in $\Gw$. The operator $P$ is said to be {\em subcritical in $\Omega$} if
there exists a nonzero nonnegative continuous function $W$ such that $\lambda_0(P,W,\Gw)>0$, otherwise, $P$ is {\em critical in $\Gw$}. So, in the critical case, $\lambda_0(P,W,\Gw) = 0$ for any nonnegative nonzero continuous function $W$.

If $P\not \geq 0$ in $\Gw$, then $P$ is said to be {\em supercritical} in $\Gw$.
}
\end{definition}
The (sub)criticality of $P$ in $\Gw$ has an equivalent characterization in terms
of the structure of the cone of positive solutions $\mathcal{C}_P(\Omega)$.
This characterization is based on the notion of positive solution of minimal growth (see \cite{Agmon}), and it is a key to our theorems and proofs. We recall the definition.
  \begin{definition}\label{def:minimal_gr} {\em
1. Let $K\Subset \Omega$, and let $u$ be a positive solution of the equation $Pw=0$ in  $\Omega\sm K$. We say that $u$ is {\em a positive solution of  minimal growth in a neighborhood of infinity in $\Omega$}  if for any $K\Subset K'\Subset \Omega$ with smooth boundary  and any (regular) positive supersolution $v\in C((\Omega\sm K')\cup \,\pd K')$ of the equation $Pw=0$ in $\Omega \sm K'$ satisfying  $u \leq v$ on $\pd K'$, we have $u \leq v$ in $\Omega \sm K'$.

\vskip 3mm

2. Let $x_1\in \Gw$. A positive solution of the equation
$$
P u = 0 \qquad \mbox{in} \quad \Gw\setminus\{x_1\}
$$
of minimal growth in a neighborhood of infinity in $\Omega$ is called a {\em positive minimal Green function}, if the singularity at $x_1$ is not removable. The appropriately normalized Green function is denoted by $G_P^\Gw(x,x_1)$.}
  \end{definition}
The aforementioned characterization of a subcritical operator is given in the following proposition.
\begin{proposition}\label{structure}
Suppose that $P \geq 0$ in $\Gw$. The
operator $P$ is  subcritical in $\Omega$ if and only if it admits a positive minimal Green function $G_P^\Gw(x,x_1)$ in $\Gw$. Moreover, in the critical case,  the equation  $Pu=0$ admits a unique (up to multiplicative constant) positive global solution in $\Gw$, which is called {\em Agmon's ground state} (or in short a ground state).

The operator $P$  is  subcritical (resp. critical) in $\Omega$ if and only if its formal adjoint $P^\star$ is  subcritical (resp. critical) in $\Omega$.
\end{proposition}
We note that a ground state of a critical operator $P$ in $\Gw$ is a positive global solution of the equation $Pu=0$ in $\Gw$ that has minimal growth in a neighborhood of infinity in $\Omega$.

Let $P$ be subcritical in $\Gw$ and $W\gneqq 0$. Clearly $\gl_0:=\gl_0(P,W,\Gw)\geq 0$, but $\gl_0$ might be either $0$ or positive. Moreover, the operator $P-\gl W$ is subcritical in $\Gw$ for $0\leq\gl<\gl_0$, but  $P-\gl_0 W$ might be either subcritical or critical in $\Gw$.  The case of a perturbation by a compactly supported potential (or more generally, by a semismall perturbation \cite{M97})  is well understood (see for example \cite[and references therein]{PReview}). In particular, we have:
\begin{proposition}\label{CompPot}
Let $P$ be a subcritical operator in $\Gw$ and $W \geq 0$ a nonzero bounded compactly
supported weight in $\Gw$ (or more generally, $W$ is a semismall perturbation potential of the operator $P$ in $\Gw$). Then $\lambda_0(P,W,\Omega)>0$. Moreover, the operator $P -
\lambda W$ is critical in $\Gw$ for $\gl=\gl_0$, and subcritical for $0\leq\gl< \gl_0$.
\end{proposition}

\begin{remark}\label{rem_par}{\em
Assume that $P$ is the Laplace-Beltrami operator $-\Delta$ on a noncompact manifold $\Omega$, then $P\mathbf{1} = 0$ and the cone of positive solutions is nonempty. The manifold is called {\em parabolic} (resp. {\em hyperbolic}) if $-\Delta$ is critical (resp. subcritical) in $\Omega$. For a thorough discussion of the probabilistic interpretation of criticality theory,  see \cite{Pinsky}.
 }
 \end{remark}
 %%%%%%%%%%%%%%%%%%%%%%%%%%%
 Next, we define {\em null-criticality}.
\begin{definition}\label{def_nc}{\em
We say that the operator $P-W$ is {\em null-critical} (resp. {\em positive critical}) {\em in $\Gw$ with respect to the measure $W \mathrm{d} \nu$}  if $P-W$ is critical in $\Gw$, and  $\varphi_0\varphi_0^\star\notin L^1(\Gw,\,W \mathrm{d}\nu)$ (resp. $\varphi_0\varphi_0^\star\in L^1(\Gw,\,W \mathrm{d} \nu)$), where $\varphi_0$, and $\varphi_0^\star$  are the corresponding ground states of $P-W$ and $P^\star-W$ in  $\Gw$.
 }
\end{definition}

Positive criticality is closely related to the large time behavior of the heat kernel (see, \cite{PReview}).
Moreover, if $P$ is symmetric, it is equivalent to the existence of a minimizer for the corresponding variational problem. Indeed, let $q$ be the quadratic form associated to a subcritical operator $P$ in $\Gw$. Consider the space  $\mathcal{D}_P^{1,2}(\Omega)$, the completion of $C_0^\infty(\Omega)$ with respect to the norm $u\mapsto \sqrt{q(u)}$. Since $P$ is subcritical, we know that $\mathcal{D}_P^{1,2}(\Omega)\hookrightarrow W^{1,2}_{\mathrm{loc}}(\Omega)$ (see \cite{Pinchover-Tintarev}) and $\lambda_0(P,W,\Omega)$ is characterized by the Rayleigh-Ritz variational problem:
\begin{equation}\label{variationnal_problem1}
\lambda_0=\inf_{u\in \mathcal{D}_P^{1,2}(\Omega)\setminus\{0\}}\frac{q(u)}{\int_\Omega u^2W\mathrm{d}\nu}\,.
\end{equation}
%%%%%%%%%%%%%%%%%%%
We have (see \cite[Lemma~1.1]{ky2}):
\begin{Lem}\label{null-critical_variationnal}
Assume that $P$ is symmetric and $W>0$ in $\Gw$. Then $P-W$ is positive-critical in $\Gw$ if and only if the infimum in the variational problem \eqref{variationnal_problem1}
is attained, and the infimum is equal $1$. Furthermore, if it is the case, then the corresponding ground state $\varphi_0$ satisfies $\varphi_0\in\mathcal{D}_P^{1,2}(\Omega)$, and realizes the infimum uniquely (up to a multiplicative constant).
\end{Lem}
%%%%%%%%%%%%%%%%%
Finally, we define precisely what we mean by saying that $W$ is ``as large as possible" weight function (cf. Definition~\ref{temp_def}).
\begin{definition}\label{def_ess}{\em
Let $P$ be a subcritical operator in $\Gw$. A nonzero nonnegative function $W$ is said to be an {\em optimal Hardy-weight with respect to $P$ and the domain $\Gw$} if $P-W$ is null-critical in $\Gw$, and for any $\gl>1$, the operator $P-\gl W$ is supercritical in any neighborhood of infinity in $\Gw$.
 }
\end{definition}

\begin{remark}\label{ab_not_c}{\em
It is natural to ask whether all the above properties of an optimal Hardy-weight are independent. The following example shows that the {\em null}-criticality is indeed an additional requirement.

Let $V\in C_0^\infty(\mathbb{R}^n)$ be a potential such that the operator $-\Gd+V(x)$ is critical in $\mathbb{R}^n$ (see Proposition~\ref{CompPot}). Consider the operator $P:=-\Gd+\mathbf{1}+V(x)$, and the potential $W(x):=\mathbf{1}$. Then $$\gl_{0}(P,W,\mathbb{R}^n)=\gl_{\infty}(P,W,\mathbb{R}^n)=1.$$ On the other hand, the operator $P-W$ is null-critical in $\mathbb{R}^n$ for $n\leq 4$, and positive-critical if $n>4$ (see \cite[the paragraph below Theorem~8.6]{PReview}).
 }
\end{remark}

The following theorem provides the precise formulation of the main result of this paper; i.e. the {\em existence of  an optimal Hardy-weight} (cf. Theorem~\ref{thm*}).
\begin{theorem}[Main Theorem]\label{thm_nsa}
Let $P$ be a subcritical operator in $\Gw$, and let $G(x):=G_P^\Gw(x,0)$ be its minimal positive Green function with a pole at $0\in \Gw$. Let $u$ be a positive solution of the equation $Pu=0$ in $\Gw$ satisfying
\begin{equation}\label{u1u0a}
\lim_{x\to\infty} \frac{G(x)}{u(x)}=0,
\end{equation}
where $\infty$ is the ideal point in the one-point compactification of $\Gw$. Consider the positive supersolution
$$
v := \sqrt{G u}
$$
of the operator $P$ in $\Gw^\star$. Then for the associated Hardy-weight
\begin{equation}\label{eq_W}
W:=\frac{Pv}{v}=\frac{1}{4}\left|\nabla \log\left(\frac{G}{u}\right)\right|_A^2
\end{equation}
we have $\lambda_0(P,W,\Gw^\star)=1$, and $W$ is an optimal Hardy-weight with respect to $P$ and the punctured domain $\Gw^\star$.

Assume further that $P$ is a symmetric operator and $W$ is positive in $\Gw^\star$, then the spectrum and the essential spectrum of the Friedrichs extension of the operator $W^{-1}P$ on $L^2(\Gw^\star, W\dnu)$ is equal to  $[1,\infty)$, and the corresponding Agmon metric
$$\qquad \ds^2:= W(x)\sum_{i,j=1}^{n} a_{ij}(x)\dx^i\dx^j, \quad \mbox{where } \big[a_{ij}\big]:= \big[a^{ij}\big]^{-1}$$
is complete.
\end{theorem}
%%%%
\begin{remark}\label{rem_ancona}
{\em
1. If $P$ is a symmetric operator, or more generally if $G_P^{\Gw}(x,y)\asymp G_P^{\Gw}(y,x)$, then a global positive solution $u$ satisfying \eqref{u1u0a} always exists \cite{Ancona02}.

\vskip 3mm

2. If  $u_0,\,u_1$ are two positive solutions of $P u = 0$ near infinity in $\Omega$ such that
\begin{equation}\label{eq_u0u1_rem}
\lim_{x \to \infty} \frac{u_0(x)}{u_1(x)} = 0,
\end{equation}
then $u_0$ is a positive solution of minimal growth in a neighborhood of infinity in $\Gw$ (see Proposition~\ref{minimal}). Therefore, in Theorem~\ref{thm_nsa} we must take $u_0=G$ (the Green function) as a solution satisfying \eqref{eq_u0u1_rem}.

\vskip 3mm

3. By  the uniqueness of the ground state, it follows that  $v=\sqrt{Gu}$ is the \textit{ground state} of $P - W$ in $\Gw^\star$.
 }
\end{remark}
As a consequence of the criticality of $P-W$, we get the following positive Liouville theorem:
\begin{Cor}\label{cor_L}
Under the assumptions of Theorem \ref{thm_nsa}, suppose that $\tilde{v}$ is a positive supersolution of the equation  $(P-W)w=0$ in $\Omega^\star$. Then $\tilde{v}$ is actually a solution of the above equation, and is equal (up to a multiplicative constant) to $\sqrt{Gu}$.
\end{Cor}
%%%%%%%%%%%%%%%%%%
%%%%
We prove Theorem~\ref{thm_nsa} in four steps, see theorems~\ref{thm_Sch}, \ref{thm_Sch7}, \ref{thm_null-critical}, and \ref{essential_spectrum}.

\subsection{Ground state transform}

We recall a standard procedure to eliminate the zero-order term of the operator $P$.
Denote by $\mathcal{V}$ the space $C^{2,\ga}_{\mathrm{loc}}(\Gw)$ (resp. $W^{1,2}_{\mathrm{loc}}(\Gw))$ if $P$ is of the form \eqref{P} (resp. \eqref{div_P}).  Let $h \in \mathcal{V}$ be a positive continuous function and define a map
\begin{equation}
\label{eq:gmap}
T_h: \mathcal{V} \to \mathcal{V},\qquad v \to \frac{v}{h}\,.
\end{equation}
The operator $P_h:=T_h \circ P \circ T^{-1}_h$ given more explicitly by
\begin{equation}
\label{htransform}
P_h u = \frac{P(h u)}{h}
\end{equation}
is called the {\em $h$-transform} of $P$.

Fix $\varphi \in \mathcal{C}_{P}(\Omega)$. Then the corresponding $h$-transform is called a {\em ground state transform}. Clearly,
\begin{equation*}
P_\vgf \mathbf{1}=0 .
\end{equation*}
Moreover, we have
\begin{proposition}[Ground state transform]
\label{gtransform}
Let $\varphi \in \mathcal{C}_{P}(\Omega)$, and let $P_\varphi$ be the corresponding ground state transform. Then
$$\gl_0(P_\vgf,W,\Gw)=\gl_0(P,W,\Gw),
\quad \lambda_\infty(P_\varphi,W,\Omega) = \lambda_\infty(P,W,\Omega).$$
Moreover, $P_\vgf$ is subcritical in $\Gw$ if and only if
$P$ is subcritical in $\Gw$.

The map $T_{\varphi}|_{\mathcal{V}\cap L^2(\Omega,\, \dnu)}$ extends to an isometry between $L^2(\Omega,\, \dnu)$ and $L^2(\Omega,\,\varphi^2 \mathrm{d}\nu)$. In the symmetric case this implies that $P$ and $P_\varphi$ are unitary equivalent.
\end{proposition}
\begin{proof}
The map $T_\varphi$ respects the structure of positive solutions,
$$
T_\varphi \mathcal{C}_P(\Omega) = \mathcal{C}_{P_\varphi}(\Omega),
$$
and preserves support of functions, namely $\supp{v} = \supp{T_\varphi v}$. The claim about $\lambda_{0}$ and $\lambda_\infty$ then follows from their definitions and Proposition~\ref{structure}. The last two claims about the isometry are standard. When $P$ is symmetric it provides independent proof of the spectral claims of the proposition.
\end{proof}
We note that in the subcritical case, the corresponding Green function satisfies
$$G_{P_\vgf}^\Gw(x,y)=\frac{1}{\vgf(x)}G_{P}^\Gw(x,y)\vgf(y).$$
On the other hand, in the critical case $\mathbf{1}$ is the ground state of the equation $P_\vgf u=0$ in $\Gw$.
In addition, if the operator $P$ is symmetric,
then
\begin{equation}\label{Pvgf}
P_\vgf u=-\frac{1}{\varphi^{2}}\mathrm{div}(\varphi^2 A(x)\nabla u),
\end{equation}
 and $P_\varphi$ is manifestly symmetric in $L^2(\Gw,\,\varphi^{2}\mathrm{d} \nu)$.

Calculations are genuinely simplified after a ground state transform. Indeed, if
$P\mathbf{1} = 0 $, then
\begin{align}
   P(uv) &= u P(v) - 2 A \nabla u \cdot\nabla v + vP(u), \label{rule1}\\[3mm]
   P(f(v)) & =  f'(v) P(v) - f''(v) |\nabla v|_A^2 \label{rule2},
\end{align}
holds for all functions $u,\,v \in \mathcal{V}$ and  $f\in C^2(\R)$.

\section{Construction of Hardy-weights}
\label{sec:construction}
The construction of the optimal Hardy-weight using the supersolution method is based on the following simple
observation (\cite[Theorem~3.1]{crit2})
\begin{lemma}[Supersolution construction]
\label{construction}
Let $v_j$ be two positive solutions (resp. supersolutions) of the equation $P u=0$, $j=0,1$, in a domain $\Omega$, and let $v:=v_1/v_0$. Then for any $0\leq \ga\leq 1$ the function
\begin{equation}
\label{solvalpha}
v_\ga(x):= \big(v_1(x)\big)^{\ga}\big(v_0(x)\big)^{1-\ga} = v^\alpha(x) v_0(x)
\end{equation}
 is a positive
solution (resp. supersolution) of the equation
\begin{equation}\label{Pga}
\big[P - 4 \ga(1-\ga)W(x)\big]u=0\qquad \mbox{in } \Gw,
\end{equation}
where $W$ is the Hardy-weight given by
\begin{equation}\label{W}
W(x):=\frac{|\nabla v|_A^2}{4 v^2} \geq 0.
\end{equation}
In fact, $v_j$ are linearly independent if and only if $W \neq 0$.
\end{lemma}
\begin{proof} The proof is obtained by a straightforward calculation (see \cite[Theorem~3.1]{crit2}, or the proof of Proposition~\ref{pro_ends} below). The nonnegativity and non-triviality of $W$ follows from the ellipticity condition \eqref{ellip}.
\end{proof}
Optimizing \eqref{Pga} in $\alpha$, we find for $\alpha=1/2$:

\begin{Cor}\label{v1/2}
The function $\sqrt{v_0 v_1}$ is a positive (super)solution of the equation
$$\big[P - W(x)\big]u=0\qquad \mbox{in } \Gw.$$
In particular, $P-W\geq 0$ in $\Gw$.
\end{Cor}
  We call the above procedure {\em the supersolution construction}, and the corresponding potential $W$ is called a {\em Hardy-weight}. When $v_j$ are positive solutions it is often
useful to apply the ground state transform with respect to $v_0$. This $h$-transform
maps the pair of solutions $(v_0,\,v_1)$ of $P$ to a pair of solutions
$(\mathbf{1},\,v_1/v_0)$ of the equation $P_{v_0}u=0$. For example, \eqref{Pga} is then obtained by applying \eqref{rule1} and \eqref{rule2} with $P=P_{v_0}$, and  $f(t)=t^\ga$. Note that the Hardy-weight $W$ is unchanged  under this ground state transform.

\begin{remark}\label{rem_conv_comb}
{\em Lemma~\ref{construction} has a straightforward generalization to the case when
$v_j$ are positive (super)solutions of $(P -V_j)v_j = 0$, $j=0,\,1$ (cf. \cite[Theorem~3.1]{crit2}). In that case
$v_\alpha$ is a (super)solution of the equation
\begin{equation}\label{Pga3}
\left[P +(1-\alpha)V_0 + \alpha V_1 - 4 \alpha(1-\alpha)W\right]u =0.
\end{equation}
 }
\end{remark}
\begin{example}\label{exconv}
{\em
Suppose that
$P=-\Gd$,  and assume that $\Omega$ is a smooth bounded {\em convex} domain. Consider the function $v_0(x):=\gd(x):= \mathrm{dist} (x,\partial \Omega)$ which due to the convexity is  a positive superharmonic function in $\Gw$, and let $v_1:=\mathbf{1}$. Then the associated weight $W(x)=\gd(x)^{-2}/4$ is the corresponding Hardy-weight, and  we get the well known Hardy inequality \cite{MMP}
\begin{equation}\label{eq_Hardy12}
\int_{\Gw}|\nabla \phi|^2\dx \geq \frac{1}{4}\int_{\Gw}\frac{|\phi|^2}{\gd(x)^2}\dx \qquad \forall \phi\in C_0^\infty(\Gw).
\end{equation}
It is known \cite{MMP} that the operator $-\Gd-W$ is subcritical in $\Gw$, but
 \begin{equation}\label{eqgl_conv}
\lambda_0(-\Gd,\gd(x)^{-2},\Gw)=\lambda_\infty(-\Gd,\gd(x)^{-2},\Gw) =1/4.
\end{equation}
That is, $1/4$ is the best constant in the above inequality in a strong sense.
In fact, \eqref{eqgl_conv} can be deduced from Theorem~\ref{thm_Sch7} (see Example~\ref{ex21}).
 Note also, that if one takes instead the superharmonic function $v_0(x)=\gd(x)^\gb$ with $0<\gb<1$, then one obtains the Hardy inequality without the best constant.
 }
\end{example}
The supersolution construction can be generalized to the case of finitely many positive supersolutions.
\begin{Pro}\label{pro_ends}
Suppose that $P\geq 0$ in $\Gw$,  and let  $u_1,\ldots,u_N$ be positive (super)solutions of $Pv=0$ in $\Gw$.  Let  $\alpha_1,\ldots,\alpha_N$ be nonnegative numbers such that $\sum_{i=1}^N\alpha_i=1$.

Then
\begin{equation}\label{prod_4}
u:=\prod_{j=1}^N u_j^{\alpha_j}
\end{equation}
is a positive supersolution of the equation $Pv=0$ in $\Gw$. Moreover, $u$ is a positive (super)solution of $(P-W)v=0$ in $\Gw$, where
$$W:=\sum_{i<j}\alpha_i\alpha_j\left|\nabla\log\left(\frac{u_i}{u_j}\right)\right|_A^2.$$
\end{Pro}
\begin{proof}
Consider the function  $u$ defined by \eqref{prod_4}.  We compute that
\begin{multline*}
Pu - \sum_{i=1}^N \alpha_i \frac{Pu_i}{u_i}  u\\[2mm]
=\left(\sum_{i=1}^N\alpha_i(1-\alpha_i)\left|\frac{\nabla u_i}{u_i}\right|_A^2-2\sum_{i<j}\alpha_i\alpha_j\left\langle A\frac{\nabla u_i}{u_i},\frac{\nabla u_j}{u_j}\right\rangle\right)u \\[2mm]
\!=\!\! \left(\!\! \sum_{i=1}^N \alpha_i(1 \!-\! \alpha_i)\left|\frac{\nabla u_i}{u_i}\right|_A^2
    \!+\!\sum_{i<j}\!\! \alpha_i\alpha_j \!\!\left[\left| \frac{\nabla u_i}{u_i} \!-\!\frac{\nabla u_j}{u_j}\right|_A^2 \!-\!  \left| \!\frac{\nabla u_i}{u_i}\right|_A^2 \!-\!  \left|\!\frac{\nabla u_j}{u_j}\right|_A^2\right]\!\right)\!u\\[2mm]
=\left(\sum_{i=1}^N\alpha_i\Big(1-\sum_{j=1}^N\alpha_j\Big)\left|\frac{\nabla u_i}{u_i}\right|_A^2+W\right)u
= W u,
\end{multline*}
since by hypothesis $\sum_{i=1}^N\alpha_i=1$.
\end{proof}
The supersolution construction given in Proposition~\ref{pro_ends} will be used in Section~\ref{sec_sever}, where we study the case of a subcritical operator which is defined on a manifold with $N$ ends, with  $N\geq 2$.

\vskip 3mm

Let us focus again on the case of two ends. Let $W$ be the Hardy-weight given in Lemma~\ref{construction} by \eqref{W}. The set of solutions of the equation $$(P-\lambda W)u=0 \qquad \mbox{in } \Gw$$ for $\gl\in\R$ plays a crucial role throughout the article. Indeed, under the assumptions of Lemma~\ref{construction}, for $\lambda < 1$ the equation $P - \lambda W$ admits two positive (super)solutions
\begin{equation}
\label{pmsolutions}
v_{\alpha_\pm}(x)= \big(v_1(x)\big)^{\ga_\pm}\big(v_0(x)\big)^{1-\ga_\pm},\qquad \mbox{where }\alpha_{\pm}:=\frac{1 \pm \sqrt{1 - \lambda}}{2}\,.
\end{equation}
At the maximum $\lambda=1$ the construction gives a  positive (super)solution $v_{1/2}$ of $(P - W)u=0$. We obtain a second solution for $\gl=1$ by differentiating \eqref{Pga} with respect to the parameter $\alpha$ and substituting $\alpha =  \frac{1}{2}$,
$$
\partial_\alpha\Big\{\big[P - 4 \ga(1-\ga)W(x)\big]v_{\alpha}\Big\}\Big|_{\alpha=\frac{1}{2}}=(P - W)\left[\sqrt{v_0 v_1} \,\log \left(\frac{v_0}{v_1}\right)\right] = 0.
$$
To avoid justification of the differentiating with respect to $\ga$, we give an independent proof of this formula.
\begin{lemma}\label{lem_logu}
Assume that $P$ is a subcritical operator in $\Gw$. Let $v_j$ be two linearly independent positive solutions of the equation $Pu=0$ in $\Omega$, where $j=0,1$.  Let $W$ be the associated Hardy-weight given by \eqref{W}. Then the equation
\begin{equation}\label{eqW1}
\left(P- W\right)u=0 \qquad \mbox{ in } \Gw,
\end{equation}
admits a solution $w:=\sqrt{v_0 v_1} \log \left(\frac{v_0}{v_1}\right)$.
\end{lemma}
\begin{proof}
In light of the ground state transform with respect to the function $v_0$,  we may assume that $v_0=\mathbf{1}$, and let us denote $v:=v_1$. So, $P\mathbf{1}=Pv=0$ and,
by the construction of $W$, $(P- W)v^{1/2}=0$ in $\Gw$.
Then using \eqref{rule1} and \eqref{rule2} we obtain
\begin{align*}
   P(v^{1/2} \log v) &= P(v^{1/2}) \log v - 2 A \nabla v^{1/2} \cdot \nabla \log v +v^{1/2} P(\log v)\\
		     &= P(v^{1/2}) \log v + v^{1/2}\frac{1}{v} P(v) \\
		    &= W v^{1/2} \log v.
\qedhere
\end{align*}
\end{proof}
\begin{remark}\label{rem_log}{\em
Another way to understand the $\log$-type solution is as follows. Suppose that $P$ is of the form $Pu=-\mathrm{div}(A \nabla u)$, and let $v$ be a nonconstant positive solution of the equation $Pu=0$ in $\Gw$. Then by the supersolution construction with respect to the solutions $\mathbf{1}$ and $v$ we have $\Big(P- W\Big) v^{1/2}=0$, where $W$ is the Hardy-weight. Moreover, by \eqref{Pvgf}, the ground state transform with respect to $v^{1/2}$ gives
$$
(P - W)_{\sqrt{v}}(u) = -\frac{1}{v} \div (v A \nabla u),
$$
which readily implies  an equivalent formulation of Lemma~\ref{lem_logu},
\begin{equation}
\label{log}
(P - W)_{\sqrt{v}} \log v = 0.
\end{equation}
 }
 \end{remark}

%%%%%%%%%%%%%%%%%%%%%%
\mysection{The criticality of $P-W$}\label{sec2}
%%%%%%%%%%%%%%%%%%%%%%%%%%%%%%%%%%%%%%%%%%%%%%%%%%
In the present section we prove the first assertion of the main theorem (Theorem~\ref{thm_nsa}). Namely, we prove that under assumption \eqref{u1u0a}, the operator $P-W$ is {\em critical} in $\Gw^\star$. We start with a preliminary result.
\begin{proposition}\label{minimal}
Let $P$ be a second-order elliptic operator in $\Omega$ and let $u_0,\,u_1$ be two positive solutions of $P u = 0$ near infinity in $\Omega$ such that
$$
\lim_{x \to \infty} \frac{u_0(x)}{u_1(x)} = 0.
$$
Then $u_0$ is a  positive solution of minimal growth in a neighborhood of infinity in $\Gw$.
\end{proposition}

\begin{proof}
Let $K$ be a smooth compact set in $\Omega$ such that $u_0$ and $u_1$ are positive and continuous in $\big(\Gw\sm K\big)\cup \partial K$, and are solutions of $Pu=0$ in $\Gw\sm K$. Let $\{\Omega_k\}$ be an exhaustion of $\Omega$, such that $K\subset \Omega_0$, and let $w_k$ be the solution of the following Dirichlet problem:
\begin{equation} \label{eq:minimal}
\left\{
\begin{array}{lr} P w_k=0 & \qquad\mbox {in } \Omega_k\setminus K, \\[2mm]
		  w_k(x)=u_0 & \qquad \mbox{on }\partial K, \\
          w_k(x)=0 & \qquad \mbox{on }\partial \Omega_k.
\end{array} \right.
\end{equation}
Then by the generalized maximum principle, $\{w_k\}_{k\in\mathbb{N}}$ is an increasing sequence of nonnegative functions, satisfying $w_k\leq u_0$, and therefore, converging to a positive solution $w$ of $Pu=0$ in $\Omega\setminus K$, that clearly has minimal growth at infinity in $\Gw$. Thus, it is enough to show that $u_0=w$ in $\Omega\setminus K$. We obviously have  $w\leq u_0$. On the other hand, by hypothesis, if $\varepsilon>0$, there is $k_{\varepsilon}$ such that $u_0\leq \varepsilon u_1$ on $\partial \Omega_k$, for every $k\geq k_{\varepsilon}$. By the generalized maximum principle, this implies that $u_0 \leq w_k+\varepsilon u_1$ in $\Omega_k\setminus K$ and it follows $u\leq w+\varepsilon u_1$  in $\Omega\setminus K$. By letting $\varepsilon \to 0$, we conclude that $u_0\leq w$. Thus, $u_0=w$  in $\Omega\setminus K$.
\end{proof}
We are ready to prove the criticality statement of Theorem~\ref{thm_nsa}.
\begin{theorem}\label{thm_Sch}
Under the hypotheses of Theorem \ref{thm_nsa}, the operator $P - W$ is critical in $\Gw^\star:=\Omega\setminus \{0\}$, and has a ground state $\sqrt{Gu}$.
\end{theorem}
%%%%%%%%%%%%%%%%%%%%%%
\begin{remark}\label{rem_a1}{\em
1. Theorem~\ref{thm_Sch} readily implies, that $\gl=1$ is the best constant for the validity of the inequality $P-\gl W\geq 0$ in $\Gw^\star$.  In particular, if $P$ is a symmetric operator, it follows that the best constant for the Hardy-type  inequality (\ref{HardyIn}) is $\lambda_0 = 1$.

2.  Theorem~\ref{thm_Sch} implies also that $\sqrt{Gu}$ is the unique (up to a multiplicative constant) positive solution of the equation $(P-W)w=0$ in $\Gw^\star$. Hence, for $\gl\leq 1$ the {\em positive Liouville theorem} holds true for the operator $P-\gl W$ in $\Gw^\star$ if and only if $\gl=1$ (cf. Corollary~\ref{cor_L}).
}
\end{remark}
We present three proofs of Theorem~\ref{thm_Sch}. The shortest one uses the log solution for $P-W$, as well as the notion of minimal growth and is as follows:
\begin{proof}[Proof of Theorem~\ref{thm_Sch}]
By Corollary \ref{v1/2} and Lemma \ref{lem_logu}, the equation $(P-W)u=0$ admits two solutions
$$
u_0 = \sqrt{G u} \quad \mbox{and} \quad u_1 = -\sqrt{G u} \log\left(\frac{G}{u}\right).
$$
By assumption (\ref{u1u0a}), these solutions are positive near infinity and
$$
\lim_{x \to \infty} \frac{u_0(x)}{u_1(x)} = 0.
$$
Proposition \ref{minimal} then implies that $u_0$ is a positive solution of the equation $(P-W)u=0$ of minimal growth in a neighborhood of infinity in $\Omega$. By the same argument and using the positive solution $-u_1$ in a neighborhood of zero, we conclude that $u_0$ has minimal growth in a neighborhood of zero.  The second part of Lemma~\ref{minimal_vanish} implies now that $u_0$ has minimal growth at infinity in $\Omega^\star$. Therefore,  $u_0$ is a ground state of $P - W$ in $\Omega^\star$, so, $P - W$ is critical in $\Omega^\star$.

\end{proof}
\paragraph{\em Alternative proof 1:} Let $\alpha \in (0,1/2) $ and consider $v_\alpha:=G^\alpha u^{1-\alpha}$ (cf.  \eqref{solvalpha}). Then $v_\alpha$ and $v_{(1 - \alpha)}$ are positive solutions of $P-4\alpha(1-\alpha)W$ that satisfies

$$\frac{v_\alpha}{v_{(1-\alpha)}}=\left(\frac{G}{u}\right)^{2\alpha - 1}.$$
Therefore, assumption (\ref{u1u0a}) and the singularity of Green's function at $0$ imply

$$\lim_{x\to \infty}\frac{v_{(1-\alpha)}(x)}{v_\alpha(x)}=0 \quad \mbox{and} \quad  \lim_{x\to 0}\frac{v_\alpha(x)}{v_{(1-\alpha)}(x)}=0.$$
Consequently, applying Proposition \ref{minimal}, we deduce that $v_{\alpha}$ has minimal growth at zero, and $v_{(1-\alpha)}$ has minimal growth at infinity (both for the operator $P-4\alpha(1-\alpha)W$). This implies that $\sqrt{Gu}=\lim_{\alpha\to1/2}v_\alpha=\lim_{\alpha\to1/2}v_{(1-\alpha)}$ has minimal growth at zero and at infinity for $P-W$, as we explain now.

Indeed, let $v$ be a positive supersolution for $P-W$ in a neighborhood of zero, that we assume for simplicity to be $B(0,1)\setminus\{0\}$. Then $v$ is a positive supersolution of $P-4\alpha(1-\alpha)W$ in $B(0,1)\setminus\{0\}$ for $0\leq \ga\leq 1$. Since on $\partial B(0,1)$, $v$ and $v_\alpha$ are bounded above and below by positive constant that does not depend on $\alpha$, we deduce that there is a constant $C$ independent of $\alpha$ such that
$$v_\alpha\leq Cv.$$
Letting $\alpha\to1/2$, we deduce that
$$\sqrt{Gu}\leq Cv,$$
hence $\sqrt{Gu}$ has minimal growth at zero. The proof at infinity repeats the same argument with the
solution $v_{(1-\alpha)}$. \qed

\vskip 3mm

\paragraph{\em Alternative proof 2:} Here we explain how to prove the criticality of $P-W$, using once more the $\log$ solution, but without the use of the notion of minimal growth. By performing a ground state transform with respect to $u$, we can assume that $u=\mathbf{1}$.

We need to prove that the operator $Q:=P - W$ is a critical operator in $\Omega^\star$. Notice that the supersolution construction gives that $Q(G^{1/2})=0$ on $\Omega^\star$, where $G$ is the Green function for $P$ with a pole $0$. Let us perform a ground state transform for $Q$ with respect to its positive solution $G^{1/2}$. We get a second-order elliptic operator $\tilde{Q}:=Q_{G^{1/2}}$. By Lemma~\ref{gtransform}, the operator  $\tilde{Q}$ is critical in $\Omega^\star$ if and only if $Q$ is critical in $\Omega^\star$.
By Lemma~\ref{lem_logu} (cf. Equation~(\ref{log})) we have,
$$\tilde{Q}(\log(G))=0\qquad \mbox{ in } \Omega^\star.$$
So, in $\Omega^\star$, we have two solutions of the equation $\tilde{Q}u=0$, namely $\mathbf{1}$ and $w:=\log(G)$. Note that
$$\lim_{x\to \infty}w(x)=-\infty, \qquad \lim_{x\to 0}w(x)=\infty,$$
where the first limit is due to our assumption (\ref{u1u0a}).

We claim that this implies that $\tilde{Q}$  is critical in $\Omega^\star$ (this is reminiscent of the Khas'minski\u{i} criterion for recurrency, cf. \cite{Pinsky}, see  also a related claim in \cite[Corollary~3.10]{Shuttle}).

Assume on the contrary that $\tilde{Q}$ is subcritical in $\Omega^\star$, and let $\tilde{G}(x)=G_{\tilde{Q}}^{\Omega^\star}(x,x_1)$ be the corresponding Green function with a pole at $x_1\in \Omega^\star$. Let $K$ be a compact annular domain around $0$ containing $x_1$ such that $G(x)=M$ on the inner boundary and $G(x)=M^{-1}$ in the outer boundary, where $M>1$ is a large positive number. So, $\Gw =K_0\cup K\cup K_\infty$ where $K_0$ is a neighborhood of $0$, and $K_\infty$ is a neighborhood of $\infty$.

By the minimality of $\tilde{G}$ and the fact that $\tilde{Q}\mathbf{1}=0$, we have $$\inf_{x\in \Omega^\star}\tilde{G}(x)=0.$$
Therefore, either
$\liminf_{{x\to 0}}\tilde{G}(x)=0$ or $\liminf_{{x\to \infty}}\tilde{G}(x)=0$.
Suppose first that  $\liminf_{{x\to 0}}\tilde{G}(x)=0$, and let $$D_k^0: =\{x\in K_0 \mid M<G(x)<k\}.$$ $D_k^0$ is a union of open, relatively compact, connected sets in $\Omega^\star$, whose boundaries are contained in $\{x : G(x)=M\}\cup\{x : G(x)=k\}$. Furthermore, the sequence $\{D_k^0\}_{k\in \mathbb{N}}$ is increasing and is an exhaustion of $K_0\setminus\{0\}$.
Let $v_k$ be the solution of the Dirichlet problem
\begin{equation} \label{eq:D0}
\left\{
\begin{array}{lr} \tilde{Q}u=0 & \qquad\mbox {in } D_k^0, \\[2mm]
		  u(x)=1 & \qquad \mbox{on } \partial D_k^0\cap\{x : G(x)=M\}, \\[2mm]
          u(x)=0 & \qquad \mbox{on } \partial D_k^0\cap\{x : G(x)=k\}.
\end{array} \right.
\end{equation}
Let $C>0$ such that $\tilde{G}\geq C^{-1}$ on $\{x : G(x)=M\}.$ Then by the maximum principle $0\leq  v_k\leq C\tilde{G}$. For $k$ big enough, the set $\partial D_k^0\cap\{x : G(x)=M\}$ is independent of $k$, and by the maximum principle  $v_k$ is a  bounded nondecreasing sequence, converging to a positive function $v_0$ which solves the equation $\tilde{Q}u=0$ in $K_0\setminus \{0\}$, and satisfies $ v_0\leq C\tilde{G}$ in $K_0\setminus \{0\}$. On the other hand, we have an explicit formula for $v_k$:
$$v_k(x)=\frac{\log k -w(x)}{\log k-\log M}.$$ Hence $v_0=\mathbf{1}$, and consequently $\tilde{G}\geq C^{-1}$ in $K_0\setminus \{0\}$ which contradicts our assumption.

A similar argument shows that $\liminf_{{x\to \infty}}\tilde{G}(x)=0$ cannot happen. Hence, we obtain a contradiction to our assumption that $\tilde{Q}$ is subcritical in $\Omega^\star$. \qed
%%%%%%%%%%%%%%%%%%%

\mysection{$\lambda_\infty(P,W,\Omega^\star)=1$}\label{sec22}
%%%%%%%%%%%%%%%%%%%%%%%%%%%%%%%%%%%%%%%%%%%%%%%%%%
In the present section we prove that for any $\gl>1$ the equation $(P-\gl W)u=0$ does not admit any positive solution neither in any neighborhood of infinity in $\Gw$, nor in any punctured neighborhood of $0$.

We first state the following lemma which extends Lemma~\ref{construction} concerning the supersolution construction. The proof is obtained by a direct computation.
 \begin{lemma}
\label{lem_construction1}
Let $v_j$ be two positive solutions of the equation $P u=0$, $j=0,1$, in a domain $\Omega$, and let $v:=v_1/v_0$. Then for any $\lambda\in \mathbb{R}$ and $\ga\in \mathbb{C}$  satisfying $\gl=4 \ga(1-\ga)$, the function
\begin{equation}
\label{solvalpha4}
v_\ga(x):= \big(v_1(x)\big)^{\ga}\big(v_0(x)\big)^{1-\ga} = v^\alpha(x) v_0(x)
\end{equation}
 is a solution of the equation
\begin{equation}\label{Pga4}
\big[P - \gl W(x)\big]u=0\qquad \mbox{in } \Gw,
\end{equation}
where
\begin{equation}\label{W4}
W(x):=\frac{|\nabla v|_A^2}{4 v^2} \geq 0.
\end{equation}
\end{lemma}
Our main result of this section is given in the following theorem.
\begin{theorem}\label{thm_Sch7}
Under the assumptions of Theorem~\ref{thm_nsa} we have
$$\lambda_\infty(P,W,\Omega)=\lambda_\infty(P,W,\Omega^\star)=1.$$
More precisely, for any $\gl>1$ the equation $(P-\gl W)u=0$ does not admit any positive solution neither in any neighborhood of infinity in $\Gw$, nor in any punctured neighborhood of $0$.
\end{theorem}
\begin{proof}
To simplify the notations we assume that $u=\mathbf{1}$ in the assumptions of Theorem~\ref{thm_nsa} (in particular, $P\mathbf{1}=0$ in $\Gw$). The general case then follows by ground state transform (see Proposition~\ref{gtransform}).

Fix $\lambda>1$ and $K$ a compact subset of $\Omega$ containing $0$. We need to show that  the operator $P-\lambda W$ cannot be nonnegative on $K^c:=\Omega\setminus K$.

  By Lemma~\ref{lem_construction1}, we have
$$\left(P- \lambda W\right)G^\alpha=0 \qquad \mbox{ in } K^c,$$
where $\ga$ is a \textit{complex} number satisfying $4 \alpha (1-\alpha) = \lambda$.
Inverting the relation, we get that
$$\left(P-\lambda W\right)G^{\frac{1}{2}+\mathrm{i}\xi}=0,$$
where
$$\xi:=\frac{\sqrt{\lambda-1}}{2}.$$
By taking the real part
$$\varphi:=\mathfrak{Re}(G^{\frac{1}{2}+\mathrm{i}\xi})=G^{1/2}\cos\left(\xi\log(G)\right),$$ we obtain an oscillatory solution of the equation
$$\left(P-\lambda W\right)u=0 \qquad \mbox{ in } K^c.$$
We claim that the existence of such an oscillatory solution $\varphi$ implies that $P-\lambda W$ is supercritical in $K^c$ (i.e.  $P-\lambda \not \geq 0$ in $K^c$).

Indeed, since $\lim_{x\to\infty} G(x)=0$, we can find a connected component $U$ of the open, relatively compact set $\{x : 0<a<G(x)<b\}$ contained in $K^c$, where $a$ and $b$ are chosen so that
$$\cos\left(\xi\log a \right)=\cos\left(\xi\log b\right)=0,$$
and such that $\varphi$ has a constant sign on $U$, for example $\varphi>0$ on $U$. Then since $\varphi$ vanishes on the boundary of $U$ and is positive on $U$, it has a local maximum point in $U$. If the generalized maximum principle for $P-\lambda W$ would hold, we would deduce that $\varphi$ is zero on $U$, which is a contradiction. Therefore, the generalized maximum principle for $P-\lambda W$ does not hold in $K^c$, and hence $P-\lambda W \not\geq 0$ in $K^c$. Since $K$ is an arbitrary compact set containing $0$, it follows that  $P-\lambda W$ cannot admit a positive (super)solution in any neighborhood of infinity in $\Gw$.

Similarly, one shows that for any $\gl>1$, the generalized maximum principle for $P-\lambda W$ does not hold in any punctured neighborhood of the origin.
\end{proof}
%%%%%%%%%%%%%%%%%%%%%%%%%%%%%%%%%%%%%%%%%%%%%%%%%
The next result demonstrates that the asymptotic behavior of the constructed optimal Hardy-weight near $0$ is exactly like the classical Hardy potential near the origin. Without loss of generality we may assume that the matrix $A=\big[a^{ij}\big]$ at $0$ is equal to the identity matrix.
\begin{theorem}\label{cor_zero}
Assume that $n\geq 3$, the coefficients of $P$ are smooth enough near $0$, and $a^{ij}(0)=\gd_{ij}$. Suppose further that the assumptions of Theorem~\ref{thm_nsa} holds true. Then
 $$\lim_{x\to 0} |x|^2W(x)=C_H=\left(\frac{n-2}{2}\right)^2.$$
\end{theorem}
\begin{proof}
It is well known that near the origin we have $G(x)\sim |x|^{n-2}$. Moreover, using \cite{M65} we know also the asymptotic near $0$ of $|\nabla G(x)|$. Hence, an elementary calculation shows that
%%%%%%%%%%%%%
$$\lim_{x\to 0}\frac{|x|^2\left|\nabla \left(\frac{G}{u}\right)\right|^2}{4\left|\left(\frac{G}{u}\right)\right|^2}=C_H.$$

\end{proof}
%%%%%%%%%%%%%%
The next result demonstrates that if $P$ is symmetric, Theorem~\ref{thm_Sch7} implies that the decay of the weight $W$ near infinity is ``optimal" in the following sense.
%%%%%%
%%%@@ Old Errata @@
%%%%%%%%%%%
\begin{corollary}\label{cor1}
Suppose that $P$ is a symmetric operator that satisfies the assumptions of Theorem~\ref{thm_nsa}, and assume further that
$$\lim_{x \to \infty} W(x) = 0.$$
 Then for every $\lambda > 1$ and every locally regular potential $\tilde{W}$ such that $\tilde{W}=W$ outside a compact neighborhood of $0$,  the (Friedrichs extension of the) operator $P-\lambda \tilde{W}$  has an infinite number of negative eigenvalues accumulating at zero.
\end{corollary}
\begin{proof}
@@Let $Q$ be a symmetric linear elliptic operator of second-order with real coefficients, and let $q$ be its associated quadratic form.@@ Assume that the bottom of the essential spectrum of $Q$ is zero, i.e. $\lambda_\infty(Q,1,\Omega) = 0$. The number of negative eigenvalues of $Q$ (counting multiplicities) is given by the Morse index
$$
\sup\{\dim(F) : F\subset \mathcal{D}_Q^{1,2}(\Omega),\,q|_F<0\}.
$$
@@ Then the finiteness of this index is characterized by the following property of positive solutions of $Q$ \cite{Fischer-Colbrie, Baptiste}:@@ The Morse index is finite if and only if there exists a positive solution of $Qu = 0$ outside of a compact neighborhood of zero.

Due to our assumption
$$\lim_{x \to \infty} W(x) = 0,$$
we have that $\lambda_\infty (P - \lambda \tilde{W},1,\Omega) = 0$ for any $\lambda \geq 1$ and by Theorem~\ref{thm_Sch7} there are no positive solutions of @@$(P - \lambda \tilde{W}) u=0$@@ in a neighborhood of infinity. The corollary then follows by the above characterization of the Morse index.
\end{proof}
%%%%%%%%%
%%%%%@@ End of Old Errata @@
%%%%%%%%%
\begin{remark}\label{rem_Dev}{\em
Recently B.~Devyver \cite[Theorem~5.6]{Dev} proved the following complementary result:

{\em Let $P$ be a (general) subcritical operator in $\Gw$, and let $V$ and $W$ be nonzero nonnegative functions defined in $\Gw$ such that
$$\lim_{x\to\infty}\frac{V(x)}{W(x)}=0.$$
If $\gl_\infty(P,W,\Gw)> 0$, then $\gl_\infty(P,V,\Gw)= \infty$.

Moreover, if $V>0$ and $P$ is symmetric, then the spectrum of $V^{-1}P$ consists of
an increasing sequence of eigenvalues tending to $\infty$, and if $\gl\in \mathbb{R}$ does not
belong to the spectrum of $V^{-1}P$, then the resolvent $(V^{-1}P-\gl)^{-1}$ is compact.}

Indeed, if for some $\gm>0$, a function $v$ is a positive supersolution of the equation $(P-\gm W)u=0$ in a neighborhood of infinity in $\Gw$,
then for any $m >0$ the function $u$ is a positive supersolution of the equation $(P-m V)u=0$ in a neighborhood of infinity in $\Gw$, and hence $\gl_\infty(P,V,\Gw)= \infty$. See \cite[Theorem~5.6]{Dev} for the proof for symmetric operators, and for further results.
 }
 \end{remark}
%%%%%%%%%%%%%%%%%%%%%%%%%%%%%%%%%%%%
\mysection{Null-criticality}\label{sec_nc}
%%%%%%%%%%%%%%%%%%%%%%%%%%%%%%%%%%%%%%%%%%%%%%%%%

Under the hypotheses of Theorem~\ref{thm_nsa}, we know (by Theorem~\ref{thm_Sch}) that the operator $P-W$ is critical in $\Gw^\star$. Let $\varphi_0$ be the ground state of $P-W$, and $\varphi_0^\star$ be the ground state of $P^\star-W$ (which is also a critical operator  in $\Gw^\star$). In this section we study integrability properties of these ground states. In particular, if $P$ is symmetric, we study whether the corresponding ground state belongs to $L^2(\Gw^\star,\,W\mathrm{d} \nu)$. Note that since $\varphi_0$ is continuous its integrability is determined by its behavior at infinity and zero.

\begin{definition}\label{def_nc1}{\em
Assume that $P-W$ is critical in $\Gw^\star$, and let $\varphi_0$ and $\varphi_0^\star$ be the ground states of $P-W$,  and $P^\star-W$, respectively. We say that $P-W$ is \textit{null-critical at infinity} if
$$\int_{\Omega\setminus K}\varphi_0(x)\varphi_0^\star(x) W(x)\mathrm{d} \nu=\infty,$$
for (any) compact set $K$ containing zero. Similarly, we define \textit{null-criticality at zero}.
 }
 \end{definition}
We have:
\begin{Thm}\label{thm_null-critical}
Under the assumptions of Theorem~\ref{thm_nsa}, the operator $P-W$ is null-critical at infinity and at zero.
\end{Thm}
\begin{Rem}\label{rem_int} {\em
 If $P$ is symmetric, then the null-criticality at zero follows at once from Theorem~\ref{cor_zero}. In fact, if $P$ is  symmetric, the null-criticality both at zero and at infinity follows readily from Corollary \ref{Estimate W}.
 }
\end{Rem}

\begin{proof}[Proof of Theorem~\ref{thm_null-critical}] Recall that the explicit form of $\varphi_0$ is known. On the other hand, in contrast to the  symmetric case, the explicit form of $\varphi_0^\star$ is unknown in the nonsymmetric case. Consequently, the proof is much subtler. Therefore, to illustrate the idea of the proof in the general case, we first present the proof in the symmetric case.

So, let us first assume that $P$ is a symmetric operator. We assume as before that $P\mathbf{1}=0$, the general case then follows by the ground state transform.  Recall that for $\xi \geq 0$,  the function
$$\varphi_\xi:=G^{1/2}\cos(\xi \log(G))$$
solves the equation
$$(P- (4 \xi^2 + 1) W)u=0.$$
In particular $\varphi_0=G^{1/2}$ is the ground state.

Define a set
\begin{equation}
\label{OmegaXi}
\Omega_{\xi} :=\left\{x\,:\,  -\frac{\pi}{2\xi} <\log G(x) < 0 \right\}
\end{equation}
and consider the solutions $\varphi_{\xi},\,\varphi_{3 \xi}$. These solutions as formal eigenfunctions of a mixed value boundary problem on $\Omega_{\xi}$ lead to the following orthogonality relation
\begin{equation}\label{Orthogonality}
\int_{\Omega_{\xi}} \varphi_{\xi} \varphi_{3 \xi} \, W \mathrm{d} \nu = 0.
\end{equation}

Let us prove (\ref{Orthogonality}) in detail.
Assume first that $\Omega_\xi$ is regular enough, then we have the following Green formula for $P$:
\begin{equation} \label{Green}
\int_{\Omega_\xi}\!\!\! \big(P[\varphi_\xi] \varphi_{3 \xi}- \varphi_\xi P[\varphi_{3 \xi}]\big)\dnu=\int_{\partial \Omega_\xi}\!\!\!\! \Big\langle A \nabla [\varphi_\xi] \varphi_{3 \xi}- A\nabla [\varphi_{3 \xi} ] \varphi_{\xi},\vec{\gs} \Big\rangle\,\mathrm{d}\gs,
\end{equation}
where $\mathrm{d}\gs$ is  the induced measure on $\partial \Omega_\xi$ and $\vec{\gs}$ is the outward unit normal vector field on $\partial \Omega_\xi$. By construction, the functions $\varphi_\xi,\,\varphi_{3 \xi}$
vanish on the set $\log G = -\pi/(2\xi)$. On the other hand, on the part of the boundary contained in $\{\log G = 0\}$ we have
\begin{equation}\label{tangent}
\varphi_\zeta= 1 \quad \mbox{and} \quad \nabla \varphi_\zeta=\nabla\varphi_0
\end{equation}
for all $\zeta$. It follows that the right hand side of the Green formula \eqref{Green} vanishes. This establishes (\ref{Orthogonality}) since the left hand sides of (\ref{Orthogonality}) and (\ref{Green}) are nonzero multiple of each other.

For a nonregular $\Omega_\xi$ the claim follows by approximation of $\Omega_\xi$ by regular domains.

Now, assume that $\varphi_0\in L^2(\Gw\sm K, W\mathrm{d}\nu)$ and note that
$$
|\varphi_\xi| \leq \varphi_0
$$
for all $\xi \geq 0$.
 Letting $\xi \to 0$ in (\ref{Orthogonality}), we conclude by the dominated convergence theorem that
$$\int_{\{G<1\}}\varphi_0^2W\mathrm{d}\nu=0,$$
which is a contradiction since $\varphi_0>0$ and $W\gneqq 0$ on $\{G<1\}$. The proof of the null-criticality near zero is analogous.

\vskip 3mm

\textit{The general case:}
The proof follows the same idea as above, but since an explicit formula for the ground state $\varphi_0^\star$ of the adjoint operator $P^\star-W$ is not available, we construct instead an approximating sequence for $\varphi^\star_0$ .

Consider the domain $\Omega_\xi$ defined by (\ref{OmegaXi}), and let $\varphi^\star_\xi$ be the solution
 of the Dirichlet problem
\begin{equation} \label{oscillation}
\left\{
\begin{array}{ll} (P^\star- W)u=0 & \qquad\mbox {in } \Omega_\xi, \\[2mm]
		  u(x)=\varphi_0^\star & \qquad \mbox{on } \{\log G=0\}, \\[3mm]
          u(x)=0 & \qquad \mbox{on } \left\{ \log G=-\pi/(2\xi) \right\}.
\end{array} \right.
\end{equation}
Since $P^\star-W$ is subcritical in $\Omega_\xi$, the generalized maximum principle implies that $\varphi^\star_\xi$ is  positive,  $\varphi^\star_\xi \leq
 \varphi_{0}^\star$ on $\Omega_\xi$, and the sequence $\{\varphi^\star_{\xi}\}$ is increasing with respect to $\xi$.

Therefore, as $\xi \searrow 0$, we have $\varphi^\star_\xi \to \varphi^\star\leq \varphi_0^\star$ locally uniformly in $\Omega^\star \setminus K$, where $K=\{G>1\}$ is a neighborhood of zero,  and $\varphi^\star$ is a nonnegative solution of the equation $(P^\star-W)u=0$ in $\Omega\setminus K$. Since $\varphi^\star_0$ is a ground state of $P^\star-W$ in $\Gw^\star$, it has minimal growth at infinity of $\Gw$, and hence $\varphi_0^\star\leq  \varphi^\star$. Thus,   $\varphi^\star= \varphi_0^\star$, and we obtain
$$\lim_{\xi\searrow 0}\varphi^\star_\xi=\varphi_0^\star.$$

We use Green's formula for the operator $Q:=P -W$:
\begin{equation}
\label{GGreen}
\int_{\Omega_{\xi}} Q[u] \varphi^\star_\xi \dnu = \int_{\Omega_{\xi}} \left( Q[u] \varphi^\star_\xi - u Q^\star[\varphi^\star_\xi]\right) \dnu = B.T.\,,
\end{equation}
where $u$ is either $\varphi_\xi$ or $\varphi_{3 \xi}$, and $B.T.$ is the corresponding boundary term. We claim that $B.T.$ is independent of the choice of either $\varphi_\xi$ or $\varphi_{3 \xi}$. Indeed, the claim readily follows from (\ref{tangent}), \eqref{oscillation}, and the explicit form
\begin{equation}\label{boundary term}
B.T. = \int_{\{G=1\}}\!\!\!\left \langle A \nabla[\varphi_0] \varphi_0^\star -  A \nabla[\varphi_\xi^\star] \varphi_0 + \mathbf{b} \varphi_0 \varphi_0^\star- \tilde{\mathbf{b}} \varphi_0 \varphi_0^\star,\vec{\gs} \right\rangle \,\mathrm{d}\gs.
\end{equation}

We have
\begin{multline*}
\int_{\Omega_\xi} 4 \xi^2 \varphi_\xi \varphi^\star_\xi \,W \mathrm{d} \nu = \int_{\Omega_{\xi}} Q[\varphi_\xi] \varphi^\star_\xi \dnu = B.T. \\
= \int_{\Omega_{\xi}} Q[\varphi_{3 \xi}] \varphi^\star_\xi \dnu= \int_{\Omega_\xi} 4 (3\xi)^2 \varphi_{3 \xi} \varphi^\star_\xi \,W \mathrm{d} \nu.
\end{multline*}
Hence,
$$ \int_{\Omega_\xi}  \varphi_\xi \varphi^\star_\xi \,W \mathrm{d} \nu =
9\int_{\Omega_\xi}  \varphi_{3 \xi} \varphi^\star_\xi \,W \mathrm{d} \nu.$$
Assuming that $\varphi_0 \varphi^\star_0W$ is $\nu$-integrable in $\Omega \setminus K$, we can pass to the limit $\xi \to 0$ and obtain the contradiction $1 = 9$. The case of a nonregular domain $\Omega_\xi$ can again be treated by approximations. The proof of null-criticality near zero is analogous.
\end{proof}
\begin{corollary}\label{cor_nc1}
Assume further that $P$ is subcritical in $\Gw$,  symmetric in $L^2(\Omega,\dnu)$, and $P\mathbf{1}=0$.
Then
\begin{equation}\label{eqG}
    \frac{|\nabla G|^2_A}{G}
\end{equation}
is not $\nu$-integrable  neither near $0$ nor near infinity in $\Gw$.
 \end{corollary}
%%%%%%%%%%%%%%%%%%%%%%%%%%%%%%%%%%%%%%%%%%%
\mysection{The essential spectrum}\label{sec_essential}
In the present section (unless otherwise stated), we assume that $P$ is a subcritical symmetric operator defined on $\Gw$.  We continue our study of the supersolution construction with the  pair $(u,G)$, where $G(x)=G_P^{\Gw}(x,0)$ and $u$ satisfy (\ref{u1u0a}). Moreover, throughout this section we assume that the corresponding (optimal) Hardy-weight $W$ is {\em strictly positive} in $\Gw^\star$.

\begin{remark}\label{rem_unique_cont}{\em A natural question is to find sufficient conditions for the strict positivity of $W$ near infinity in $\Gw$. Recall that the unique continuation property holds true for a second-order elliptic equation $Pu=0$ in $\Gw$ if the coefficients of $P$ are smooth enough (see for example \cite{Kenig}). Since $G/u$ is a positive solution of a second-order elliptic equation, and the zero set of the above optimal Hardy-weight $W$ is equal to the zero set of $|\nabla (G/u)|$, it follows that under appropriate smoothness assumptions, the zero set of $W$ has an empty interior. Moreover, if a level set $\Gg$ of $G/u$ is smooth enough, then by Hopf lemma, $|\nabla (G/u)|\neq 0$ on $\Gg$. For results concerning the set of critical points of Green functions on complete manifolds see \cite{EP} and the references therein.
   }
\end{remark}

Recall that  for any $\gl>1$,  the function
\begin{equation}\label{eq_phixi}
\varphi_\xi:=\varphi(\xi,x)= u\,\left(\frac{G}{u}\right)^{1/2}\exp(\mathrm{i}\xi \log(G/u))
\end{equation}
with $\xi=\pm \sqrt{\lambda-1}/2$  solves the equation
$$(P-\lambda W)u=\big(P-(1+4\xi^2)W\big)u=0 \qquad \mbox{in } \Gw^\star.$$
So, for any $\gl> 1$  the equation $(P-\lambda W)u=0$ admits (at least two) ``non-growing" generalized eigenfunctions.
Therefore, \v{S}nol's principle (or Bloch-type property) suggests that the spectrum $\gs$ and the essential spectrum $\gs_{\mathrm{ess}}$ of $W^{-1}P$ in $L^2(\Gw^\star,W\mathrm{d} \nu)$ is equal to $[1,\infty)$.
In fact, for  such an operator $P$, we find an invariant subspace ``spanned" by the functions $\varphi_\xi$ on which $P$ has a canonical form with purely absolutely continuous spectrum that is equal to $[1,\infty)$.

Define $\mathcal{U}_{\mathrm{rad}}(\Omega^\star)$ to be the space of measurable functions that are proportional to $u$ on the level sets of $G/u$, and denote by $L^2_{\mathrm{rad}}(\Omega^\star,W \mathrm{d}\nu)$ the space  $L^2(\Omega^\star,W \mathrm{d}\nu)\cap \mathcal{U}_{\mathrm{rad}}(\Omega^\star)$.
Explicitly, $v \in \mathcal{U}_{\mathrm{rad}}(\Omega^\star)$ if and only if  $v = u f(G/u)$ for some measurable function $f : (0,\,\infty) \to \mathbb{C}$.

%%%%%%%%%%%%%%%%%%%%%%
\begin{lemma}
\label{radspace}
Under the normalization $u(0)=1$, the map
\begin{align}
\label{eq_isoA}
 L^2_{\mathrm{rad}}\left(\Omega^\star, W \mathrm{d}\nu\right)  \quad &\rightarrow \quad  L^2\left((0,\,\infty),\frac{1}{4 t^2} \dt\right), \nonumber \\
 v = uf (G/u) \quad & \mapsto \quad f(t),
 \end{align}
 is an isometry.
\end{lemma}

\begin{proof}
 Assume first that $P$ has smooth coefficients. Then by Sard's lemma, almost every point $t\in \R_+$ is a regular value of the function $G/u$, and hence for such points $t$,  the set $\{G/u=t\}$ is a smooth $(n-1)$-dimensional submanifold. Note also that the function $W|\grad (G/u)|^{-1}$ is smooth in $\Gw^\star$ (see the computation below).

On the other hand, by Green's formula, for any smooth neighborhood $\tilde\Omega$ of $0$, we have
\begin{equation}
\label{ConstGamma}
 \int_{\pd{\tilde\Omega}} \left\langle u A\grad G - G A\grad u,\vec{\gs} \right\rangle \mathrm{d} \sigma = \gamma,
\end{equation}
where $\gamma=u(0)=1$.

Consequently, the coarea formula and \eqref{ConstGamma} imply that for any two functions in $L^2_{\mathrm{rad}}(\Omega^\star,W \mathrm{d}\nu)$ we have
\begin{multline}\label{eq_coarea}
\int_{\Omega^\star} u f\left(\frac{G}{u}\right) u g^*\left(\frac{G}{u}\right) W \dnu \\
= \int_{\Omega^\star} u f\left(\frac{G}{u}\right) u g^*\left(\frac{G}{u}\right) \frac{W}{|\grad (G/u)|_A} |\grad (G/u)|_A\dnu \\
 =\int_0^\infty \mathrm{d}t \int_{\{G/u = t\}} f(t) g^*(t) u^2 \frac{1}{4 t^2}
 \left\langle A\grad \left(\frac{G}{u} \right),\vec{\gs}\right\rangle \mathrm{d} \sigma \\
= \int_0^\infty \mathrm{d}t f(t) g^*(t) \frac{1}{4 t^2} \int_{\{G/u = t\}} \Big\langle \left(u A\grad G - GA \grad u \right) ,\vec{\gs}\Big\rangle \,\mathrm{d} \sigma \\
= \int_0^\infty f(t) g^*(t) \frac{1}{4 t^2} \,\mathrm{d} t,
\end{multline}
where in passing from the second line to the third line of \eqref{eq_coarea} we used the coarea formula, and that $\grad (G/u)$ is parallel (in the metric $|\cdot|_A)$ to the normal vector $\vec{\sigma}$ of the level set $\{G/u = t\}$, and therefore,
$$ \frac{W}{|\grad (G/u)|_A}=\frac{1}{4 \big(G/u\big)^2}|\grad (G/u)|_A=\frac{\langle A \grad(G/u),\, \vec{\sigma}\rangle}{4 t^2}\, .$$
Hence, in the smooth case we have the isometry
\begin{equation}\label{eq_iso}
\int_{\Omega^\star} u f\left(\frac{G}{u}\right) u g^*\left(\frac{G}{u}\right) W \dnu = \int_0^\infty f(t) g^*(t) \frac{1}{4 t^2} \,\mathrm{d} t.
 \end{equation}
The regular case is obtained by a standard approximation argument (note that one may assume that $u=\mathbf{1}$).
\end{proof}
In the sequel of the present section, we assume that the positive solution $u$ is \textit{normalized} so that $u(0)=1$.

\vskip 3mm

Before proceeding with the study of the essential spectrum we note that the proof of Lemma~\ref{radspace} implies the  following corollary, which allows us to estimate in average the potential $W$, and provides (in the symmetric case) an alternative proof of the null-criticality of the operator $P-W$ near $0$ and $\infty$.

\begin{Cor}\label{Estimate W}
Suppose that the hypotheses of Theorem \ref{thm*} are satisfied,  and that $u(0)=1$. Then for any $0<a<b$ and $\xi\in\R$ we have
\begin{equation}\label{eq_coar}
\int_{\{a \leq \frac{G}{u}  \leq b\}} uGW\, \mathrm{d} \nu=\int_{\{a \leq \frac{G}{u}  \leq b\}} |\varphi_\xi|^2 W \mathrm{d} \nu= \frac{1}{4}(\log b- \log a).
\end{equation}
\end{Cor}
\begin{proof}
As in \eqref{eq_coarea}, we use the coarea formula on the domain $\{a \leq \frac{G}{u}  \leq b\}$ (instead of the domain $\Gw^\star$) with the functions $f(x)=x$ and $g(x)=\mathbf{1}$, to obtain
\begin{equation*}
\int_{\{a \leq \frac{G}{u}  \leq b\}} uGW \mathrm{d} \nu
              							= \frac{1}{4}\int_a^b t^{-1} \mathrm{d} t = \frac{1}{4}\left( \log b - \log a \right).
\end{equation*}
\end{proof}

\begin{theorem}\label{essential_spectrum}
Suppose that the hypotheses of Theorem \ref{thm*} are satisfied, and $W>0$  in $\Gw^\star$. Then the spectrum $\gs$ and the essential spectrum $\gs_{\mathrm{ess}}$ of (the Friedrichs extension of) $\tilde{P}:=W^{-1}P$ acting on $L^2(\Omega^\star,W\mathrm{d}\nu)$ satisfy  $$\gs(\tilde{P},\Gw^\star)=\gs_{\mathrm{ess}}(\tilde{P},\Gw^\star)=[1,\infty).$$

In fact, the spectrum of $\tilde{P}$ restricted to $L^2_{\mathrm{rad}}(\Omega^\star,W \mathrm{d} \nu)$ is purely absolutely continuous with respect to the Lebesgue measure.

Moreover, for any neighborhood $U\subset \Gw^\star$ of $0$ or infinity of $\Gw$, the (essential) spectrum of the Friedrichs extension of the operator $\tilde{P}$ on $L^2(U,W\mathrm{d}\nu)$ satisfies $$\gs(\tilde{P},U)=\gs_{\mathrm{ess}}(\tilde{P},U)=[1,\infty).$$
\end{theorem}

\begin{proof}
Using formulas (\ref{rule1}) and (\ref{rule2}) we find that
\begin{equation}\label{eq_doubleprime}
\frac{1}{W}P \left(u f(G/u) \right) = -4 u f''(G/u) \left(\frac{G}{u}\right)^2.
\end{equation}
This proves that $L^2_{\mathrm{rad}}\left(\Omega^\star,W \mathrm{d}\nu\right)$ is an invariant subspace of $\tilde{P}$, and the operator restricted to this subspace is unitarily equivalent to the symmetric operator
$$
D :  L^2\left((0,\,\infty),\frac{1}{4 t^2} \dt\right) \to L^2\left((0,\,\infty),\frac{1}{4 t^2} \dt\right)
$$
defined by
\begin{equation}\label{eq_D}
 \quad (Df)(t) := -4t^2f''(t).
\end{equation}
The spectral representation of $D$, in terms of the Mellin transform (with $n=1$), has been derived in Section~\ref{short_pf}  (see in particular, \eqref{eq_spect_rep}). More explicitly, it is the composition of the Mellin transform with the isometry from $L^2\left((0,\infty),\frac{1}{4 t^2}\mathrm{d}t\right)$ to $L^2\left((0,\infty),\mathrm{d}t\right)$, which is given by
\begin{equation}\label{L2L2}
f(t) \mapsto \frac{1}{2} f\left(\frac{1}{t}\right).
\end{equation}
It follows
\begin{equation}\label{eq_spD}
\sigma(D,(0,\infty)) = \sigma_{\mathrm{ac}}(D,(0,\infty)) = [1,\,\infty).
\end{equation}
Recall that by theorems~\ref{thm_Sch} and \ref{thm_Sch7} we have $$\gs(\tilde{P},\Gw^\star)=\gs_{\mathrm{ess}}(\tilde{P},\Gw^\star)\subset [1,\infty).$$ Therefore, \eqref{eq_spD} implies that
$$\gs(\tilde{P},\Gw^\star)=\gs_{\mathrm{ess}}(\tilde{P},\Gw^\star)= [1,\infty).$$

\vskip 3mm

It remains to explain why we can localize the spectral result at a neighborhood $U\subset \Gw^\star$ of either $0$ or infinity of $\Gw$.

It is not difficult to check using the above results that $\tilde{P}$ on $L^2_{\mathrm{rad}}\left(\Omega^\star,W \mathrm{d}\nu\right)$ is unitarily equivalent to the operator
$$\tilde{D}f=-4(t^2f')' \qquad \mbox{defined on } L^2\left((0,\infty),\mathrm{d}t\right).$$
 Moreover, a neighborhood of $0$ (resp. of $\infty$) in $\Omega^\star$ corresponds to a neighborhood of $0$ (resp. of $\infty$) in $(0,\infty)$.

 Therefore, it is enough to prove that the essential spectrum of $\tilde{D}$ restricted to a neighborhood of $0$ or $\infty$ in $(0,\infty)$ is $[1,\infty)$. First, we know that the essential spectrum is preserved under compactly supported perturbation, and this implies that $\gs_{\mathrm{ess}}(\tilde{D},(0,\infty))$ is equal to the union of $\gs_{\mathrm{ess}}(\tilde{D},U_0)$ and $\gs_{\mathrm{ess}}(\tilde{D},U_\infty)$, where $U_0$ (resp. $U_\infty$) is any neighborhood of $0$ (resp. $\infty$) in $(0,\infty)$. Let $U_0$ be a neighborhood of $0$, and define $U_\infty$ to be the neighborhood of $\infty$ obtained from $U_0$ by the transformation $t\mapsto\frac{1}{t}$. Consider the following isometry $T$ between $L^2(U_0,\mathrm{d}t)$ and $L^2(U_\infty,\mathrm{d}t)$ given by

$$Tf(t)=\frac{1}{t}f\left(\frac{1}{t}\right).$$
A computation shows that

$$T\tilde{D}=\tilde{D}T,$$
and this implies that the essential spectrum of $\tilde{P}$ restricted to $U_0$ is equal to the essential spectrum of $\tilde{P}$ restricted to $U_\infty$. Since the union of these two essential spectra is $[1,\infty)$, we get that each one is equal to $[1,\infty)$.
\end{proof}
\begin{remark}\label{rem_another_pr}{\em
The latter assertion of Theorem~\ref{essential_spectrum} provides us with an alternative proof (in the symmetric case) that $\lambda_\infty(P,W,\Omega^\star)=1$.
 }
 \end{remark}

Collecting the transformations \eqref{eq_isoA},\eqref{eq_D}, and \eqref{L2L2}, we obtain a spectral representation of $\tilde{P}=W^{-1} P$ restricted to $L^2_{\mathrm{rad}}\left(\Omega^\star,W \mathrm{d}\nu\right)$.
%%%%%%%%%%%%%%%%

\begin{Cor}
The operator $\mathcal{F}$ given by
\begin{equation}\label{trans}
\mathcal{F}f(\xi):=\sqrt{\frac{2}{\pi}}\int_{\Omega^\star}f(x)\varphi(\xi,x)W(x)\mathrm{d}\nu(x)\qquad \xi\in \R,
\end{equation}
(where $\varphi(\xi,x)$ is defined by \eqref{eq_phixi}) is a well defined unitary operator from $L_\mathrm{rad}^2(\Gw^\star,W\mathrm{d}\nu)$ onto $L^2(\R,\dxi)$, whose inverse is given by
$$\mathcal{F}^{-1}g(x)= \sqrt{\frac{2}{\pi}}\int_{\R}g(\xi)\varphi(-\xi,x)\,\mathrm{d}\xi.$$
Furthermore,
$$ \mathcal{F} \frac{1}{W} P \mathcal{F}^{-1} f(\xi) = (1 + 4 \xi^2) f(\xi).$$
\end{Cor}
%%%%%%%%%%%%%%%%%%%%%%%%%%%%%%%%%%%%
%%%%%%%%%%%%%%%%%%

\begin{remark}\label{rem_rad}{\em
Formula \eqref{eq_doubleprime} is valid also in the nonsymmetric case. So, the operator $\tilde{P}=W^{-1}P$ restricted to ``radial" functions (i.e., functions in  $\mathcal{U}_{\mathrm{rad}}(\Omega^\star)\cap\mathcal{V}$) is in fact a Fuchsian-type ordinary differential operator. In particular, under the assumptions of Theorem~\ref{thm_nsa}, for any $\gl<1$ the set  of all ``radial" positive solutions of the equation $(P-\gl W)u=0$ in $\Gw^\star$  (i.e., the set  $\mathcal{C}_{\tilde{P}-\gl}(\Omega^\star)\cap \mathcal{U}_{\mathrm{rad}}(\Omega^\star)$) is a two dimensional cone, while by Theorem~\ref{thm_Sch}, the entire cone $\mathcal{C}_{\tilde{P}-1}(\Omega^\star)=\mathcal{C}_{P-W}(\Omega^\star)$ is a singleton.
}
\end{remark}

\vskip 3mm

We provide below a more detailed and explicit construction of the above transform $\mathcal{F}$ using methods related to classical Fourier transform.  This also gives independent proof of Theorem~\ref{essential_spectrum}.

\paragraph {\em Alternative proof of Theorem~\ref{essential_spectrum}:}
The idea is to find a spectral representation of $\tilde{P}$ restricted to $L^2_\mathrm{rad}(\Omega^\star,W\mathrm{d}\nu)$, that is a unitary operator

$$U : L^2_\mathrm{rad}(\Omega^\star,W\mathrm{d}\nu) \mapsto L^2(\R)$$
such that $U\tilde{P}U^{-1}$ is the multiplication by a real function with values in $[1,+\infty)$. Since the ground state transform is unitary, we may assume that $u=\mathbf{1}$.
For the sake of brevity, we will denote $\mathcal{F}f(\xi)$ by $\hat{f}(\xi)$. We thus have to prove that for every $f\in C_0^\infty(\Omega^\star)$ which is constant on the level sets of $G$, the following  two identities hold
\begin{equation}\label{parseval}
\int_{\Omega^\star}|f|^2W\mathrm{d}\nu=\int_{\R}|\hat{f}|^2\,\mathrm{d}\xi \quad \mbox{(Plancherel-type formula)}
\end{equation}
 and
\begin{equation}\label{inversion}
f(x)=\sqrt{\frac{2}{\pi}}\int_{\R}\hat{f}(\xi)\varphi(-\xi,x)\,\mathrm{d}\xi\qquad \forall x\in \Gw^\star \quad \mbox{(the inversion formula)}.
\end{equation}
For a fixed  $r>0$, we define $\Omega(r)$ to be the open, relatively compact set
$$\Omega(r):=\{-r\pi<\log(G)<r\pi\},$$
and for any $k\in \mathbb{Z}$, we denote

$$\varphi^r_k(x):=\varphi\left(\frac{k}{r},x\right)=\sqrt{G}\exp\left(\mathrm{i}\frac{k}{r} \log(G)\right) \qquad x\in \Gw(r).$$
Consider the ``torus" $\mathbf{T}_r$ to be the closure of $\Omega(r)$ divided by the equivalence relation

$$x\equiv y\Leftrightarrow \log G (x)=\log G (y) \mod \left(2\pi r \right).$$

The set of complex valued continues functions $C(\mathbf{T}_r;\mathbb{C})$ can be identified to the set of complex valued continuous functions on the closure of $\Omega(r)$,  each of which is constant on the level sets of $G$, and its value on the set $\{\log G=-\pi r\}$ is equal to its value on the set $\{\log G=\pi r\}$. In particular, for every $k\in \mathbb{Z}$, we have $\exp(\mathrm{i}\frac{k}{r}\log G)\in C(\mathbf{T}_r;\mathbb{C})$. We also define the space $L_\mathrm{rad}^2(\mathbf{T}_r;\mathbb{C})$, with the induced measure from $\Omega_r$. We want to decompose the elements of $L_\mathrm{rad}^2(\Omega(r);\mathbb{C})$ in ``Fourier series" with respect to the family $\{\varphi^r_k\}_{k\in \mathbb{Z}}$. First, we check the orthonormality.

\begin{Lem}\label{lem_orth}
For any $r>0$ it holds

$$\frac{2}{ \pi r}\int_{\Omega(r)}\varphi^r_k\overline{\varphi^r_l}W\mathrm{d}\nu=\delta_{k,l} \qquad \forall\,  k, l\in\mathbb{Z}.$$

\end{Lem}
\begin{proof}[Proof of Lemma~\ref{lem_orth}]
Notice that $\overline{\varphi^r_l}=\varphi_{-l}^r$. If $k\neq l$ and $k\neq-l$, then $\varphi_k^r$ and $\varphi_{l}^r$ are generalized eigenfunctions of $P$ with different associated eigenvalues, and to prove their orthogonality we need to establish the identity

$$\int_{\Omega(r)}\left(P[\varphi_k^r]\varphi_l^r-\varphi_k^rP[\varphi_k^r]\right)\dnu=0.$$
To this end, we have to check that the boundary term in the corresponding Green formula is zero. This boundary term is
given by
$$B.T.:=\int_{\partial\Omega(r)}\left\langle A\nabla[\varphi_k^r] \varphi_l^r- A\nabla[\varphi_l^r]\varphi_k^r ,\vec{\gs}\right\rangle\,\mathrm{d}\sigma.$$
We compute
$$\nabla\varphi(\xi,\cdot)=\exp(\mathrm{i}\xi\log(G))\left(\nabla G^{1/2}\right)+\mathrm{i}\xi G^{-1/2}\exp(\mathrm{i}\xi\log(G))\nabla G.$$
Since $\exp(\mathrm{i}\frac{k}{r}\log(G))$ and $\exp(\mathrm{i}\frac{l}{r}\log(G))$ are constant (equal to $(-1)^k$ and $(-1)^l$ respectively) on $\partial \Omega(r)$, we have

$$B.T.=\mathrm{i}(-1)^{k+l}\frac{(k-l)}{r}\int_{\partial\Omega(r)}\left\langle A\nabla[ G],\vec{\gs}\right\rangle \,\mathrm{d}\sigma.$$
On the other hand, applying the Green formula on the pair $(1,G)$, we obtain
$$\int_{\Omega(r)}P[G]\mathbf{1}\dnu-\int_{\Omega(r)}GP[\mathbf{1}]\dnu=\int_{\partial\Omega(r)}\left\langle A\nabla[G]\mathbf{1}- A\nabla[\mathbf{1}] G,\vec{\gs}\right\rangle\,\mathrm{d}\sigma,$$
and recalling that we assumed that $P\mathbf{1}=0,$ and that also $PG=0$ on $\Omega(r)$, we get

$$\int_{\partial\Omega(r)} \left\langle A\nabla[ G],\vec{\gs}\right\rangle  \,\mathrm{d}\sigma=0,$$
and thus $B.T.=0$.

If $k\in \mathbb{Z}$ and $l=-k\neq 0$, then $\varphi_k^r\overline{\varphi_{-k}^r}=\varphi_{2k}^r\varphi_0$ and the orthogonality of $\varphi_{2k}$ and $\varphi_0$ have been already established.
On the other hand, for $k\in \mathbb{Z}$, and $l=k$, we have

$$\int_{\Omega(r)}|\varphi^r_k|^2W\mathrm{d}\nu =\int_{\Omega(r)}GW\mathrm{d}\nu,$$
and the integral is equal to $\pi r/2$ according to Corollary~\ref{Estimate W}.
\end{proof}

\paragraph {\em Continuation of the alternative proof of Theorem~\ref{essential_spectrum}:} Since $\mathbf{T}_r$ is compact, the Stone-Weierstrass theorem implies that the vector space generated by the sequence $\big\{\exp(\mathrm{i}k/r\log(G))\big\}_{k\in \mathbb{Z}}$ is dense in $C(\mathbf{T}_r;\mathbb{C})$ (in the topology of uniform convergence). Therefore, the orthonormal  series $\left\{(\pi r / 2)^{-1/2}\varphi_k^r\right\}_{k\in \mathbb{Z}}$ is complete  in $L_\mathrm{rad}^2(\Omega(r);\mathbb{C})$. Consequently, by Parseval's equalities, the following discrete analogues of \eqref{parseval} and \eqref{inversion} are available for every $f\in L_\mathrm{rad}^2(\Omega(r);\mathbb{C})$:
\begin{equation}\label{parseval_discret}
\int_{\mathbf{T}_r}|f|^2W\mathrm{d}\nu=
\frac{1}{r}\sum_{k\in\mathbb{Z}}\left|\hat{f}\left(\frac{k}{r}\right)\right|^2
\end{equation}
and

\begin{equation}\label{inversion1}
f(x)=\frac{1}{r} \sqrt{\frac{2}{\pi}}\sum_{k\in\mathbb{Z}}\hat{f}\left(\frac{k}{r}\right)\varphi\left(-\frac{k}{r},x\right)
\end{equation}

Fix now $f\in C_0^\infty(\Omega^\star)\cap L_\mathrm{rad}^2(\Omega^\star,W\mathrm{d}\,\nu)$, and choose $r>0$ such that the support of $f$ is included in $\Omega(r)$ (this is possible since the fact that $G$ tends to $0$ at infinity implies that $\{\Omega(r)\}_{r>0}$ is an exhaustion of $\Omega^\star$). Let us apply \eqref{parseval_discret} and \eqref{inversion1} to the function $g:=\exp(\mathrm{i}\alpha \log(G))f$, for $\alpha\in (0,1/r)$: we get

$$\int_{\mathbf{T}_r}|f|^2W\mathrm{d}\nu=\frac{1}{r}\sum_{k\in\mathbb{Z}}\left|\hat{f}\left(\frac{k}{r}+\alpha\right)\right|^2\,,$$
and

$$f=\frac{1}{r} \sqrt{\frac{2}{\pi}}\sum_{k\in\mathbb{Z}}\hat{f}\left(\frac{k}{r}+\alpha\right)\varphi\left(-\frac{k}{r}-\alpha,\cdot\right).$$
We integrate these two equalities with respect to $\alpha\in(0,1/r)$: recalling that $f$ has support in $\Omega(r)$, we obtain

$$\int_{\Omega^\star}|f|^2W\mathrm{d}\nu= \int_{\R}|\hat{f}|^2\dxi,$$
and
$$f(x)= \sqrt{\frac{2}{\pi}} \int_{\R}\hat{f}(\xi)\varphi(-\xi,x)\,\mathrm{d}\xi .$$
This is exactly \eqref{parseval} and \eqref{inversion}.  \qed

We conclude this section with the following conjecture that arises naturally from our study.
\begin{conjecture}\label{conj_gen_ef}
Assume that $\gf_0$ is a ground state of a symmetric critical operator $P$ in $\Gw$. For $\gl>0$ let $\gf_\gl$ be a solution of the equation $(P-\gl)u=0$ in $\Gw$ satisfying
\begin{equation}\label{eq_gen_ef}
|\gf_\gl(x)|\leq C\gf_0(x)  \qquad \forall x\in \Gw,
\end{equation}
where $C>0$ is a constant. Then $\gl$ belongs to the spectrum of the Friedrichs extension of the operator $P$ on $L^2(\Gw, \dx)$.
\end{conjecture}

\begin{Rem}
{\em
(Added after the paper was accepted) The results concerning the (essential) spectrum of $\frac{1}{W}P$ have recently been extended to some non-optimal potentials $W$ (see \cite{Dev}).
}

\end{Rem}

\mysection{Completeness of the induced Agmon metric and Rellich-type inequalities}\label{sec_agmon_m}

\subsection{Completeness of the induced metric}\label{sec-compl}
In this subsection we prove that the Agmon metric corresponding to optimal Hardy-weight $W$ in $\Omega^\star$ is {\em complete}. The completeness of $\Gw^\star$ in this metric implies sharp decay estimates for solutions of the equations the $Pu=f$ in $\Gw^\star$ (see Subsection~\ref{sec_decay}).
\begin{lemma}\label{lem_complete}
Suppose that the assumptions of Theorem~\ref{thm_nsa} are satisfied and let $W$ be the corresponding optimal Hardy-weight. Assume further that $W$ is strictly positive. Then $\Gw^\star$ is complete in the Agmon (Riemannian) metric
\begin{equation}\label{eq_Ag_metric}
\qquad \ds^2:= W(x)\sum_{i,j=1}^{n} a_{ij}(x)\dx^i\dx^j, \quad \mbox{where } \big[a_{ij}\big]:= \big[a^{ij}\big]^{-1}.
\end{equation}
\end{lemma}
\begin{proof}
We follow the proof of Lemma 1.5 in \cite{LW}. Let $\gamma$ be a curve in $\Omega^\star$ such that $\gamma(t)\to\infty$ in $\Omega^\star$ when $t\to T$. Here, $T$ is finite or infinite. We have to show that the length $L(\gamma)$ of $\gamma$ for the metric $\mathrm{d}s^2$ is infinite. Denoting $v:=\frac{G}{u}$, we compute

\begin{multline}
L(\gamma)=\int_0^T\!\!\!\sqrt{W(\gamma(s))}|\gamma'|_{A^{-1}}\ds
=\dfrac{1}{2}\int_0^T\!\!\!|\nabla \log v|_A(\gamma(s))|\gamma'|_{A^{-1}}\ds.
\end{multline}
Define $\nabla_A$ to be the gradient with respect to the metric $|\cdot|_A$. For a function $f$ and a vector $v\in T_x\Omega$, by definition of the gradient, we have the following identity
$$\mathrm{d}f_x(v)=\langle \nabla f,v\rangle=\langle A\nabla_Af,v\rangle,$$
which shows that $\nabla_A=A^{-1}\nabla$. From this, we see that
\begin{align*}
|\nabla_{A^{-1}}f|_{A^{-1}}^2=\langle A^{-1}\nabla_{A^{-1}}f,\nabla_{A^{-1}}f\rangle=
\langle A\nabla f,\nabla f\rangle=|\nabla f|_A^2.
\end{align*}
Using this last identity, we get
\begin{align*}
L(\gamma)&=\frac{1}{2}\int_0^T|\nabla_{A^{-1}}\log v|_{A^{-1}}(\gamma(s))|\gamma'(s)|_{A^{-1}}\ds
\geq \frac{1}{2}\int_0^T\left|\frac{\mathrm{d}}{\ds}\log v(\gamma(s)) \right|\ds\\[3mm]
&\geq\frac{1}{2}\left|\int_0^T\frac{\mathrm{d}}{\ds}\Big(\log v\big(\gamma(s)\big)\Big)\ds\right|
=\frac{1}{2}\lim_{t\to T}|\log v(\gamma(t))-\log v(\gamma(0))|.
\end{align*}
Since $\gamma(t)\to\infty$ in $\Gw^\star$ as $t\to T$, and  $\lim_{x\to \infty} |\log v(x)|=\infty$, we deduce that $L(\gamma)=\infty$.
\end{proof}

%%%%%%%%%%%%%%%%%%%%%%%%%%%%%%%%%%%%%%%
\subsection{Decay of solutions of $Pu=f$ and Rellich-type inequality}\label{sec_decay}
%%%%%%%%%%%%%%%%
Let $P$ be a Schr\"odinger operator of the form
\be \label{Sch}
Pu=-\sum_{i,j=1}^{n}
\partial_{i}\big(a^{ij}(x)\partial_{j}u\big)+c(x)u
\end{equation}
defined on a domain $\Gw\subset \mathbb{R}^n$. A theorem of Agmon
\cite[Theorem~1.5]{Ag82} states that under certain conditions on $P$, solutions $u$ of the equation
$Pu=f$ in $\Gw$ that do not grow too fast, in fact, decay rapidly. The main condition which is required for the validity of the theorem is given by
\begin{equation}\label{Ag_cond}
(P\gf,\gf)\geq \int_{\Gw} \gl(x)|\gf|^2 \dx
\qquad \forall \gf\in C_0^\infty(\Gw),
\end{equation}
where $\gl$ is a nonnegative weight function. The decay is then given in terms of a function $h$ satisfying
\begin{equation}\label{h_cond}
    |\nabla h(x)|^2_A< \gl(x) \qquad \mbox{ a.e. } \Gw.
\end{equation}
Any Hardy-weight $W$ given by \eqref{W} provides us with a natural candidate for $\gl$ and $h$. Assume that our Hardy-weight $W$ obtained by the supersolution construction with a pair $(v_0,v_1)$ is strictly positive a.e, in $\Gw$, and set
$$\gl:=W, \qquad h :=\frac{\gm}{2}\log\left(\frac{v_0}{v_1}\right),$$ where $0<\gm<1$.
Then $\gl$ and $h$  clearly satisfy \eqref{h_cond}.
Suppose also that a solution $u$ of $Pu=f$ in $\Gw$ satisfies the growth condition (1.13) in \cite{Ag82}. By Lemma~\ref{lem_complete} the induced Riemannian metric
\begin{equation}\label{eq_metric}
\qquad \ds^2:= W(x)\sum_{i,j=1}^{n} a_{ij}(x)\dx^i\dx^j, \quad \mbox{where } \big[a_{ij}\big]:= \big[a^{ij}\big]^{-1}
\end{equation}
is complete. Therefore, by \cite[Theorem~1.5]{Ag82}, the following \textit{Rellich-type inequality} holds true
\begin{equation}\label{exp}
 \left(1-\gm^2\right)^2\int_{\Gw} |u|^2W(x)\left(\frac{v_0}{v_1}\right)^\gm \dx\leq
  \int_{\Gw} \frac{|Pu|^2}{W(x)}\left(\frac{v_0}{v_1}\right)^\gm \dx\,.
\end{equation}
Assume that for some $0<\gm<1$ we have
$$\int_{\Gw} \frac{|Pu|^2}{W(x)}\left(\frac{v_0}{v_1}\right)^\gm \dx<\infty.$$
Then letting $\gm\to0$ (using the monotone and dominated convergence theorems) we obtain  the following \textit{Rellich-type inequality}:
\begin{equation}\label{exp1}
 \int_{\Gw} |u|^2W(x) \dx\leq
  \int_{\Gw} \frac{|Pu|^2}{W(x)} \dx\,.
\end{equation}
That a Rellich-type inequality follows \textit{via} Agmon's theory from a Hardy inequality was already observed in \cite{Grillo}.
\begin{remark}\label{remRell}{\em
One can obtain the above inequalities (\eqref{exp} and \eqref{exp1}) for functions $u\in C_0^\infty(\Gw)$ (for a {\em general} subcritical symmetric operator $P$) using only the supersolution construction
and the associated Hardy inequality.

Indeed, without loss of generality assume that $P\mathbf{1} = 0$. Then using \eqref{rule1} and  \eqref{rule2} it follows that for any two smooth enough functions $u$ and $v$
with $u\in C_0^\infty(\Gw)$ we have
\begin{equation}\label{Pvu}
(Pu,uv^2) = (P(vu),vu) + \frac{1}{2}(u^2,Pv^2) - ( Pv,u^2v).
\end{equation}

Now let $w$ be a positive solution of the equation $Pw=0$ in $\Gw$, and let $0\leq \gm \leq 1$. We use \eqref{Pvu}
with the pair $u\in C_0^\infty(\Gw)$ and $v=w^{\mu/2}$, recalling that $$Pw^\mu = 4 \mu(1-\mu) W w^\mu, \qquad \mbox{ where } W:=\frac{Pw^{1/2}}{w^{1/2}}.$$ It follows that
\begin{align}
(Pu, u w^{\mu}) &= (P(w^{\mu/2} u), w^{\mu/2} u) +[2 \mu (1-\mu) - \mu(2-\mu)](u^2,Ww^\mu) \nonumber\\
			&\geq \left(1-\gm^2\right)\int_\Gw u^2Ww^\mu\dx, \label{Pvu1}
\end{align}
where we used the Hardy inequality $P-W\geq 0$ to derive the second line.
Assume now that $W>0$ in $\Gw$, then the Cauchy-Schwarz inequality implies the following Rellich-type inequality
$$
 \left(1-\gm^2\right)^2\int_\Gw u^2Ww^\mu \dx \leq \int_\Gw (Pu)^2 \frac{1}{W} w^\mu\dx \qquad \forall u\in C_0^\infty(\Gw).
$$
Therefore, for a general symmetric, subcritical operator $P$, and a positive Hardy-weight $W$ obtained
by the supersolution construction with a pair $(v_0,v_1)$ of two positive solutions,  we obtain for $0\leq \gm \leq 1$ that
\begin{equation}\label{exp9}
 \left(1-\gm^2\right)^2\!\!\int_{\Gw} |u|^2W(x)\!\left(\frac{v_0}{v_1}\right)^\gm\!\!\! \dx\leq
  \!\int_{\Gw} \!\frac{|Pu|^2}{W(x)}\!\left(\frac{v_0}{v_1}\right)^\gm \!\!\!\dx \;\; \forall u\in C_0^\infty(\Gw).
\end{equation}
Moreover, using an approximation argument, it follows that if  $P- W\geq 0$ is critical in $\Gw$, and $0\leq \gm < 1$, then $(1-\gm^2)^2$ is the {\em best constant} for the inequality \eqref{exp9}.
 }
\end{remark}
We summarize these results in the following corollary.
\begin{corollary}\label{cor1_hardy_rell}
Assume that $P$ is a symmetric subcritical operator in $\Gw$, and let $W>0$ be a Hardy-weight obtained
by the supersolution construction with a pair $(v_0,v_1)$ of two positive solutions $v_0$ and $v_1$ of the equation $Pu=0$ in $\Gw$.
Fix $0\leq \gl\leq 1$. Then
\begin{itemize}
   \item[(a)] For fixed  $0\leq \gm<  1$ and all $u\!\in\! C_0^\infty(\Gw)$ the following Rellich-type inequality holds true
    \begin{equation}\label{exp8}
\gl\left(1- \gm^2\right)^2  \int_{\Gw}  |u|^2W(x)\left(\frac{v_0}{v_1}\right)^\gm \! \\dx \leq
  \int_{\Gw} \frac{|Pu|^2}{W(x)}\left(\frac{v_0}{v_1}\right)^\gm \!\dx .
\end{equation}

\item[(b)]  For any $0\leq \ga\leq 1$ and all $ u\!\in \!C_0^\infty(\Gw)$ the following Hardy-Rellich-type inequality holds true
    \begin{equation}\label{exp13}
\gl \int_{\Gw} |u|^2W(x) \dx \leq  \ga\int_{\Gw} uP[u] \dx +(1- \ga)
  \int_{\Gw}  \frac{|Pu|^2}{W(x)} \dx .
\end{equation}

\item[(c)] If $P-W$ is critical in $\Gw$,  then  $\gl=1$ is the best constant in inequalities \eqref{exp8} and \eqref{exp13}.
\end{itemize}
\end{corollary}
\begin{example}\label{ex_rell}{\em
Consider the Poisson equation in the punctured space $\Gw^\star=\mathbb{R}^n\sm \{0\}$, $n\geq 3$ with the {\em optimal} Hardy-weight $$W(x):= \left(\frac{n-2}{2}\right)^2|x|^{-2}.$$
The corresponding induced Riemannian metric is given by
$$\ds^2:= W(x)\sum_{i=1}^{n} (\dx^i)^2.$$
By Lemma~\ref{lem_complete}, $\Gw^\star$ is complete in the above Agmon metric.
By  \eqref{exp}, \eqref{exp1}, and \eqref{exp9}, for any $0\leq \gm <1$   the following Rellich-type inequality (with the best constant) holds true
\begin{equation}\label{exp2}
 \left(\frac{n\!-\!2}{2}\right)^4\!\!\!\left(1\!-\!\gm^2\right)^2\!\!\!\!\int_{\Gw^\star} \!\!\frac{|u(x)|^2}{|x|^{2 +(n-2)\gm}}\dx\!\leq\!
  \int_{\Gw^\star}\!\! |\Gd u|^2|x|^{2-(n-2)\gm} \dx\;\;\: \forall u\in C_0^\infty(\Gw^\star) .
\end{equation}
In fact, it is known that $\left(\frac{n-2}{2}\right)^4\left(1-\gm^2\right)^2$ is indeed the {\em best constant} for the above inequality, see \cite[Theorem~3.14 and the references therein]{GM1}. Note also that the choice $\mu = 2/(n-2)$ recovers the classical Rellich inequality:
$$\frac{n^2(n-4)^2}{16}\int_{\Gw^\star} \!\!\frac{|u(x)|^2}{|x|^{4}}\dx\!\leq\!
  \int_{\Gw^\star}\!\! |\Gd u|^2 \dx\;\;\:, \forall u\in C_0^\infty(\Gw^\star) .$$
    }
\end{example}
%%%%%%%%%%%%%%%%%%%%%%%%%%%%%%%%%%%%%%%%%%%%%%%%%%%%%
\mysection{Boundary singularities}\label{sec_bound_sing}
In the present section we explain how our results can be extended to the case of boundary singularities, where the singularities of the Hardy-weight are located at $\pd\Omega\cup \{\infty\}$ and not at an isolated interior point of $\Omega$ as above. So, we apply the supersolution construction with two {\em global} positive solutions $u_0, \; u_1$ of the equation $Pu=0$ in $\Gw$  that have singularities ``at the boundary", instead of at an interior point, and get an optimal Hardy-weight $W$ in the {\em entire} domain $\Gw$. To understand the setting, we begin by presenting an example.
%%%%%%%%%%%%%%%%%%%%%%%%%%%%%%%%%%%
\begin{example}\label{ex1}
{\em
Let $P=-\Gd$, and consider the cone
$$\Gw:=\{x\in \mathbb{R}^n\mid  r>0,\gw\in \Gs\}\,,$$ where $\Gs$ is a Lipschitz domain in the unit sphere $S^{n-1}\subset \mathbb{R}^n$, $n\geq 2$, and $(r,\gw)$ denotes the spherical coordinates of $x$.
Let $\gth$ be the principal eigenfunction of the (Dirichlet) Laplace-Beltrami operator on $\Gs$ with eigenvalue $\lambda_0=\gl_0(\Gs)$, and set $$\ga_j:= \frac{2-n+(-1)^j\sqrt{(2-n)^2+4\gl_0}}{2}\,.$$ Then for $j=0$ (resp. $j=1$) the positive harmonic function  $u_j(r,w):=r^{\ga_j}\gth(\gw)$ is the (unique) Martin
kernel at $\infty$ (resp. $0$) \cite{P94}.

Applying the supersolution construction with the pair $(u_0,u_1)$, we obtain the Hardy-weight
$$W(x):=\frac{(n-2)^2+4\gl_0}{4|x|^2}\,. $$ Consequently, the corresponding Hardy-type inequality reads as
\begin{equation}\label{eq_4n2}
\int_{\Gw}|\nabla \phi|^2\dx\geq \frac{(n-2)^2+4\gl_0}{4}\int_{\Gw}\frac{|\phi|^2}{|x|^2}\dx \qquad \forall\phi\in C_0^\infty(\Gw).
\end{equation}
It follows from  Theorem \ref{thm_bs} that $W$ is an optimal Hardy-weight, and that the spectrum and the essential spectrum of $W^{-1}(-\Gd)$ is $[1,\infty)$. Note that \eqref{eq_4n2} and the {\em global} optimality of the constant is known (cf. \cite{FM,LLM}).
 }
\end{example}
Throughout this section (unless otherwise stated), we assume that the Martin boundary $\delta\Omega$ of $\Omega$ and $P$ is equal to the minimal Martin boundary and consists of $\partial\Omega\cup \{\xi_0,\xi_1\}$, where $\partial\Omega\sm \{\xi_0,\xi_1\}$ is assumed to be a regular manifold of dimension $n-1$ without boundary (in fact, it is enough to assume that $\partial\Omega\sm \{\xi_0,\xi_1\}$ is Lipschitz and satisfies the interior sphere condition). Note that it might be that one or two of the Martin points $\xi_0,\xi_1$ belong to $\pd\Gw$ (cf. Example~\ref{ex1}).

We denote by $\hat{\Omega}$ the Martin compactification of $\Omega$. Hence,
$$\hat{\Omega}:=\overline{\Omega}\cup\{\xi_0,\xi_1\}.$$
We assume that there exists a bounded domain $D\subset \Gw$ such that $\xi_0$ and $\xi_1$ belongs to two different connected components of $\hat{\Omega}\setminus \bar{D}$ that are neighborhoods of $\xi_0$ and $\xi_1$.

We need the following definition of minimal growth at a portion of the boundary  $\gd \Omega$:
%%%%%%%%%%%%%%%%%
\begin{definition}\label{def_9}{\em
Let $\omega\subset\delta\Omega$ be a closed set, and let $u$ be a positive solution of $Pu=0$ in a neighborhood $\Gw_1\subset\Gw$ of $\omega$. We say that $u$ has minimal growth at $\omega$ if for every positive supersolution $v$ of the equation $Pu=0$ in a relative neighborhood of $\omega$, we have
$$u\leq Cv$$
in a neighborhood $\Gw_2\subset\Gw_1$ of $\omega$.
 }
\end{definition}
We need two lemmas. The first one concerns minimal growth:

\begin{Lem}\label{minimal_vanish}
Assume that the coefficients of $P$ are locally regular up to  a Lipschitz portion $\Gg$ of $\partial\Omega$. Let $W$ be a nonnegative potential which is $L^\infty_{\mathrm{loc}}$ up to $\Gg$, such that $P-W\geq 0$ in $\Gw$.

\begin{enumerate}
\item Let $\omega\subset\Gg$ be the closure of a nonempty open set, and let $u$ be a positive solution of $P-W$ in a relative neighborhood of $\omega$. The following are equivalent:

\begin{enumerate}
 \item $u$ has minimal growth for $P-W$ at $\omega$.
 \item $u$ vanishes continuously on $\omega$.
\end{enumerate}
 \item Let $\omega=\omega_1\cup\omega_2$, where $\omega_1$ and $\omega_2$ are closed sets in $\delta\Omega$, and let $u$ be a positive solution of $P-W$ in a neighborhood of $\omega$. If $u$ has minimal growth for $P-W$ at $\omega_1$ and at $\omega_2$, then $u$ has minimal growth for $P-W$ at $\omega$.
\end{enumerate}
\end{Lem}
\begin{proof}
1) First, we extend $P$ (resp. $W$) in a neighborhood $U$ of $\omega$ in $\R^n$ such that the corresponding extension $\hat{P}$ (resp. $\hat{W}$) has H\"{o}lder continuous coefficients (resp. the extension is $L^\infty$). If $U$ is small enough, then the extended operator $\hat{P}-\hat{W}$ is nonnegative in $U$, and we can find a positive solution $\theta$ of the equation $(\hat{P}-\hat{W})u=0$  in $U$. By elliptic regularity, $\theta\in C_{\mathrm{loc}}^{1,\alpha}(U)$, and therefore $\tilde{P}:=\theta^{-1}(\hat{P}-\hat{W})\theta$ has H\"{o}lder continuous coefficients in $\bar{\Gw}\cap U$. By performing a ground state transform with respect to $\theta$, we see that it is enough to prove the lemma for $\tilde{P}$ instead of $P-W$ ; so we will assume that $u$ is a solution of $\tilde{P}$ instead. The fact that $(1a)$ implies $(1b)$ now follows from Lemma 3.2 in \cite{P94}.

For the proof that $(1b)$ implies $(1a)$ we may assume that $\gw$ is bounded. Let $\mathcal{O}\subset \Gw$ be a neighborhood of $\omega$ on which $u$ is a positive solution of the equation $\tilde{P}u=0$ that vanishes continuously on $\omega$. Let $\{\Omega_k\}_{k\in\mathbb{N}}$ be an exhaustion of $\Omega$ such that $\mathcal{O}_k:=\mathcal{O}\cap\Omega_k$ is regular. Let $w:=\lim_{k\to\infty}w_k$, where $w_k$ solves the Dirichlet problem
\begin{equation}\label{Dirichlet_minimal}
\left\{
\begin{array}{lr} \tilde{P}w_k=0 & \qquad\mbox {in } \mathcal{O}_k, \\[2mm]
		  w_k(x)=u & \qquad \mbox{on }\partial\mathcal{O}\cap\partial\mathcal{O}_k, \\[2mm]
          w_k(x)=0 & \qquad \mbox{on }\partial \Omega_k\cap\partial\mathcal{O}_k.
\end{array} \right.
\end{equation}
Then $w$ has minimal growth at $\omega$ (this follows from the \textit{local boundary Harnack principle}, see \cite{P94}). For every $\varepsilon>0$, we can find $k_0$ big enough such that $u<\varepsilon$ on $\partial \Omega_k\cap\partial\mathcal{O}_k$ for every $k\geq k_0$. Then, since $\tilde{P}\mathbf{1}=0$, $u+\varepsilon$ is a solution of $\tilde{P}$, and by the maximum principle $u<w_k+\varepsilon$. Letting $k\to\infty$ and then $\varepsilon\to 0$, we obtain $u\leq w$, which concludes the first part of the lemma.

Part 2) follows directly from the definition of minimal growth.
\end{proof}
We now turn to the second lemma concerning the regularity of the supersolution construction and the corresponding Hardy-weight on a portion of the boundary where the solutions $u_0$ and $u_1$ vanish.
\begin{Lem}\label{boundary_continuous}
Let  $\Sigma$ be an open subset of $\partial\Omega$. Assume that $\Omega$ is equipped with a Riemannian metric $\mathfrak{g}$, regular up to $\Sigma$. Let $u_0$ and $u_1$ be two positive functions defined in a neighborhood $\Gw'\subset \Gw$ of $\Sigma$ that are $C^2$ up to $\Gs$ and vanish continuously on $\Sigma$. Suppose that the gradients of $u_0$ and $u_1$ restricted to $\Sigma$ vanish nowhere. Then
$$W:=\frac{1}{4}\left|\nabla \log\left(\frac{u_0}{u_1}\right)\right|^2$$
has a continuous extension up to $\Sigma$ (here the gradient and its norm are computed with respect to $\mathfrak{g}$ and not to the Euclidean metric). If, in addition, $u_0/u_1$ has a continuous extension to $\Sigma$, then $u_0/u_1$ is in fact $C^1$ up to $\Sigma$.
\end{Lem}
\begin{proof}
Let us denote by $\vec{\gs}$ the unit exterior normal to $\Sigma$. Since $u_0$ and $u_1$ vanishes on $\Sigma$, the gradient of $u_0$ and $u_1$ are collinear to $\vec{\gs}$ on $\Sigma$. Next, we claim that near $\Gs$ we have for $i=0,1$,
\begin{equation}\label{eq_uiai}
\frac{|\nabla u_i|_A}{u_i}=\frac{1}{\delta}+g_i,
\end{equation}
where $\delta$ is the distance to $\partial\Omega$ with respect to the metric given by $\mathfrak{g}$, and $g_i$ is continuous up to $\Sigma$. Indeed, for $x_1$ be a point of $\Sigma$, let $\gamma_{x_1}$ be the unit speed geodesic starting at $x_1$, with $\gamma'(0)=-\vec{\gs}$ the interior normal. Let $r\geq 0$ be the coordinate on $\gamma$ (so that $r=\delta$ in restriction to $\gamma_{x_1}$, for $r$ small enough), then the restricting $u_i$ (resp. $|\nabla u_i|$) to $\gamma$ provides us with a function $f_i(r)$ (resp. $g_i(r)$). Notice that $f_i$ is $C^2$, $g_i$ is $C^1$ and $f_i'(0)=g_i(0)=|\nabla u_i|\neq 0$ (this comes from the fact that $\nabla u_i$ is collinear to $\vec{\gs}$, since $u_i$ vanishes on $\Sigma$). A Taylor expansion in $r$ gives (dropping the subscript $i$)

$$\frac{g(r)}{f(r)}=\frac{g(0)+g'(0)r+o(r)}{rf'(0)+\frac{r^2}{2}f''(0)+o(r^2)}=\frac{1}{r}+\left(\frac{g'(0)}{f'(0)}-\frac{f''(0)}{2f'(0)}\right)+o(r),$$
hence \eqref{eq_uiai} follows. From the same kind of consideration, we get in a neighborhood of $\Sigma$,

$$\frac{\nabla u_i(x)}{u_i(x)}=\frac{1}{\delta}\gamma'_{\exp^{-1}(x)}(x)+X_i,$$
where $X_i$ is a continuous vector field defined in a neighborhood of $\Sigma$ and $\exp^{-1}$ is the mapping sending a point $x$ to the unique point on $x_1\in \Sigma$ such that $x\in \gamma_{x_1}$.
The lemma follows at once, by noticing that

$$\nabla \left(\frac{u_0}{u_1}\right)=\frac{u_0}{u_1}\left(\frac{\nabla u_0}{u_0}-\frac{\nabla u_1}{u_1}\right)=\frac{u_0}{u_1}(X_0-X_1),$$
and that

$$W=\frac{\left|\nabla \left(\frac{u_0}{u_1}\right)\right|^2}{4\left|\frac{u_0}{u_1}\right|^2}\;.\qquad \qquad\qquad\qedhere$$
\end{proof}
We also need the following analogue of Proposition \ref{minimal} for a domain with boundary:

\begin{Pro}\label{minimal_boundary}
Let $P$ be a second-order nonnegative elliptic operator on $\Omega$ either of the form \eqref{P} or \eqref{div_P} with coefficients that are locally regular up to $\partial\Omega\setminus\{\xi\}$, where  $\xi\in\delta\Omega$. If $u$ and $v$ are two positive solutions of the equation $Pw=0$ in a relative neighborhood of $\xi$, which satisfy
$$\lim_{\substack{x\to\xi\\x\in \Gw}}\frac{u(x)}{v(x)}=0,$$
and both vanish on a punctured neighborhood of $\xi$ in $\delta\Omega$, then $u$ has minimal growth at $\xi$.
\end{Pro}
\begin{proof}The proof is almost exactly the same as the proof of Proposition \ref{minimal}. This time, we take a sequence of bounded sets   $\{\Omega_k:=B_1\setminus B_k\}$, where $\{B_k\}$ is a decreasing sequence of relative neighborhoods in $\hat{\Gw}$ of $\xi$ converging to $\xi$  such that $\pd\Omega_k$ is piecewise smooth.  With this definition, $\{\pd\Omega_k\}$ exhausts a punctured neighborhood  $\xi\in \gd\Omega$ (cf. the proof of Proposition \ref{minimal}). Let $w:=\lim_{k\to\infty}w_k$, where $w_k$ is the solution of the Dirichlet problem

\begin{equation}
\left\{
\begin{array}{lr} Pw_k=0 & \qquad\mbox {in } \Omega_k, \\[2mm]
		  w_k(x)=u & \qquad \mbox{on }\partial B_1\setminus\partial\Omega, \\
          w_k(x)=0 & \qquad \mbox{on }\big(\partial \Omega_k\cap\partial\Omega\big)\cup \pd B_k.
\end{array} \right.
\end{equation}
It follows (using the boundary Harnack principle and arguments similar to those in \cite{P94}) that $w$ has minimal growth at $\xi$. The end of the proof follows exactly the lines of the proof of Proposition \ref{minimal}.
\end{proof}

We now establish the main result of the present section.
\begin{theorem}\label{thm_bs}
Assume that $P$ is subcritical in $\Gw$.  Suppose that the corresponding Martin boundary $\delta\Omega$ is equal to the minimal Martin boundary and is equal to $\partial\Omega\cup \{\xi_0,\xi_1\}$, where $\partial\Omega\sm \{\xi_0,\xi_1\}$ is assumed to be a regular manifold of dimension $n-1$ without boundary, and the coefficients of $P$ are locally regular up to $\partial\Omega\sm \{\xi_0,\xi_1\}$.

Denote by $\hat{\Omega}$ the Martin compactification of $\Omega$, and assume that there exists a bounded domain $D\subset \Gw$ such that $\xi_0$ and $\xi_1$ belongs to two different connected components $D_0$ and $D_1$ of $\hat{\Omega}\setminus \bar{D}$ such that each $D_j$ is a neighborhood in $\hat{\Omega}$ of $\xi_j$, where $j=0,1$.

Let $u_0$ and $u_1$ be the minimal Martin functions at $\xi_0$ and $\xi_1$ respectively. Consider the supersolution
$v := \sqrt{u_0 u_1}$, and assume that

\begin{equation}\label{u1u0bs}
\lim_{\substack{x\to \gz_0\\x\in \Gw}} \frac{u_1(x)}{u_0(x)}=\lim_{\substack{x\to \gz_1\\ x\in \Gw}} \frac{u_0(x)}{u_1(x)}= 0.
\end{equation}

\vskip 3mm

Then the associated Hardy-weight $W:=Pv/v$ is {\em optimal} in $\Gw$. Moreover, if $P$ is symmetric and $W$ does not vanish on $\hat{\Omega}\setminus\{\xi_0,\xi_1\}$, then the (essential) spectrum of the operator $W^{-1}P$ acting on $L^2(\Omega,W\mathrm{d}\nu)$ is $[1,\infty)$.
\end{theorem}
%%%%%%%%%%%%%%%%%%%%%%%%%%%%%%%%%%%
\begin{proof}
We know that $u_i$ vanishes continuously on $\partial\Omega\setminus\{\xi_0,\xi_1\}$. Also, by Hopf's boundary point lemma, we know that the gradient of $u_i$ does not vanish on $\partial\Omega$. Define a metric $\mathfrak{g}$ on $\Omega$, regular up to $\partial\Omega$, by

$$\mathfrak{g}(\cdot,\cdot):=\langle A^{-1}\cdot,\cdot\rangle.$$
We have $\nabla_\mathfrak{g}=A\nabla$, and therefore,

$$W:=\frac{\left|\nabla_\mathfrak{g} \left(\frac{u_0}{u_1}\right)\right|_\mathfrak{g}^2}{4\left|\frac{u_0}{u_1}\right|^2}=\frac{\left|\nabla \left(\frac{u_0}{u_1}\right)\right|_A^2}{4\left|\frac{u_0}{u_1}\right|^2}\,.$$

\vskip 3mm

Now, we can apply Lemma \ref{boundary_continuous} with $\Sigma=\partial\Omega\setminus\{\xi_0,\xi_1\}$, to get that $W$ is continuous up to the boundary $\partial\Omega\setminus\{\xi_0,\xi_1\}$. Also, we know that $u_0/u_1$ has a continuous positive extension up to $\partial\Omega\setminus\{\xi_0,\xi_1\}$ (see part (i) of Theorem 7.1 in \cite{P94}). Hence, the $\log$ solution

$$\sqrt{u_0u_1}\log\left(\frac{u_0}{u_1}\right),$$
as well as the oscillating solutions

$$\sqrt{u_0u_1}\cos\left(\xi\log\left(\frac{u_0}{u_1}\right)\right),$$
vanish continuously on $\partial\Omega\setminus\{\xi_0,\xi_1\}$. By elliptic regularity up to the boundary, since $W$ is continuous up to $\partial\Omega\setminus\{\xi_0,\xi_1\}$, all these solutions are in fact $C^{1,\alpha}$ up to $\partial\Omega$, for some $\alpha\in (0,1)$.

Consequently, \eqref{u1u0bs} and Proposition \ref{minimal_boundary} imply that $\sqrt{u_0u_1}$ has minimal growth at $\xi_0$ and $\xi_1$. It also vanishes continuously on $\partial\Omega\setminus\{\xi_0,\xi_1\}$, and therefore has minimal growth on $\partial\Omega\setminus\{\xi_0,\xi_1\}$ by Lemma \ref{minimal_vanish}. Therefore, again by Lemma \ref{minimal_vanish}, it has minimal growth on $\delta\Omega$, i.e. at infinity in $\Omega$,
 and the criticality of $P-W$ follows.

 The optimality of the constant $1$ near $\xi_0$ and $\xi_1$ follows from the existence of the oscillating solutions.  Such a solution contradicts the generalized maximum principle near $\xi_0$ and $\xi_1$  for the operator $P-\lambda W$ with the corresponding  $\lambda>1$ (as in Theorem \ref{thm_Sch7}).

Concerning the null-criticality, the proof follows the same lines as in the proof of Theorem \ref{thm_null-critical}; here again we use the vanishing of the oscillating solutions on $\partial \Omega\setminus\{\xi_0;\xi_1\}$. This implies that the boundary of $\Omega$ will not cause trouble in the various integrations by part. The same remark also applies to the proof concerning the entire spectrum in the symmetric case.
\end{proof}
\begin{remark}\label{rm_ends}{\em
In the one-dimensional case (i.e. $n=1$, $\Gw=(a,b)$, where $-\infty\leq a<b\leq\infty$), with a {\em general} subcritical operator $P$, there are always two positive solutions of the equation $Pu=0$ in an interval $\Gw \subset \mathbb{R}$ that satisfy \eqref{u1u0bs}. Indeed, in this case one should take the two {\em minimal} positive solutions (Martin's kernels) of the equation $Pu=0$ in $\Gw$ corresponding to the two end points (cf. \cite{Murata86}).
 }
\end{remark}

The following example deals with an important class of operators with boundary singularities which satisfy the assumptions of Theorem \ref{thm_bs}, and in particular \eqref{u1u0bs}.

\begin{example} {\em \textbf{Fuchsian type operators}

\noindent Consider a \textit{Fuchsian} linear subcritical elliptic operator of the form \eqref{P} defined on the cone $\Gw:=\{x\in \mathbb{R}^n\mid  r>0,\gw\in \Gs\}$, where $\Gs$ is a Lipschitz domain in the unit sphere $S^{n-1}$ in $\mathbb{R}^n$, $n\geq 2$, and $(r,\gw)$ denotes the spherical coordinates of $x$. We assume that the coefficients of $P$ are up to the boundary locally H\"older continuous except at the origin. The operator $P$ has {\em Fuchsian singularities} both at $0$ and $\infty$ means that there exists a positive constant $M$ such that  near $0$ and $\infty$ we have
$$M^{-1} \sum_{i=1}^{n} \xi_{i}^2 \leq  \sum_{i,j=1}^{n} a^{ij}(x)\xi_{i}\xi_{j}\leq M \sum_{i=1}^{n} \xi_{i}^2 \qquad  \xi\in \R^n,$$ and
$$ |x|\sum_{i=1}^{n}|b_{i}(x)|+|x|^2|c(x)|\leq M .$$

It is known from \cite{P94} that the Martin boundary of $\Omega$ for $P$ is equal to the minimal Martin boundary, and is the union of the Euclidean boundary and $\infty$. For $j=0$ (resp. $j=1$), denote by $u_j$ the minimal Martin function with pole $0$ (resp. $\infty$). By \cite{P94} $u_j$ vanish on $\pd\Gw\sm \{0\}$, and
\begin{equation}\label{lim0}
\lim_{\substack{x\to 0\\x\in \Gw}} \frac{u_1(x)}{u_0(x)}=\lim_{\substack{x\to \infty\\ x\in \Gw}} \frac{u_0(x)}{u_1(x)}= 0.
\end{equation}
Applying Theorem \ref{thm_bs}, we conclude that if $W$ is the weight obtained by the supersolution construction applied to $u_0$ and $u_1$, then $W$ is an optimal Hardy-weight. Moreover, in the symmetric case the spectrum of $W^{-1}P$ is equal to $[1,\infty)$. In particular, the Hardy-weight of Example~\ref{ex1} is optimal. The same conclusions hold true for a bit more general domains (for example, truncated cones), see \cite{P94}.
}
\end{example}
%%%%%%%%%%%%
The following example deals with the case where one of the conditions of \eqref{u1u0bs} is not satisfied.
\begin{example}\label{ex1a}
{\em
Let $P=-\Gd$ and $\Omega=\Real^n_+$, $n>1$. Let $v_0(x):=C_nx_n/|x|^n$ be the Poisson kernel at the origin, and $v_1:=\mathbf{1}$. We note that in contrast to the pair $(v_0, x_n)$, the pair $(v_0,\mathbf{1})$ does not satisfy one of the assumptions in \eqref{u1u0bs}. An elementary computation shows that
$$W(x):=\frac{1}{4}\left(\frac{1}{|x_n|^2}+\frac{n(n-2)}{|x|^2}\right)\,, $$ which
is obviously greater than the corresponding well known Hardy potential $1/(2|x_n|)^{2}$,  and  we get the following Hardy inequality
$$\int_{\Real^n_+}|\nabla \phi|^2\dx\geq \int_{\Real^n_+}W(x)|\phi|^2\dx \qquad \forall\phi\in C_0^\infty(\Real^n_+).$$
Georgios Psaradakis kindly informed us recently that indeed the above inequality can be improved, and in fact, the following improved Hardy inequality holds true
\begin{equation}\label{psar}
\int_{\Real^n_+}|\nabla \phi|^2\dx\geq \frac{1}{4}\int_{\Real^n_+}\left(\frac{1}{|x_n|^2}+\frac{(n-1)^2}{|x|^2}\right)|\phi|^2\dx \quad \forall\phi\in C_0^\infty(\Real^n_+).
\end{equation}
This inequality was proved by Filippas, Tertikas and Tidblom in \cite[Theorem A]{FTT}. We show below, that this inequality is in fact optimal.

We note that by \cite[Theorem A]{FTT}, for every $0\leq \mu\leq \frac{1}{4}$ one can consider the Hardy inequality
\begin{equation}\label{half_space}
\int_{\Real^n_+}|\nabla \phi|^2\dx\geq \int_{\Real^n_+}\left(\frac{\mu}{|x_n|^2}+\frac{\beta(\mu)}{|x|^2}\right)|\phi|^2\dx \quad \forall\phi\in C_0^\infty(\Real^n_+),
\end{equation}
where for a fixed $\mu$, $\beta(\mu):=1-n-\sqrt{1-4\mu}$ is the best constant. Moreover, by \cite[Theorem B]{FTT}, inequality \eqref{half_space} cannot be improved by a Sobolev term.

{\bf Claim:}  {\em The Hardy inequality \eqref{half_space} is optimal. In particular, the operator $-\Gd -\frac{\mu}{|x_n|^2}-\frac{\gb(\gm)^2}{4|x|^2}$ is critical in $\Real^n_+$ with the ground state $\psi(x):=x_n^{\ga_+}|x|^{\beta(\gm)/2}$. Furthermore, no Sobolev improvement of \eqref{half_space} is possible.}

Indeed, for  $\mu\leq 1/4$, consider the subcritical operator $$P_\mu:=-\Gd -\frac{\mu}{|x_n|^2}$$ in $\Gw=\Real^n_+$.  Let $\ga_+$ be the largest root of the equation
$\ga(1-\ga)=\gm$, and let $$\beta(\gm):=1-n-\sqrt{1-4\gm}$$ be the nonzero root of the equation
\begin{equation*}
\beta\big(\beta+n-1+\sqrt{1-4\mu}\big)=0.
\end{equation*}
Then
$$w_0(x):=x_n^{\ga_+}, \qquad w_1(x):=x_n^{\ga_+}|x|^{\beta(\gm)}$$ are two positive solutions of the equation
$P_\gm u=0$ in $\Gw$.  Moreover, $w_1$ has minimal growth on $\partial \Gw$, and $w_0$ has minimal growth on $\partial \Gw\cup \{\infty\}\setminus \{0\}$. In particular,
$$\lim_{x\to 0} \frac{w_0(x)}{w_1(x)}=\lim_{x\to \infty} \frac{w_1(x)}{w_0(x)}=0.$$

Although the potential $|x_n|^{-2}$ is not smooth on $\partial \Gw \setminus \{0\}$, it can be easily checked that the proof of Theorem~\ref{thm_bs} applies also to the case of the operator $P_\mu$ in $\Gw$ with the pair of the positive solutions $w_0$ and $w_1$. This yields that inequality \eqref{half_space} is optimal. The criticality of $-\Gd -\frac{\mu}{|x_n|^2}-\frac{\gb(\gm)^2}{4|x|^2}$ implies that no Sobolev improvement is possible.
 }
\end{example}
\begin{remark}\label{rem_cone}{\em
A generalization of the optimal Hardy inequality \eqref{half_space} to the case of a cone will appear in a forthcoming  paper.
 }
\end{remark}

%%%%%%%%%%%%%%%%%%%%%
\mysection{Several ends}\label{sec_sever}
The {\em criticality result} for Hardy-weights obtained by a particular supersolution construction (Theorem~\ref{thm_Sch}) can be extended to the case where we have {\em a finite number of ends} in $\Omega$, instead of just two ends (e.g., one isolated singularity and $\infty$). For related results see also propositions~\ref{Pro_appB} and \ref{lem_severB} in Appendix~\ref{appen2}, and \cite{A1,BDE, CZ}.
%%%%%%%%%%%%%%
\begin{definition}\label{def_ends}{\em
Let $M$ be a noncompact manifold. We say that $M$ has $N$-{\em ends} $E_1,\ldots,E_N$, if each $E_i$ is a smooth non-compact connected manifold with boundary such that
%%%
$$M=\bigcup_{i=0}^NE_i,\quad \mbox{and } \qquad\bigcap_{i=1}^NE_i=\emptyset ,$$
where $E_0$ is a relatively compact, open set of $M$.  We denote the ideal ``infinity" point of each $E_i$ by   $x_i$ (that is, $x_i$ is the ideal limit point when  $x\to \infty$ in $M\cap E_1\setminus \big(\overline{\pd E_1\cap M}\big)$).
 }
\end{definition}
We need the following lemma, which is a slight extension of the results of \cite[Corollary~3.6]{P07}:

\begin{Lem}\label{min_growth_ends}

Let $P$ be a symmetric  operator on a manifold $M$ with ends $E_1,\ldots,E_N$ . For $i=1,2$, let $P_i=P+W_i$ be a nonnegative operator in $M$, where $W_i$ is a potential, and let $\phi_i$ be a positive solution of $P_iu=0$ in $E_1$. Assume further that
$$\varphi_2\leq C\varphi_1 \qquad \mbox{in } E_1,$$
and that  $\varphi_1$ has a minimal growth at  $x_1$  with respect to $P_1$.
Then $\varphi_2$ has a minimal growth at  $x_1$  with respect to $P_2$.
\end{Lem}
\begin{proof}
We first modify $\varphi_i$ so that it has minimal growth for $P_i$ on $E_1$ (seen as a manifold with boundary). To this purpose, let us consider $U$ a compact, smooth, open set which is a neighborhood of $\partial E_1$ in $E_1$. Let $\psi$ be a positive solution of $P_1u=0$ in $U$, with minimal growth at $\partial E_1\cap M$. Now consider a positive function $\tilde{\varphi}_1(x)$ (resp. $\tilde{\varphi}_2(x)$) which is equal to $\psi(x)$ on a neighborhood of $\partial E_1$, and to $\varphi_1(x)$ (resp. $\varphi_2(x)$) near $x_1$. Let $\tilde{W}_i$ be a potential such that $(P+\tilde{W}_i)\tilde{\varphi}_i=0$ in $E_1$. By the (AAP) theorem, $\tilde{P}_i:=P+\tilde{W}_i$ is nonnegative. Also, by construction, $\tilde{\varphi}_1$ has minimal growth (globally) in $E_1$, considered as a subdomain of $M$, and therefore $\tilde{P_1}$ is critical in $E_1$. Furthermore, we still have (with a different constant $C$)
$$\tilde{\varphi}_2\leq C\tilde{\varphi}_1\qquad \mbox{ in } E_1.$$
 Now, \cite[Theorem 1.7 or Corollary 3.6]{P07} implies that $\tilde{P}_2$ is critical in $E_1$, and $\tilde{\varphi}_2$ is its ground state. Therefore, $\tilde{\varphi}_2$ has minimal growth (globally) on $E_1$. Since $\tilde{\varphi}_2(x)=\varphi_2(x)$ near $x_1$, the lemma is proved.
\end{proof}
We now formulate the main result of the present section that claims that in the case of finitely many ends the supersolution construction produces an $N-1$-parameter family of critical Hardy-weights.
 \begin{theorem}\label{lem_sever}
 Suppose that $P$ is a symmetric subcritical operator in a manifold $M$ with ends $E_1,\ldots,E_N$, $N\geq 2$. Assume that  for each $1\leq i\leq N$ there exists a function  $u_i$ which is a positive solution of  the equation $Pu=0$ in $M$ of minimal growth near each end $x_j$, $j\neq i$, and satisfying
$$\lim_{x\to x_i} \frac{u_j(x)}{u_i(x)}=0  \qquad \forall j\neq i.$$

Consider the supersolution construction

$$v:=\prod_{j=1}^N u_j^{\alpha_j}\,,$$
 where $0<\alpha_j\leq 1/2$ for all $1\leq j\leq N$, and $\sum_{j=1}^N\alpha_j=1$.

Then the corresponding Hardy-weight $W:=Pv/v$ is critical with respect to $P$ and $M$.
\end{theorem}
\begin{proof}
Note that by the definition of minimal growth, for each $i$, and every $k,j \neq i$ we have
$$u_j(x)\asymp  u_k(x)\qquad \mbox{ as } x\to x_i.$$

Denote $\hat{u}_i:=\prod_{j\neq i}u_j^{\alpha_j}$. Fix $1\leq i \leq N$ and $k\neq i$. Then near $x_i$ the following inequality holds

$$v = u_i^{\alpha_i}  \hat{u_i}^{1-\alpha_i} \leq  C u_i^{\alpha_i}  u_k^{1-\alpha_i} =
 C(u_i u_k)^{1/2}\left( \frac{u_k}{u_i}\right)^{1/2-\alpha_i} \leq (u_i u_k)^{1/2}.$$
Recall that it follows from the {\em proof} of  Theorem~\ref{thm_Sch} (or Theorem~\ref{thm_bs}) that  $(u_i u_k)^{1/2}$ has minimal growth at $x_i$ with respect the symmetric operator $P-W_{i,k}$, where $W_{i,k}$ is the Hardy-weight corresponding to the pair $(u_i,u_k)$.

Hence, by Lemma \ref{min_growth_ends}, $v$ has minimal growth at $x_i$. Since this is true for all
$1\leq i\leq N$, it follows that the operator $P-W$ is critical in $M$.
\end{proof}

\begin{remark}\label{rem_sev}{\em
1. Theorem~\ref{lem_sever} should hold also in the nonsymmetric case (cf. Proposition~\ref{lem_severB}).

\vskip 3mm

2. We plan to study the other optimality properties of the Hardy-weights of Theorem~\ref{lem_sever} (besides the criticality) in a subsequent paper. Note that if the number of ends $N$ is greater than $2$, then these critical Hardy-weights might be not optimal, see Remark~\ref{rem_Bosi_Zu}.

\vskip 3mm

3. For a slightly different approach in some particular cases, see propositions~\ref{Pro_appB} and \ref{lem_severB} in Appendix~\ref{appen2}.
 }
 \end{remark}

%%%%%%%%%%%%%%%%%%%%%%%%%%%%%%%%%%%%%%%%
\mysection{Examples, applications and problems}\label{examples}
%%%%%%%%%%%%%%%%%%%%%%%%%%%%%%%%%%
In this section we present some further examples, and discuss some additional applications and extensions. First, we present a straightforward example of an optimal Hardy-weight.
\subsection{Further examples}\label{further_ex}
\begin{example}\label{ex2}
{\em
Consider the Laplace operator $P=-\Gd$ on the unit disk $\Omega=B(0,1)\subset \mathbb{R}^2$. Take $v_0(x):=-\frac{1}{2\pi}\log|x|$, the Green function of the unit ball with a pole at the origin,  and let $v_1:=\mathbf{1}$. Then the corresponding optimal Hardy-weight is given by $W(x)=\big(4|x|\log|x|\big)^{-2}$ defined on $B(0,1)^\star:=B(0,1)\setminus \{0\}$. We obtain the classical {\em Leray inequality} \cite{Leray} with the best constant
$$\int_{B(0,1)^\star}|\nabla \phi|^2\dx \geq \frac{1}{4}\int_{B(0,1)^\star}\frac{|\phi|^2}{\big(|x|\log|x|\big)^2}\dx \qquad \forall \phi\in C_0^\infty(B(0,1)^\star),$$
 cf. \cite[(1.3)]{AS}. In particular, the operator $-\Delta- W$ is null-critical in $B(0,1)^\star$, and  $\lambda_0(-\Delta,W,B(0,1)^\star)=\gl_\infty(-\Gd,W,B(0,1)^\star)=1$.

Analogously, in higher dimension $n\geq 3$, let $v_0(x):= C_n(|x|^{2-n}-1)$ be the Green function of the unit ball $B(0,1)$ with a pole at the origin, and let $v_1:=\mathbf{1}$.  Then
$$W(x)= \frac{(n-2)^2}{4\big(|x|(1-|x|^{n-2})\big)^2}\,,$$
and the following optimal inequality holds true
$$\int_{B(0,1)^\star}\!\!\!\!|\nabla \phi|^2\dx \!\geq\!\! \Big(\frac{n\!-\!2}{2}\Big)^2\!\!\!\int_{B(0,1)^\star}\!\!\frac{|\phi|^2}{\big(|x|(1-|x|^{n-2})\big)^2}\dx \;\;\; \forall \phi\in C_0^\infty(B(0,1)^\star).$$ In particular, the operator $\Delta-W$ is null-critical in $B(0,1)^\star$, furthermore, cf. \cite[Section~1.1]{AS}, we have
$$\lambda_0(-\Delta,W,B(0,1)^\star)=\lambda_\infty(-\Delta,W,B(0,1)^\star) =1.$$
 }
\end{example}

\begin{example}\label{ex21}
{\em
The aim of the present  example is to give an alternative proof that $1/4$ is the best constant in the classical Hardy inequality \eqref{eq_Hardy12} for a \textit{smooth} convex bounded domain $\Gw$ (see the discussion in Example~\ref{exconv}).
If we use the supersolution construction with $P=-\Gd$, $u_0=G$ (the Green function), and $u_1=\mathbf{1}$, we get an optimal Hardy-weight $W:=\frac{1}{4}\left|\frac{\nabla G}{G}\right|^2$. Recall that $G$ vanishes on $\partial\Omega$ (in fact, $G(x)\asymp \gd(x)$ near the boundary). By Hopf's lemma, $\partial G/\partial \vec{\gs}$ does not vanish on $\partial\Omega$, where $\vec{\gs}$ is the outer normal vector to $\partial\Omega$. Hence, by the proof of Lemma \ref{boundary_continuous}, we have
\begin{equation}\label{asymptotic_W}
W(x)\sim \frac{1}{4\gd(x)^2} \qquad \mbox{ as } x\to \pd \Gw.
\end{equation}
Since we know that $\lambda_\infty(P,W,\Omega)=1$, we deduce that $1/4$ is indeed the best constant in the classical Hardy inequality \eqref{eq_Hardy12}. It is also easy to deduce from the fact that $P-W$ is null-critical that the classical Hardy inequality \eqref{eq_Hardy12} has no minimizer (this also follows from the subcriticality of $-\Gd-\gd(x)^{-2}/4$). We do not know if the asymptotic of $W$ given by \eqref{asymptotic_W} remains true if $\Omega$ has a rougher boundary. On the other hand, B.~Devyver recently proved \cite{Dev} that the spectrum and the essential spectrum of $-4\gd(x)^{2}\Gd$ on $L^2(\Gw,\big(4\gd(x)\big)^{-2}\mathrm{d} \nu)$ is equal to $[1,\infty)$.
 }
\end{example}
In the next two examples we apply the supersolution construction to positive solutions with boundary singularities.
\begin{example}\label{ex22}
{\em Consider the operator $Pu:=-u''$ on $\mathbb{R}_+$, and apply the supersolution construction with the positive solutions $u_0(x)=x$, $u_1(x)=\mathbf{1}$. By Theorem \ref{thm_bs}, we readily get
the classical Hardy inequality on $\mathbb{R}_+$ with the optimal Hardy-weight $W(x):=1/(4x^2)$. We note that the corresponding transform \eqref{trans} is just the classical {\em Mellin transform}.
 }
\end{example}
\begin{example}\label{ex23}
{\em
Consider the operator $Pu:=-u''+u$ defined on $\mathbb{R}$, with $u_0(x)=\mathrm{e}^x$, $u_1(x)=\mathrm{e}^{-x}$. Applying Theorem \ref{thm_bs} we obtain the optimal Hardy-weight $W:=\mathbf{1}$, and we get the trivial inequality $P-W=-\mathrm{d}^2/\dx^2\geq 0$ in $\R$. The corresponding transform \eqref{trans} is just the classical {\em Fourier transform} on $\mathbb{R}$.
 }
\end{example}
%%%%%%%%%%%%%%%%%%%%

\subsection{Decay of solutions and estimates of $W$ near infinity. }
The supersolution construction provides bounds for solutions near infinity in terms of the Green function $G$ and a global solution $u$. In particular, we have
%%%%%%%%%%%%%%%%%%%%
\begin{lemma}\label{thm_decay}
Let $P$ be a subcritical operator in $\Gw$, and let $W$ be a Hardy-weight in $\Gw$ associated to a pair $(v_1,v_2)$, where $v_1,\,v_2$ are positive solutions of the equation $Pu=0$ in $\Gw$.  Let $v$ be a positive supersolution of the equation
$$
(P -V) v=0
$$
of minimal growth at  infinity with respect to $P-V$ in $\Gw$.

Suppose further that  $$V(x) \leq  4 \alpha (1-\alpha) W(x)\qquad \mbox{in} \quad \Omega'$$ holds true for some $1/2\leq\alpha\leq 1$, and some neighborhood $\Gw'$ of infinity in $\Gw$. Then for any $1/2 \leq \beta \leq \alpha$ there exists a constant $C$ such that the inequality
$$
 v(x) \leq C v_1^{1 -\beta}(x) v_2^\beta(x)
$$
holds true in a neighborhood of infinity of $\Omega$.
\end{lemma}
\begin{proof}
The function $v_1^{1-\beta} v_2^\beta$ is a positive supersolution of the operator $P - V$ in a neighborhood of infinity of $\Gw$. The claim then follows by the definition of positive solutions of minimal growth.
\end{proof}

\begin{example}\label{ex_est_conv}{\em
Let $\Omega$ be a smooth bounded convex domain and $u$ a positive solution of the equation
$$ (P - V)u = 0,$$
  of minimal growth in a neighborhood of infinity in $\Omega$. Suppose further that for some $1/2 \leq \alpha \leq 1$, the inequality $V(x) \leq \alpha (1- \alpha) \delta^{-2}(x)$ holds true in a neighborhood $\Omega'$ of infinity of $\Omega$. Then
 $$
  u(x) \leq C \delta(x)^\alpha \qquad \mbox{in }\, \Omega'.
 $$
  }
\end{example}
\vskip 5mm
In order to apply Lemma~\ref{thm_decay} for the pair $(u,\,G)$, where $G$ is the Green function and $u$ is a global  positive solution satisfying \eqref{u1u0a}, one needs to know the behavior of the optimal Hardy-weight $W$ near infinity, and to compare pointwise $V$ and $W$, if $V$ is a (non-optimal) Hardy potential. In full generality, it seems hopeless to get an asymptotic of $W$ at infinity, since $\nabla G$ might vanish on a nonempty set with an accumulation point at infinity in $\Gw$ (of course, it is expected that this set should be small).

However, in the symmetric case  we have an asymptotic of $W$ \textit{in average}
 at infinity, as follows from Corollary \ref{Estimate W}, which, if $u$ is normalized so that $u(0)=1$, gives that
$$\int_{\{a\leq\frac{G}{u}\leq b\}}uGW\,d\nu= \frac{1}{4}(\log b-\log a).$$
 Moreover, in average, we can compare $W$ and any Hardy-weight $V$ near infinity.
\begin{Pro}
Suppose that $ P$ is symmetric and the hypotheses of Theorem \ref{thm*} are satisfied (with $P$, $u$, $G$ and $W$ as in the theorem).  Let $V$ be a nonnegative potential such that $P-V\geq 0$ in $\Gw^\star$. Then for every $1<a<b<\infty$ (or $-\infty<a<b<-1$), we have
$$\int_{\{a\leq \log \frac{G}{u}\leq b\}} \hspace{-10mm}uGV\dnu\leq 5\int_{\{a-1\leq \log \frac{G}{u}\leq b+1\}} \hspace{-10mm} uGW\dnu= \frac{5}{4} [b- a+2].$$

\end{Pro}
\begin{proof}
By performing a ground state transform, we may assume that $u=\mathbf{1}$. We start with the following inequality

$$\int_{\Gw^\star} Vv^2\dnu\leq \int_{\Gw^\star} v\,P[v]\dnu \qquad \forall v\in C_0^\infty(\Omega^\star),$$
which holds true by our assumption.

Fix $1<a<b<\infty$, and  let $\psi$ be a smooth nonnegative cut-off function supported in $\{a-1\leq \log G\leq b+1\}$,  such that $\psi=1$ on $\{a\leq \log G\leq b\}$. Set $v:=G^{1/2}\psi$, and recall that $(P-W)(G)^{1/2}=0$.  Therefore, by \eqref{rule1} we have
$$\int_{\{a-1\leq \log G\leq b+1\}} \hspace{-10mm}v\,P[v]\dnu=\int_{\{a-1\leq \log G\leq b+1\}}\hspace{-10mm}\big[GW\psi^2-\frac{1}{2}\langle \nabla G,\nabla \psi^2\rangle_A +G\psi P[\psi]\big]\dnu.$$
Now, integrate by part the last term to get

\begin{equation}\label{estimate V}
\int_{\{a\leq \log G\leq b\}} \hspace{-10mm}GV\dnu\leq\int_{\{a-1\leq \log G\leq b+1\}} \hspace{-10mm}GW\psi^2 \dnu +\int_{\{a-1\leq \log G\leq b+1\}} \hspace{-10mm} G|\nabla\psi|_A^2\dnu.
\end{equation}
Consider the function $\psi$ defined by
$$\psi(x):=\left\{
             \begin{array}{ll}
               1 & \quad x\in \{a\leq \log G\leq b\}, \\[2mm]
               b+1-\log G(x) &\quad x\in \{b\leq \log G\leq b+1\}, \\[2mm]
               \log G(x)-a+1 & \quad x\in \{a-1\leq \log G\leq a\}, \\[2mm]
               0 & \quad \hbox{ elsewhere.}
             \end{array}
           \right.
$$
Now, take a sequence $\{\psi_k\}\subset C_0^\infty(\Gw^\star)$ of  smooth function $0\leq \psi_k\leq \psi$ which converges in $W^{1,2}$ to $\psi$. Since $\psi$ is in $W_0^{1,2}(\Gw^\star)$, we can find such a sequence $\{\psi_k\}$. Applying \eqref{estimate V} to $\psi_k$ and passing to the limit as $k\to\infty$ gives
$$\int_{\{a\leq \log G\leq b\}} \hspace{-10mm} GV\psi\dnu\leq\int_{\{a-1\leq \log G\leq b+1\}}\hspace{-10mm} GW\psi^2 \dnu+ \int_{\{a-1\leq \log G\leq b+1\}}\hspace{-10mm} G|\nabla\psi|_A^2 \dnu.$$
We use finally the fact that $\psi$ is supported in $\{a-1\leq \log G\leq b+1\}$, that $0\leq\psi\leq 1$ and that $|\nabla\psi|^2_A\leq 4W$ to get the result.
\end{proof}

It is natural to formulate the following conjecture about the pointwise asymptotic of the optimal Hardy-weight $W$.
\begin{conjecture}\label{13_8}
Let $P u = -\laplace u + V(x) u + u$ be subcritical in $\R^n$  and assume that $\lim_{x\to \infty} V(x) = 0$.
  If $u$ is a positive solution of $Pu=0$ in $\R^n$ satisfying \eqref{u1u0a}, then the optimal Hardy-weight $W$ associated to the pair $(u,\,G)$ satisfies
$$
\lim_{x \to \infty} W(x) =1.
$$
\end{conjecture}
\begin{remark*}[Added on 15/10/2016; after the publication of the paper in JFA]{\em
Mr.~Idan Versano kindly pointed out to us that Conjecture~\ref{13_8} is not correct. He gave the following elementary counterexample.

Consider the operator  $P:=-\Gd+1$ on $\R^n$, and let $G(x):=G_P^{\R^n}(x,0)$ be the corresponding Green function. Consider  the positive solution $u(x):=\exp(x_1)$, where $x=(x_1,\ldots,x_n)$, and let $W$ the corresponding optimal Hardy weight.
Then one can easily verified that $\liminf_{x\to \infty} W(x)=0$, but $\limsup_{x\to \infty} W(x)=1$.
 }
\end{remark*}
%%%%%%%%%%%%%%%%%%%%%%%%
\begin{remark}\label{rem_asymp}{\em
 In many cases the asymptotic of the Green function at infinity is known. Therefore, knowing the asymptotic of the optimal Hardy-weight $W$ associated to a pair $(\mathbf{1},G)$ will lead to the asymptotic of $|\nabla G|_A$ at infinity,  such information is rarely available.
 }
\end{remark}
%%%%%%%%%%%%%%%%%%%%%%%%%%%%%%%%%
\subsection{Regularization}\label{reg} The main result of our paper (Theorem~\ref{thm_nsa}) provides us with an optimal Hardy-weight $W$ defined in the {\em punctured} domain $\Gw^\star$ rather in $\Gw$. This drawback can be easily relaxed using the following regularization procedure. Let $\tilde{W}\leq W$ be a (locally) regular nonnegative potential in $\Gw$ such that $\tilde{W}  =W$ outside
a punctured neighborhood of the origin. Clearly, $P-\tilde{W}$ is subcritical in $\Gw$. Let $V\in C_0^\infty(\Gw)$ be a smooth nonzero nonnegative function such that $P-\tilde{W}-V$ is critical in $\Gw$ (see, Lemma~\ref{CompPot}). Then the potential
$\hat{W}:=\tilde{W}-V$ is critical in $\Gw$, null-critical at infinity of $\Gw$, and $\lambda_\infty(P,\hat{W},\Omega)=1$. Moreover, in the symmetric case, by Theorem~\ref{essential_spectrum}, the corresponding spectrum and essential spectrum of $\hat{W}^{-1}P$ is equal $[1,\infty)$. So, $\hat{W}$ is an optimal Hardy-weight  for  $P$ in $\Gw$.

\subsection{The quasilinear case}
%\label{sec33}
In this section, we briefly discuss some extensions of the previous results to the case of $p\,$-Laplacian type equations (for some related results see \cite{AS,DA,KO}). Throughout the present subsection we assume that $p\neq 2$. The celebrated $p\,$-Laplacian is the quasilinear elliptic operator
$$\Delta_p(u):= \mathrm{div}\left(|\nabla u|^{p-2}\nabla u\right).$$
Let $V\in L^\infty_{\mathrm{loc}}(\Gw)$ be a given function (potential), we  consider the functional
\be\label{Q_V}
Q_V(\phi):=\int_\Omega(|\nabla \varphi|^p+V|\varphi|^p)\dx\qquad \varphi\in C_0^\infty(\Omega)
\ee and the associated differential operator
\begin{equation}\label{eq:plaplace}
    Q'_V(u):=-\Delta_p(u)+V|u|^{p-2}u.
\end{equation}
The notions of criticality and subcriticality of $Q_V$ have been studied in this context, and we refer to \cite{ky3} for an account on this. In particular, the Agmon-Allegretto-Piepenbrink theorem extends to this case \cite[Theorem~2.3]{ky3}. Due to the nonlinearity of the operator,  if the potential $V$ is nonzero it is likely that our supersolution construction will not yield in general an optimal weight, as we can see from the following result in the radially symmetric case.
\begin{theorem}\label{thm_super_p_constr}
Assume that the functional
\be\label{Q_V1}
Q_V(\phi):=\int_\Omega(|\nabla \phi|^p+V(|x|)|\phi|^p)\dx\qquad \phi\in C_0^\infty(\Omega),
\ee is subcritical in a radially symmetric domain $\Omega\subset\mathbb{R}^n$, where  the potential $V$ is radially symmetric. Suppose further that either $1<p \leq 2$ and $V\geq 0$, or $p \geq 2$ and $V\leq 0$. Let $v_0,v_1$ be two linearly independent positive radially symmetric supersolutions of the equation $Q'_V(u)=0$ in $\Omega^\star: =\Omega\setminus \{0\}$.

Define the function $$v_\ga(|x|):= (v_1(|x|))^{\ga}(v_0(|x|))^{1-\ga}\qquad x\in\Gw^\star,$$
where $0\leq \ga\leq 1$, and let
 $$W_\ga(t):= \ga(1-\ga)(p-1)
 \left|\left[\log\left(\frac{v_0(t)}{v_1(t)}\right)\right]'\right|^2\left|\left[\log(v_\ga(t))\right]'\right|^{p-2}.$$
Then $v_\ga$ is a positive supersolution of the equation
\begin{equation}\label{eqQv-Wa}
  Q_{V-W_\ga}'(u)= 0 \qquad \mbox{ in } \Omega^\star,
\end{equation}
and the following improved inequality holds
$$Q_V(\phi) \geq \int W_\ga |\phi|^{p}\dx \qquad \forall \phi\in C_0^\infty(\Omega^\star). $$
Moreover, if $p\neq 2$ and $V$ is not identically zero, then for every $\alpha\in(0,1)$ the functional $Q_{V-W_\ga}$ is {\em subcritical} in $\Omega^\star$.
\end{theorem}
The proof of this theorem will appear somewhere else ({\em added after the paper was accepted for publication:} the proof has appeared in \cite[Theorem A1]{DP14}). Notice that in the case where $p\neq 2$, the supersolution construction yields a weight $W_\alpha$ which is not easy to optimize with respect to $\alpha$. On the other hand, for the case of the $p\,$-Laplacian itself in a general domain $\Gw\subset \R^n$, we can take $u_0=\mathbf{1}$, and thus optimize $W_\alpha$ (cf. \cite[Lemma 3.57]{Martio}):
\begin{Pro}\label{supersolution p-Laplacian}
Assume that $v$ is a  positive  supersolution (resp. solution) of the equation $-\Delta_p(u)=0$ in $\Gw$. Then for $\alpha\in (0,1)$, $v^\alpha$ is a positive supersolution (resp. solution) of the equation $Q'_{W_\alpha}(u)=0$ in $\Gw$, where
$$W_\alpha:=\alpha^{p-2}\alpha(1-\alpha)(p-1)\left|\frac{\nabla v}{v}\right|^p.$$
In particular, for the optimal value $\alpha=\frac{p-1}{p}$, the following logarithmic Caccioppoli inequality holds:
\begin{equation}\label{cacc}
\left(\frac{p-1}{p}\right)^p\int_\Omega \left|\frac{\nabla v}{v}\right|^p |\varphi|^p\dx\leq\int_\Omega |\nabla\varphi|^p\dx\qquad \varphi\in C_0^\infty(\Omega),
\end{equation}
where $v$ is any positive $p$-superharmonic function in $\Gw$.
\end{Pro}
We omit the proof of Proposition~\ref{supersolution p-Laplacian}. We note that the proposition and its generalization have been independently derived by L.~D'Ambrosio and S.~Dipierro, \cite{DAD}. However, the following problem remains open:
%%%%%%%%%%%%%%%%%%%%%%%%
\begin{problem}\label{open_pr}
Let $\Omega^\star:=\Omega\setminus\{0\}$, and assume that the functional $\int_\Omega|\nabla u|^p$ is subcritical in $\Gw$. Let $G$ be the Green function for $-\Delta_p$ in $\Gw$ with a pole at zero, and assume that

$$\lim_{x\to\infty}G(x)=0.$$
(for example, this holds if $\Omega=\R^n$, and $p<n$). Is the weight $$W=\left(\frac{p-1}{p}\right)^p\left|\frac{\nabla G}{G}\right|^p$$
an optimal Hardy-weight for the $p\,$-Laplacian in $\Gw^\star$?
\end{problem}
\begin{remark}[Added after the paper was accepted]\label{remDP14}{\em
Problem~\ref{open_pr} was recently affirmatively solved in \cite{DP14}.
 }
\end{remark}
\appendix
\section{Radially symmetric Schr\"odinger operators}\label{appendix1}
%%%%%%%%%%%%%%%%%%%%%%%%%%%%%
In this appendix we discuss the important family of radially symmetric Schr\"odinger operators defined on radially symmetric domains. The results of our paper obliviously apply to this case. On the other hand, since the technique we used  throughout the paper is based on a one-variable approach, it is natural to consider this particular family of operators, and give an alternative direct proof of some parts of Theorem~\ref{thm*} for this case.

We consider the supersolution construction in the case of a nonnegative Schr\"odinger operator $P=-\Delta+V$ in $\R^n$, (or a radially symmetric subdomain) where $n\geq 2$,  and $V$ is a radially symmetric potential.
\begin{theorem}\label{thm_radial}
Consider a subcritical Schr\"odinger operator $P=-\Delta+V$ in $\R^n$, where $n\geq 2$,  and $V$ is a radially symmetric potential. Let $\psi$ be the unique global positive radial solution of the equation $Pu=0$ in $\R^n$, and $g_0$ the corresponding positive minimal Green function with a pole at $0$. In the supersolution construction, take  $G_0:=\sqrt{\psi g_0}$, and let $$W(r):=\frac{1}{4}\Big|\big[ \log(g_0(r)/\psi(r)) \big]'\Big|^2\qquad r>0$$ be the corresponding  Hardy-weight.

Then the operator $P-W$ is null-critical in $\Gw^\star:= \mathbb{R}^n\setminus \{0\}$.
In particular, $\gl=1$ is the best constant for the inequality
\be \label{improved_ineq1}
\int_{\Gw^\star}\left(\left|\nabla u\right|^2+V(x)u^2\right)\dx  \geq \gl \int_{\Gw^\star} W(|x|)\, u^2 \dx \qquad u\in C_0^\infty(\Gw^\star).\ee
Moreover, $\lambda_\infty(P,W,\Gw^\star)=1$.
\end{theorem}
\begin{remark}\label{rem_GU}
{\em
1. In the radially symmetric case, condition \eqref{u1u0} always holds true. In particular, for $\Gw=\R^n$ we have, $\lim_{r\to\infty} \frac{g_0(r)}{\psi(r)}=0$.

2. The results of our paper gives an alternative proof of a @@recent@@ result of Gesztesy and \"{U}nal  \cite[Theorem~2.1]{GU}.
 }
 \end{remark}

\begin{proof}
 By Murata's criterion for the subcriticality of radially symmetric Schr\"odinger operators  \cite[Theorem~3.1]{Murata86}, we know that the operator $P$ is subcritical in $\mathbb{R}^n$ if and only if
\begin{equation}\label{murata}
\int_1^\infty r^{1-n} (\psi(r))^{-2}\dr < \infty ,
\end{equation}
 and in this case, $$g_0(|x|):=\psi(|x|) \int_{|x|}^\infty r^{1-n} (\psi(r))^{-2}\dr$$  is the corresponding positive minimal Green function with a pole at $0$ (up to a multiplicative constant).

Assume that \eqref{murata} is satisfied, and take for the supersolution construction the function $G_0:=\sqrt{\psi g_0}$. So,
\begin{equation}\label{eq_G_0}
G_0(t) := \sqrt{\psi(t) g_0(t)} = \sqrt{(\psi(t))^2 \int_t^\infty r^{1-n} (\psi(r))^{-2}\dr}.
\end{equation}
Hence, $ G_0$ is  a positive global solution of the equation
\begin{equation}\label{eqW}
\left(-\Delta+V-W\right)u=0 \qquad \mbox{in } \Gw^\star,
\end{equation}
where
\begin{equation}\label{eqW2}
W(r):=\frac{1}{4}\Big|\big[ \log(g_0(r)/\psi(r)) \big]'\Big|^2=\frac{r^{2-2n}}{4\left[\psi(r)g_0(r)\right]^{2}}\,.
\end{equation}

It follows from the criterion \eqref{murata} that the operator $-\Delta+V-W$ is critical in $\Gw^\star$, with a ground state $G_0(t)$ if and only if
\begin{eqnarray}\label{murata1}
\begin{array}{lll}
\displaystyle{\int_1^\infty t^{1-n}\Big[(\psi(t))^2 \int_t^\infty r^{1-n} (\psi(r))^{-2}\dr\Big]^{-1}\dt} &=&\infty, \\[7mm] \hspace{4cm} \mbox{ and}\\[5mm]
\displaystyle{\int_0^1 t^{1-n} \Big[(\psi(t))^2 \int_t^\infty r^{1-n} (\psi(r))^{-2}\dr\Big]^{-1}\dt} &=& \infty.
\end{array}
\end{eqnarray}
But by \eqref{murata} we have
\begin{multline*}
\int_1^\infty t^{1-n} \Big[(\psi(t))^2 \int_t^\infty r^{1-n} (\psi(r))^{-2}\dr\Big]^{-1}\dt \\[3mm]
=-\int_1^\infty  \Big[\log(g_0(t)/\psi(t)) \Big]' \dt = \log(g_0(1)/\psi(1)) -\lim_{t\to \infty} \log(g_0(t)/\psi(t))\\[3mm]
=  \log \left[\int_1^\infty r^{1-n} (\psi(r))^{-2}\dr\right] - \lim_{t\to \infty} \log \left[\int_t^\infty r^{1-n} (\psi(r))^{-2}\dr\right] =\infty.
\end{multline*}
Moreover, since $n\geq 2$ we have,
\begin{multline*}
\int_0^1 t^{1-n} \Big[(\psi(t))^2 \int_t^\infty r^{1-n} (\psi(r))^{-2}\dr\Big]^{-1}\dt \\
=-\int_0^1 \Big[ \log(g_0(t)/\psi(t)) \Big]' \dt = -\log(g_0(1)/\psi(1))+\lim_{t\to 0} \log(g_0(t)/\psi(t))\\
=   -\log \left[\int_1^\infty r^{1-n} (\psi(r))^{-2}\dr\right]+ \lim_{t\to 0} \log \left[\int_t^\infty r^{1-n} (\psi(r))^{-2}\dr\right] =\infty.
\end{multline*}
So, \eqref{murata1} is satisfied and the operator $-\Delta+V-W$ is critical in $\Gw^\star$.
Similarly, one shows that $-\Delta+V-W$ is null-critical in $\Gw^\star$.

Next, we investigate the bottom of the essential spectrum of the corresponding operator. Let $\tilde{W}$ be a positive continuous function in a neighborhood of the origin such that $\tilde{W}=W$ outside a ball $B$ centered at the origin. We need to prove that $\lambda_\infty:=\lambda_\infty(P,\tilde{W},\R^n)=1$.

Clearly, $\lambda_\infty \geq 1$. On the other hand, since $\psi$  is a positive solution of minimal growth (i.e. a principal solution) near  $0$ and not near $\infty$ of the equation
\begin{equation}\label{eq_SL}
    -(t^{n-1}v')'+t^{n-1}V(t)=0 \qquad t\in (0,\infty),
\end{equation}
the oscillatory criterion \cite[Theorem~2.1]{GU} implies that the equation
\begin{equation}\label{eq_SLq}
    -(t^{n-1}v')'+t^{n-1}(V(t)+q(t))v=0 \qquad t\in (0,\infty),
\end{equation}
is oscillatory near infinity if
\begin{multline}\label{GUcriterion}
    \limsup_{t\to\infty} \Big[q(t)t^{2n-2}(\psi(t))^4 \left(\int_t^\infty r^{1-n} (\psi(r))^{-2}\dr\right)^2\Big]\\
    =\limsup_{t\to\infty} \frac{q(t)}{4W(t)}< -\frac{1}{4}\,,
\end{multline}
where we have used \eqref{eqW2}.

In particular, \eqref{GUcriterion} implies that for any $\vge>0$ the equation
$$\Big(-\Delta+V-(1+\vge)\tilde{W}\Big)u=0$$ does not admit a positive {\em radial} solution near infinity. Consequently, the equation $\Big(-\Delta+V-(1+\vge)\tilde{W}\Big)u=0$ does not admit {\em any} positive solution near infinity. Hence  $\lambda_\infty(-\Delta+V,\tilde{W},\mathbb{R}^n) = 1$.
\end{proof}

\section{Some more results concerning several ends}\label{appen2}
%%%%%%%%%%%%%%%%%%
 Here we study a particular case of Theorem \ref{lem_sever} without using the Liouville comparison argument, an argument that applies only in the symmetric case.  In particular, we recover, and in fact improve, the results of \cite{CZ}.

Consider a subcritical operator $P$ in $\Gw$, and let $x_1,\ldots,x_N$ be given distinct points in $\Omega$. Set $\Omega^\star:=\Omega\setminus\{x_1,\ldots,x_N\}$. Let $u_i=G(\cdot,x_i)$ be the Green function with a pole at $x_i$, and $u_0$ be a positive solution of $Pu=0$ in $\Gw$ such that
$$\lim_{x\to \infty} \frac{G(x,0)}{u_0(x)}=0.$$

We consider the supersolution construction $v:=\prod_{j=0}^N u_j^{\alpha_j}$   with the functions $u_0,\ldots,u_N$, and  with weights $\alpha_i\in(0,\frac{1}{2}]$, where $i=0,\ldots,N$, and $\sum_{i=0}^N\alpha_i=1$. We claim that in a certain number of cases, this construction gives a critical weight $W$.

First, consider the Laplacian on $\R^n$, $n\geq 3$.  Let us compute explicitly the Hardy-weight $W$ for the case $\alpha_i=\frac{1}{N+1}$, where $i=0,\ldots,N$. If $i\neq 0$, $j\neq 0$, then $\frac{u_i}{u_j}=\left|\frac{x-x_i}{x-x_j}\right|^{2-n}$, and therefore,
$$W(x)=\left(\frac{n-2}{N+1}\right)^2\sum_{i=1}^N \frac{1}{|x-x_i|^2}+\frac{1}{(N+1)^2}\sum_{1\leq i<j\leq N}\left|\nabla\log\left(\frac{u_i(x)}{u_j(x)}\right)\right|^2.$$
Now,
\begin{multline*}
\left|\nabla\log\left(\frac{u_i(x)}{u_j(x)}\right)\right|^2=
\left(\frac{n-2}{2}\right)^2\left|2\frac{x-x_i}{|x-x_j|^2}-2\frac{x-x_j}{|x-x_j|^2}\right|^2\\[2mm]
=\left(n-2\right)^2\frac{1}{|x-x_i|^4|x-x_j|^4}\left| |x-x_j|^2(x-x_i)-|x-x_i|^2(x-x_j)\right|^2\\[2mm]
=\left(n-2\right)^2\frac{1}{|x-x_i|^4|x-x_j|^4}\big(|x-x_j|^4|x-x_i|^2\\[2mm]
+|x-x_i|^4|x-x_j|^2-2|x-x_i|^2|x-x_j|^2\langle x-x_i,x-x_j\rangle\big)\\[2mm]
=\left(n-2\right)^2\frac{|x_i-x_j|^2}{|x-x_i|^2|x-x_j|^2}\,.
\end{multline*}
Hence,
\begin{equation}\label{eq_weight_9}
W(x)=\left(\frac{n-2}{N+1}\right)^2\left(\sum_{i=1}^N \frac{1}{|x-x_i|^2}+\sum_{1\leq i<j\leq N}\frac{|x_i-x_j|^2}{|x-x_i|^2|x-x_j|^2}\right).
\end{equation}

We claim that the Hardy-weight $W$ given by \eqref{eq_weight_9} is a \textit{critical} weight, but if $N>1$, the constant is neither optimal at any of the $x_i$, nor at infinity.
\begin{Pro}\label{Pro_appB}
Let $P=-\Gd$ in $\R^n$, $n\geq 3$, and $\{x_i\}_{i=1}^N$ be distinct points in $\R^n$.  Consider the preceding supersolution construction $v:=\prod_{j=0}^N u_j^{\alpha_j}$  with $u_0=\mathbf{1}$, $u_i=G_{-\Gd}^{\R^n}(\cdot,x_i)$, $i=1,\ldots, N$, the weights $\alpha_i\in (0,\frac{1}{2}]$, $i=0,\ldots,N$, such that $\sum_{i=0}^N\alpha_i=1$, and the Hardy-weight \eqref{eq_weight_9}.

\vskip 3mm

Then $-\Gd-W$ is a critical in $\Omega^\star:=\Omega\setminus\{x_1,\ldots,x_N\}$.
\end{Pro}
\begin{proof}
We may assume that $N>1$. The first part of the proof is general and apply to a general (not necessarily symmetric) subcritical operator $P$ in $\Omega$.

We know that $v$ is a positive solution of $(P-W)v=0$, and therefore, it is enough to prove that $v$ has minimal growth at infinity in $\Omega^\star$, that is $v$ has minimal growth at infinity in $\Omega$ and at each of the points $x_i$. By convention, we will set $x_0=\infty$.

Fix $i\in\{0,\ldots,N\}$. Denote $\alpha:=\alpha_i$. We want to show that $v$ has minimal growth at $x_i$. Denote $\hat{u}_i:=\left(\prod_{j\neq i}u_j^{\alpha_j}\right)^{\frac{1}{1-\alpha}},$ then $\hat{u}_i$ is a positive solution of $(P-V)u=0$, where

$$V:=\frac{1}{(1-\alpha)^2}\left(\sum_{\substack{k<l\\k,\l\neq i}}\alpha_k\alpha_l\left|\nabla\log\left(\frac{u_k(x)}{u_l(x)}\right)\right|^2\right).$$

Now if we apply the supersolution construction to $\hat{u}_i$ (that is a solution of $(P-V)u=0$),  and $u_i$ (a positive solution of $Pu=0$), with a weight $\beta\in(0,1)$, we obtain (see Remark~\ref{rem_conv_comb})
$$\left[P-(1-\beta) V-\beta(1-\beta)\tilde{W}\right]u_i^{\beta}\hat{u}_i^{1-\beta}=0.$$
For $\beta=\alpha$, we obtain

$$\left[P-(1-\alpha)V-\alpha(1-\alpha)\tilde{W}\right]v=0,$$
whence $W=(1-\alpha)V-\alpha(1-\alpha)\tilde{W}$. Similarly, for $\beta=1-\alpha$, we get

$$\left[P- \alpha V-\alpha(1-\alpha)\tilde{W}\right]u_i^{1-\alpha}\hat{u}_i^{\alpha}=0.$$
Write $L:=P-W$, and $w:=u_i^{1-\alpha}\hat{u}_i^{\alpha}$, then we have

$$Lv=0, \qquad \hbox{and } \left[L+(1-2\alpha)V\right]w=0,$$
and we want to deduce from it that $v$ has minimal growth at infinity in $\Omega^\star$. Notice that
$$\lim_{x\to x_i}\frac{v}{w}=\lim_{x\to x_i}\left(\frac{\hat{u}_i}{u_i}\right)^{1-2\alpha}=0.$$
If $w$ would be  a (super)positive solution of $Lu=(P-W)u=0$ near $x_i$, then Proposition \ref{minimal} will imply that $v$ has minimal growth at $x_i$. However, instead $w$ is only a subsolution of $L$ (notice that by hypothesis, $1-2\alpha\geq0$). We will show that in a neighborhood of $x_i$, we can find a positive solution $\tilde{w}$ of $L\tilde{w}=0$, such that $\tilde{w}(x)\geq w(x)$. Hence we will have

$$\lim_{x\to x_i}\frac{v}{\tilde{w}}=0,$$
and this will imply that $v$ has minimal growth at $x_i$.

\vskip 3mm

 Define
\begin{equation}\label{eq_GLU}
h(x):=(1-2\alpha)\int_UG_L^U(x,y)V(y)w(y)\dy,
\end{equation}
where $G_L^U$ is the minimal positive Green function of $L$ in a relatively compact neighborhood $U\subset \Omega\setminus\{x_0,\ldots,\hat{x}_i,\ldots,x_N\}$ of $x_i$ in $\Omega^\star$  (hence, a sequence in $U$ which goes to infinity in $\Omega^\star$ necessarily goes to $x_i$).

Let $\tilde{w}:=w+h$. Formally, it is obvious that $\tilde{w}$ is a positive solution of $Lu=0$ in a neighborhood of $x_i$, and that $\tilde{w}\geq w$ since $h\geq0$. So, it remains to show is that indeed that the integral in \eqref{eq_GLU} is finite.

Assume now that $P$ is symmetric. Since the singularity of $G_L^U(x,\cdot)$ is locally integrable, we only need to show is that $G_L^U(x,\cdot)Vw$ is in $L^1$ around $x_i$. Since $G_L^U(x,\cdot)$ is a positive solution of $Lu=0$ of minimal growth at $x_i$, $G_L^U$ is symmetric (i.e. $G_L^U(x,y)=G_L^U(y,x)$), and $v$ is a positive solution of $Lu=0$, we necessarily have

$$G_L^U(x,\cdot)\leq C(x)v$$
in a neighborhood of $x_i$.  Consequently,
\begin{equation}\label{GLUw}
  G_L^U(x,\cdot)w\leq C(x)u_i\hat{u}_i,
\end{equation}
in a neighborhood of $x_i$. We distinguish two cases:

\vskip 3mm

$\bullet$ First, assume that $i\neq 0$. Then $V(y)\sim 1$ and $\hat{u}_i(y)\sim 1$ when $y\to x_i$. Thus,
$$G_L^U(x,y)V(y)w(y)\leq C(x)u_i(y)\sim C(x)|y-x_i|^{2-n},$$
which is $L^1$ at $x_i$.

\vskip 3mm

$\bullet$ Assume now that $P=-\Gd$ and $\Gw=\R^n$. We treat the case $i=0$. Then for fixed $x$,  and $y$ near infinity we find that
$$G_L^U(x,y)V(y)w(y)\leq C(x) G_{-\Gd}^{\R^n}(x,y)u_0(y)V(y).$$ We already know that for an optimal Hardy-weight $\bar{W}$ with respect to the pair $(\mathbf{1},G)$, the function  $Gu_0\bar{W}$ is not integrable at infinity, but in many cases $Gu_0V$ is integrable at infinity. In particular, in the case of $P=-\Gd$ in $\R^n$, if $u_0=\mathbf{1}$, then at infinity $V\sim r^{-4}$, hence $G_{-\Gd}^{\R^n}u_0V\sim r^{-n-2}$, which is indeed integrable at infinity.
\end{proof}
\begin{remark}\label{rem_Caz_Zua}{\em
Recently, Cazacu and Zuazua \cite[Theorem~3.1]{CZ} used the supersolution construction with uniform weights $\alpha_i=\frac{1}{N}$,  $i=1,\ldots,N$, and the positive solutions $u_1,\cdots,u_N$, where $u_i=G_{-\Gd}^{\R^n}(\cdot,x_i)$, $i=1,\ldots, N$ (i.e. discarding $u_0$), and obtained the Hardy inequality $-\Delta-W_2(x)\geq 0$ in
$\Omega^\star=\Omega\setminus\{x_1,\ldots,x_N\}$  with the multipolar Hardy-weight
\begin{equation}\label{eq_CZ}
W_2(x):=\left(\frac{n-2}{N}\right)^2\left(\sum_{1\leq i<j\leq N}\frac{|x_i-x_j|^2}{|x-x_i|^2|x-x_j|^2}\right)\,.
\end{equation}
We note that the minimizing sequence used in \cite{CZ} for the proof of optimality of the constant $(\frac{n-2}{N})^2$ is clearly a null sequence, and therefore $-\Delta-W_2(x)$ is in fact {\em critical} in $\Omega^\star$. The criticality of $-\Delta-W_2(x)$ can be also proved using Lemma~\ref{min_growth_ends}.

Moreover, Lemma~\ref{min_growth_ends} can be applied also to the case of nonuniform weights. So, let $\ga=(\ga_0,\ldots,\ga_N)$ be a multi-index such that  $0<\alpha_i \leq 1/2$, $\sum_{i=1}^N\alpha_i=1$, and let $W_{\ga}$ be the Hardy potential obtained by the supersolution construction with respect to the Laplacian and the above positive solutions $u_i$. Then $-\Delta-W_{\ga}(x)$ is critical in $\Omega^\star=\Omega\setminus\{x_1,\ldots,x_N\}$.
 }
 \end{remark}

\begin{remark}\label{rem_Bosi_Zu}{\em
For each $j=1,\ldots,N$, the Hardy-weight  $W$ in \eqref{eq_weight_9} satisfies
$$ \lim_{x\to x_j}W(x)|x-x_j|^{2}=  C(n,N), \; \mbox{ where }\;C(n,N):=\frac{4N}{(N+1)^2}C_H,$$ (so, $C(n,N)< C_H$ if $N> 1$, and $W$  is  not optimal near $x_j$),  and
$$ \lim_{x\to \infty}W(x)|x|^{2}= C(n,N)\leq C_H ,$$
near infinity (so, for $N>1$ it is not optimal near infinity).

We note that in \cite{BDE} the obtained multipolar Hardy-weight
\begin{equation}\label{eq_BDE}
W_1(x):=\frac{C_H}{N}\left(\sum_{i=1}^N \frac{1}{|x-x_i|^2}\right)+\frac{C_H}{N^2}\left(\sum_{1\leq i<j\leq N}\frac{|x_i-x_j|^2}{|x-x_i|^2|x-x_j|^2}\right)\,.
\end{equation}
satisfies
$$ \lim_{x\to x_j}W_1(x)|x-x_j|^{2}= C_1(n,N),\; \mbox{ where }\; C_1(n,N):=\frac{2N-1}{N^2}C_H\,$$ (so, $C_1(n,N)<C(n,N) <C_H$ if $N> 1$, and $W_1$  is not optimal near $x_j$), but like $C_H|x|^{-2}$ near infinity.

On the other hand, in \cite{CZ} the obtained Hardy-weight $W_2$ (see\eqref{eq_CZ}) behaves asymptotically near each singular point $x_j$, $j=1,\ldots,N$  like
$$ \lim_{x\to x_j}W_2(x)|x-x_j|^{2}=  C_2(n,N), \; \mbox{ where }\;C_2(n,N):=\frac{4N-4}{N^2}C_H,$$ (so, $C_1(n,N)\leq C(n,N)<C_2(n,N)$ if $N> 1$),  but like $|x|^{-4}$ near infinity (so, it is not optimal near infinity).
 }
 \end{remark}
For general domains and not necessarily symmetric operators we have.
\begin{proposition}\label{lem_severB}
 Consider a subcritical operator $P$ in $\Gw\subset \R^n$, and let $\{x_i\}_{i=1}^N$  be distinct points in $\Gw$. Set $\Omega^\star:=\Omega\setminus\{x_i\}_{i=1}^N$. Let $u_i:=G_P^\Gw(\cdot,x_i)$ be the Green function with a pole at $x_i$, and $u_0$ be a positive solution of $Pu=0$ in $\Gw$ such that
$$\lim_{x\to \infty} \frac{G(x,0)}{u_0(x)}=0.$$

More generally, consider a manifold $\Gw^\star$ with ends $\{x_i\}_{i=0}^N$. Assume that  for $0\leq i\leq N$  and $j\neq i$, the function  $u_i$ are positive solutions of  the equation $Pu=0$ in $\Gw^\star$ of minimal growth near each end $x_j$,  and
$$\lim_{x\to x_i} \frac{u_j(x)}{u_i(x)}=0  \qquad \forall j\neq i.$$
Consider the supersolution construction and the corresponding Hardy-weight given by
\begin{equation*}
v_\ga:=\prod_{j=0}^N u_j^{\alpha_j}\,\quad
W_\ga := \sum_{i<j}\ga_i\ga_j\left|\nabla\log\left(\frac{u_i}{u_j}\right)\right|_A^2,
\end{equation*}
 where  $\ga=(\ga_0,\ldots,\ga_N)$ is a multi-index such that  $0<\alpha_j < 1/2$, and $\sum_{j=0}^N\alpha_j=1$.
Assume further that near each $x_i$ we have
\begin{equation}\label{eq_asympi}
W_\ga  \sim\ga_i\sum_{\substack{0\leq j\leq N\\[1mm]j \neq i}}  \ga_j \left|\nabla\log\left(\frac{u_i}{u_j}\right)\right|_A^2.
\end{equation}

Then the corresponding Hardy-weight is critical.
\end{proposition}
\begin{proof}
For a fixed $i$ consider the one-parameter family
$$\beta_i(t) := \alpha_i + t, \quad \beta_j(t) := \alpha_j - \frac{\alpha_j}{1- \alpha_i} t \quad (j \neq i).$$
Then for every $t \in (0,\,1 - \alpha_i)$ we have $\sum \beta_j(t) = 1$ and
$$
(P - W_{\beta})v_\beta = 0.
$$

It can be easily checked that  for any $j\neq i$ the function $\beta_i(t) \beta_j(t)$ has maximum at the point $t_M = 1/2 -  \alpha_i$. In view of \eqref{eq_asympi}, the function
$v_{\beta(t_M)}$ is a positive supersolution of the equation $(P - W_\alpha)u=0$ near $x_i$. Furthermore, since $t_M > 0$, we have
$$
\lim_{x \to x_i} \frac{v_\alpha(t)}{v_{\beta(t_M)}(x)} = 0.
$$
Now notice that  Proposition~\ref{minimal} holds true even when $u_1$ is just a positive supersolution. Thus, $v_\alpha$ has minimal growth at the end $x_i$ and the lemma follows.
\end{proof}

%%%%%%%%%%%%%%%%%%%%%%%%%%%%%%%%%
\begin{center}{\bf Acknowledgments} \end{center}
The authors wish to thank K.~Tintarev who raised to our attention the main problem studied in this paper. Furthermore we thank M.~Marcus for pointing out his paper \cite{M65}, G.~Psaradakis for communicating results concerning Example~\ref{ex1a} to us, and for many valuable comments, and I.~Versano for providing (after the publication of the paper in JFA) a counterexample to Conjecture~\ref{13_8}.
The authors acknowledge the support of the Israel Science Foundation (grants No. 587/07 and 963/11) founded by the Israel Academy of
Sciences and Humanities. B.~D. was supported in part by a Technion fellowship.

\end{document}